\documentclass[11pt,reqno]{article}
\usepackage{amsmath}
\usepackage{mathrsfs}
\usepackage{amsfonts}
\usepackage{amssymb}

\usepackage{amssymb}
\usepackage{amsthm}
\usepackage{graphicx}              
\usepackage{amsmath}               
\usepackage{amsfonts}              
\usepackage{amsthm}                
\usepackage{setspace}
\usepackage{epstopdf}
\usepackage{epsfig}

\newtheorem{theorem}{Theorem}[section]
\newtheorem{proposition}{Proposition}[section]
\newtheorem{lemma}{Lemma}[section]
\newtheorem{corollary}{Corollary}[section]
\newtheorem{remark}{Remark}[section]
\newtheorem{definition}{Definition}[section]

\makeatletter

\newcommand{\Rmnum}[1]{\expandafter\@slowromancap\romannumeral #1@}
\makeatother

\allowdisplaybreaks
\numberwithin{equation}{section}

\title {Asymptotic stability of harmonic maps between 2D hyperbolic spaces under the wave map equation. II.  Small energy case}

\author{Ze Li \qquad Xiao Ma \qquad Lifeng Zhao }

\date{ }

\begin{document}

\maketitle
\medskip

\begin{abstract}
In this paper, we prove that the small energy harmonic maps from $\Bbb H^2$ to $\Bbb H^2$ are asymptotically stable under the wave map equation in the subcritical perturbation class.
This result may be seen as an example supporting the soliton resolution conjecture for geometric wave equations without equivariant assumptions
on the initial data. In this paper, we construct Tao's caloric gauge in the case when nontrivial harmonic map occurs. With the ``dynamic separation" the master equation of the heat tension field appears as a semilinear magnetic wave equation. By the endpoint and weighted Strichartz estimates for magnetic wave equations obtained by the first author \cite{Lize1}, the asymptotic stability follows by a bootstrap argument.
\end{abstract}

\maketitle
\tableofcontents

\section{Introduction}
Let $(M,h)$ and $(N,g)$ be two Riemannian manifolds without boundary. A wave map is a map from the Lorentz manifold $\Bbb R \times M$ into $N$,
$$u:\Bbb R\times M\to N,
$$
which is locally a critical point for the functional
\begin{align}\label{w1}
F(u) = \int_{\Bbb R \times M} {\left(- {{{\left\langle {{\partial _t}u,{\partial _t}u} \right\rangle }_{{u^*}g}} + {h^{ij}}{{\left\langle {{\partial _{{x_i}}}u,{\partial _{{x_j}}}u} \right\rangle }_{{u^*}g}}} \right)} {\rm{dtdvo{l_h}}}.
\end{align}
Here ${h_{ij}}dx^idx^j$ is the metric tension under a local coordinate $(x^1,...,x^m)$ for $M$.
In a coordinate free expression, the integrand in the functional $F(u)$ is the energy density of $u$ under the Lorentz metric of $\Bbb R\times M$, $$\mathbf{\eta}=-dt\otimes dt+h_{ij}dx^i\otimes dx^j.$$
Given a local coordinate $(y^1,...,y^n)$ for $N$, the Euler-Lagrange equation for (\ref{w1}) is given by
\begin{align}\label{wmap1}
\Box u^k +{\eta ^{\alpha \beta }}\overline{\Gamma}_{ij}^k(u){\partial _\alpha }{u^i}{\partial _\beta }{u^j}= 0,
\end{align}
where $\Box=-\partial_t^2+\Delta_M$ is the D'Alembertian on $\Bbb R\times M$, $\overline{\Gamma}^k_{ij}(u)$ are the Christoffel symbols at the point $u(t,x)\in N$. In this paper, we consider the case $M=\Bbb H^2$, $N=\Bbb H^2$.

The wave map equation on a flat spacetime, which is sometimes known as the nonlinear $\sigma$-model, arises as a model problem in general relativity and particle physics, see for instance \cite{MS}. The wave map equation on curved spacetime is related to the wave map-Einstein system and the Kerr Ernst potential, see \cite{AGS,IK,GL}. We remark that the case where the background manifold is the hyperbolic space is of particular interest. Indeed, the anti-de Sitter space (AdSn), which is the exact solution of Einstein's field equation for an empty universe with a negative cosmological constant, is asymptotically hyperbolic.

There exist plenty of works on the Cauchy problem, the long dynamics and blow up for wave maps on $\Bbb R^{1+m}$. We first recall the non-exhaustive lists of results on equivariant maps. The critical well-posedness theory was initially considered by Christodoulou, Tahvildar-Zadeh \cite{CT} for radial wave maps and Shatah, Tahvildar-Zadeh \cite{STZ2} for equivariant wave maps. The global well-posedness result of \cite{CT} was recently improved to scattering by Chiodaroli, Krieger, Luhrmann \cite{CKL}.  The bubbling theorem of wave maps was proved by Struwe \cite{S}. The explicit construction of blow up solutions behaving as a perturbation of the rescaling harmonic map was achieved by Krieger, Schlag, Tataru \cite{KST}, Raphael, Rodnianski \cite{RR}, and Rodnianski, Sterbenz \cite{RS} for the $\Bbb S^2$ target in the equivariant class. And the ill-posedness theory was studied in D'Ancona, Georgiev \cite{DG} and Tao \cite{Tao10}.

Without equivariant assumptions on the initial data the sharp subcritical well-posedness theory was developed by Klainerman, Machedon \cite{KM1,KM2} and Klainerman, Selberg \cite{KS2}. The small data critical case was started by Tataru \cite{Tataru2} in the critical Besov space, and then completed by Tao \cite{Tao1,Tao2} for wave maps from $\Bbb R^{1+d}$ to $\Bbb S^m$ in the critical Sobolev space. The small data theory in critical Sobolev space for general targets was considered by Krieger \cite{J1,J2}, Klainerman, Rodnianski \cite{KR3}, Shatah, Struwe \cite{SS}, Nahmod, Stefanov, Uhlenbeck \cite{NSU}, and Tataru \cite{Tataru3}.

The dynamic behavior for wave maps on $\Bbb R^{1+2}$ with general data was obtained by Krieger, Schlag \cite{KS} for the $\Bbb H^2$ targets, Sterbenz, Tataru \cite{ST1,ST2} for compact Riemann manifolds and initial data below the threshold, and Tao \cite{Tao7} for the $\Bbb H^n$ targets. In fact, Sterbenz, Tataru \cite{ST1,ST2} proved that for any initial data with energy less than that of the minimal energy nontrivial harmonic map evolves to a global and scattering solution.

The works on the wave map equations on curved spacetime were relatively less. The existence and orbital stability of equivariant time periodic wave maps from $\Bbb R\times \Bbb S^2$ to $\Bbb S^2$ were proved by Shatah, Tahvildar-Zadeh \cite{STZ1}, see Shahshahani \cite{S} for an generalization of $\Bbb S^2$. The critical small data Cauchy problem for wave maps on small asymptotically flat perturbations of $\Bbb R^4$ to compact Riemann manifolds was studied by Lawrie \cite{LA}. The soliton resolution and asymptotic stability of harmonic maps under wave maps on $\Bbb H^2$ to $\Bbb S^2$ or $\Bbb H^2$ in the 1-equivariant case were established by Lawrie, Oh, Shahshahani \cite{LOS1,LOS2,LOS4,LOS5}, see also \cite{LOS} for critical global well-posedness for wave maps from $\Bbb R\times \Bbb H^d$ to compact Riemann manifolds with $d\ge4$.

In this paper, we study the asymptotic stability of harmonic maps to (\ref{w1}). The motivation is the so called soliton resolution conjecture in dispersive PDEs which claims that every global bounded solution splits into the superposition of divergent solitons with a radiation part plus an asymptotically vanishing remainder term as $t\to\infty$. The version for wave maps and hyperbolic Yang-Mills has been verified by Cote \cite{C} and Jia, Kenig \cite{JK} for equivariant maps along a time sequence, see also \cite{KLLS1,KLLS2} for exotic-ball wave maps and \cite{Gy} for wormholes. Recently Duyckaerts, Jia, Kenig, Merle \cite{DJKM} obtained the universal blow up profile for type II blow up solutions to wave maps $u:\Bbb R\times\Bbb R^2\to \Bbb S^2$ with initial data of energy slightly above the ground state. For wave maps from $\Bbb R\times \Bbb H^2$ to $\Bbb H^2$, Lawrie, Oh, Shahshahani \cite{LOS,LOS4} raised the following soliton resolution conjecture,\\
{\bf Conjecture 1.1}
 Consider the Cauchy problem for wave map $u:\Bbb R\times \Bbb H^2\to \Bbb H^2$ with finite energy initial data $(u_0,u_1)$. Suppose that outside some compact subset $\mathcal{K}$ of $\Bbb H^2$ for some harmonic map $Q:\Bbb H^2\to \Bbb H^2$ we have
$$
u_0(x)=Q(x), \mbox{  }{\rm{for}}\mbox{  }x\in \Bbb H^2\backslash\mathcal{K}.
$$
Then the unique solution $(u(t),\partial_tu(t))$ to the wave map scatters to $(Q(x),0)$  as $t\to\infty$.

In this paper, we consider the easiest case of Conjecture 1.1, i.e., when the initial data is a small perturbation of harmonic maps with small energy.  In order to state  our main result, we introduce the notion of admissible harmonic maps.
\begin{definition}\label{2as}
Let $D=\{z:|z|<1\}$ with the hyperbolic metric be the Poincare disk.
We say the harmonic map $Q:D\to D$ is admissible if $Q(D)$ is a compact subset of $D$ covered by a geodesic ball centered at 0 of radius $R_0$, $\|\nabla^kdQ\|_{L^2}<\infty$ for $k=0,1,2$,
and there exists some $\varrho>0$ such that $e^{\varrho r}|dQ|\in L^{\infty}$, where $r$ is the distance between $x\in D$ and the origin point in $D$.
\end{definition}
For any given admissible harmonic map $Q$, we define the space $\bf{H}^k\times\bf{H}^{k-1}$ by (\ref{h897}).
Our main theorem is as follows.
\begin{theorem}\label{a1}
Fix any $R_0>0$. Assume the given admissible harmonic map $Q$ in Definition \ref{2as} satisfies
\begin{align}\label{as4}
\|dQ\|_{L^2_x}<\mu_1, \mbox{  }\|e^{\varrho r}|dQ|\|_{L^{\infty}_x}<\mu_1, \mbox{  }\|\nabla^2dQ\|_{L^{\infty}_x}+\|\nabla dQ\|_{L^{\infty}_x}<\mu_1.
\end{align}
And assume that the initial data $(u_0,u_1)\in {\bf{H}^3\times\bf{H}^2}$ to (\ref{wmap1}) with $u_0:\Bbb H^2\to\Bbb H^2$, $u_1(x)\in T_{u_0(x)}N$ for each $x\in \Bbb H^2$ satisfy
\begin{align}\label{as3}
\|(u_0,u_1)-(Q,0)\|_{{\bf{H}}^2\times {\bf{H}^1}}<\mu_2.
\end{align}
Then if $\mu_1>0$ and $\mu_2>0$ are sufficiently small depending only on $R_0$,
(\ref{wmap1}) has a global solution $(u(t),\partial_tu(t))$ which converges to the harmonic map $Q:\Bbb H^2\to \Bbb H^2$ as $t\to\infty$, i.e.,
$$
\mathop {\lim }\limits_{t\to\infty }\mathop {\sup }\limits_{x \in {\mathbb{H}^2}} {d_{{\mathbb{H}^2}}}\left( {u(t,x),Q(x)} \right) = 0.
$$
\end{theorem}

The initial data considered in this paper are perturbations of harmonic maps in the $\bf H^2$ norm. If one considers perturbations in the energy critical norm $H^1$, the $S_k$ v.s. $N_k$ norm constructed by Tataru \cite{Tataru2} and Tao \cite{Tao2} should be built for the hyperbolic background.

\noindent{\bf{Remark 1.1}} Notice that the limit harmonic map coincides with the unperturbed harmonic map in Theorem 1.1. The reason for this coincidence is that the $\bf H^2\times\bf H^1$ norm assume the initial data coincide with $Q$ at the infinity, then the uniqueness of harmonic maps with prescribed boundary map shows the limit harmonic map is exactly the unperturbed one.

\noindent{\bf{Remark 1.2}}{\bf(Examples for the admissible harmonic maps)}\\
\noindent Denote $D=\{z:|z|<1\}$ to be the Poincare disk. Then any holomorphic map $f:D \to D$ is a harmonic map. If we assume that $f(z)$ can be analytically extended into a larger disk than the unit disk, then $\mu_1f:D\to D$ satisfies all the conditions in Definition 1.1 and Theorem 1.1 if $0<\mu_1\ll1$.
Hence the harmonic maps involved in Theorem 1.1 are relatively rich.
See [Appendix,\cite{LZ}] for the proof of these facts. It is important to see in theses examples that the dependence of $\mu_1$ on $R_0$ is neglectable.

\noindent{\bf{Remark 1.3}}(Examples for the perturbations of admissible harmonic maps)
 Since we have global coordinates for $\Bbb H^2$ given by (\ref{vg}), the perturbation in the sense of (\ref{as3}) is nothing but perturbations of $\Bbb R^2$-valued functions.

Since we are dealing with non-equivariant data where the linearization method seems to be hard to apply, we use the caloric gauge technique introduced by Tao \cite{Tao4} to prove Theorem 1.1.
The caloric gauge of Tao was applied to solve the global regularity of wave maps from $\Bbb R^{2+1}$ to $\Bbb H^n$ in the heat-wave project. We briefly recall the main idea of the caloric gauge. Given a solution to the wave map $u(t,x):\Bbb R^{1+2}\to \Bbb H^n$, suppose that $\widetilde u(s,t,x)$ solves the heat flow equation with initial data $u(t,x)$
$$
\left\{ \begin{array}{l}
\partial_s \widetilde{u}(s,t,x)=\sum^2_{i=1}\nabla_i\partial_i\widetilde{u}\\
\widetilde{u}(s,t,x) \upharpoonright_{s=0}= u(t,x). \\
\end{array} \right.
$$
Since there exists no nontrivial finite energy harmonic map from $\Bbb R^2$ to $\Bbb H^n$,  one can expect that the corresponding heat flow $\widetilde{u}(s,t,x)$ converges to a fixed point $Q$ as $s\to\infty$. For any given orthonormal frame at the point $Q$, one can pullback the orthonormal frame parallel with respect to $s$ along the heat flow to obtain the frame at $\widetilde{u}(s,t,x)$, particularly $u(t,x)$ when $s=0$. Then rewriting (\ref{wmap1}) under the constructed frame will give us a scalar system for the differential fields and connection coefficients. Despite the fact that the caloric gauge can be viewed as a nonlinear Littlewood-Paley decomposition, the essential advantage of the caloric gauge is that it removes some troublesome frequency interactions, which is of fundamental importance for critical problems in low dimensions.

Generally the caloric gauge was used in the case where no harmonic map occurs, for instance energy critical geometric wave equations with energy below the threshold. In our case nontrivial harmonic exists no mater how small the data one considers. However, as observed in our work \cite{LZ}, the caloric gauge is still extraordinarily powerful. In fact, denoting the solution of the heat flow with initial data $u(0,x)$ by $U(s,x)$, it is known that $U(s,x)$ converges to some harmonic map $Q(x)$ as $s\to\infty$. And one can expect that the solution $u(t,x)$ of (\ref{wmap1}) also converges to the same harmonic map $Q(x)$ as $t\to\infty$. This heuristic idea combined with the caloric gauge reduces the convergence of solutions to (\ref{wmap1}) to proving the decay of the heat tension filed.

There are three main ingredients in our proof. The first is to guarantee that all the heat flows initiated from $u(t,x)$ for different $t$ converge to the same harmonic map. This enables us to construct the caloric gauge. The second is to derive the master equation for the heat tension field, which finally reduces to a linear wave equation with a small magnetic potential.  The third is to design a suitable closed bootstrap program. All these ingredients are used to overcome the difficulty that no integrability with respect to $t$ is available for the energy density because the harmonic maps prevent the energy from decaying to zero as $t\to\infty$.

The key for the first ingredient is using the decay of $\partial_tu$ along the heat flow. In order to construct the caloric gauge, one has to prove the heat flow initiated from  $u(t,x)$ converges to the same harmonic map independent of $t$. If one only considers $t$ as a smooth parameter, i.e., in the homotopy class, the corresponding limit harmonic map yielded by the heat flow initiated from $u(t,x)$ can be different when $t$ varies. Indeed, there exist a family of harmonic maps $\{Q_{\lambda}\}$ which depend smoothly with respect to $\lambda\in(0,1)$. Therefore the heat flow with initial data $Q_{\lambda}$ remains to be  $Q_{\lambda}$, which changes according to the variation of $\lambda$. This tells us the structure of (\ref{wmap1}) should be considered.  The essential observation is $\partial_t u$ decays fast along the heat flow as $s\to\infty$. By a monotonous property observed initially by Hartman \cite{Hh} and the decay estimates of the heat semigroup, we can prove the distance between the heat flows initiated from $u(t_1)$ and $u(t_2)$ goes to zero as $s\to\infty$. Therefore the limit harmonic map for the heat flow generated from $u(x,t)$ are all the same for different $t$. Similar idea works for the Landau-Lifshitz flow, see our paper \cite{LZ}. And we remark that this part can be adapted to energy critical wave maps form $\Bbb R\times\Bbb H^2$ to $\Bbb H^2$ since essentially we only use the $L^2_x$ norm of $\partial_tu$ in the arguments which is bounded by the energy.

Different from the usual papers on the asymptotic stability, we will not use the linearization arguments involving spectrum analysis of the linearized operator and modulation equations. But the master equation appears naturally as a semilinear wave equation with a small magnetic potential.  Indeed, the main equation we need to consider is the nonlinear wave equation for the heat tension filed. The point is that although the nonlinear part of this equation is not controllable, one can separate part of them to be a magnetic potential with a remainder likely to be controllable. This is why we need the Strichartz estimates for magnetic wave equations.

The second ingredient is to control the remained terms in the nonlinear part of the master equation after we separate the magnetic potential away. In fact, the terms involving one order derivatives of the heat tension filed can not be controlled only by Strichartz estimates, even if we are working in the subcritical regularity.
In this paper, the one order derivative terms are controlled by the weighted Strichartz estimates and the exotic Strichartz estimates owned only by hyperbolic backgrounds compared with the flat case. These estimates were obtained in the first author's work \cite{Lize1}.

The third ingredient is to close the bootstrap, by which the global spacetime norm bounds of the heat tension field follows.
The caloric gauge yields the gauged equation for the corresponding differential fields $\phi_{x,t}$, connection coefficients $A_{x,t}$ and the heat tension filed. It has been discovered in Tao \cite{Tao7} that the key field one needs to study is the heat tension field which satisfies a semilinear wave equation. And for the small data Cauchy problem of wave maps on $\Bbb R\times\Bbb H^4$, Lawrie, Oh, Shahshahani \cite{LOS} shows in order to close the bootstrap arguments it suffices to firstly proving a global spacetime bound for the heat tension filed $\phi_s$. In our case, since the energy will not decay, one has to get rid of the inhomogeneous terms which involve only the differential fields $\phi_x$ in the master equation. Furthermore, these troublesome terms involving only $\phi_x$ are much more serious in the study of the equation of wave map tension filed. This difficulty is overcome by using identities from intrinsic geometry to gain some cancelation and adding a space-time bound for $|\partial_tu|$ on the basis of the bootstrap arguments of \cite{LOS,Tao7}.

This paper is organized as follows. In Section 2, we recall some notations and notions and prove an equivalence between the intrinsic and extrinsic Sobolev norms in some sense.  In Section 3, we construct the caloric gauge and obtain the estimates of the connection coefficients. In Section 4, we we derive the master equation.  In Section 5, we first recall the non-endpoint and endpoint Strichartz estimates, Morawetz inequality, and weighted Strichartz estimates for the linear magnetic wave equation. Then we close the bootstrap and deduce the global spacetime bounds for the heat tension field. In Section 6, we finish the proof of Theorem 1.1. In Section 7, we prove some remaining claims in the previous sections.

We denote the constants by $C(M)$ and they can change from line to line. Small constants are usually denoted by $\delta$ and it may vary in different lemmas.
$A\lesssim B$ means there exists some constant $C$ such that $A\le CB$.

\section{Preliminaries}
Some standard preliminaries on the geometric notions of the hyperbolic spaces, Sobolev embedding inequalities and an equivalence relationship for the intrinsic and extrinsic formulations of the Sobolev spaces are recalled first. As a corollary we prove the local well-posedness for initial data $(u_0,u_1)$ in the $\bf{H}^3\times \bf{H}^2$ regularity and a conditional global well-posedness proposition. In addition, the smoothing effect of heat semigroup is recalled.

\subsection{The global coordinates and definitions of the function spaces}
The covariant derivative in $TN$ is denoted by $\widetilde{\nabla}$, the covariant derivative induced by $u$ in $u^*(TN)$ is denoted by $\nabla$. We denote the Riemann curvature tension of $N$ by $\mathbf{R}$. The components of Riemann metric are denoted by $h_{ij}$ for M and $g_{ij}$ for N respectively.  The Christoffel symbols on $M$ and $N$ are denoted by $\Gamma^{k}_{ij}$ and $\overline{\Gamma}^{k}_{ij}$ respectively.

We recall some facts on hyperbolic spaces. Let $\Bbb R^{1+2}$ be the Minkowski space with Minkowski metric $-(dx^0)^2+(dx^1)^2+(dx^2)^2$. Define a bilinear form on $\Bbb R^{1+2}\times \Bbb R^{1+2}$,
$$
[x,y]=x^0y^0-x^1y^1-x^2y^2.
$$
The hyperbolic space $\mathbb{H}^2$ is defined by
$$\mathbb{H}^2=\{x\in \Bbb R^{2+1}: [x,x]=1 \mbox{  }{\rm{and}}\mbox{  }x^0>0\},$$
with a Riemannian metric being the pullback of the Minkowski metric by the inclusion map $\iota:\mathbb{H}^2\to \Bbb R^{1+2}.$
By Iwasawa decomposition we have a global system of coordinates. Indeed, the diffeomorphism $\Psi:\Bbb R\times \Bbb R\to \mathbb{H}^2$ is given by
\begin{align}\label{vg}
\Psi(x_1,x_2)=({\rm{cosh}} x_2+e^{-x_2}|x_1|^2/2, {\rm{sinh}} x_2+e^{-x_2}|x_1|^2/2, e^{-x_2}x_1).
\end{align}
The Riemannian metric with respect to this coordinate system is given by
$$
e^{-2x_2}(dx_1)^2+(dx_2)^2.
$$
The corresponding Christoffel symbols are
\begin{align}\label{christ}
\Gamma^1_{2,2}=\Gamma^2_{2,1}=\Gamma^2_{2,2}=\Gamma^1_{1,1}=0; \mbox{ }\Gamma^1_{2,1}=-1, \mbox{  }\Gamma^2_{1,1}=e^{-2x_2}.
\end{align}
For any $(t,x)$ and $u:[0,T]\times\mathbb{H}^2\to \Bbb H^2$, we define an orthonormal frame at $u(t,x)$ by
\begin{align}\label{frame}
\Theta_1(u(t,x))=e^{u^2(t,x)}\frac{\partial}{\partial y_1}; \mbox{  }\Theta_2(u(t,x))=\frac{\partial}{\partial y_2}.
\end{align}
where $(u^1,u^2)$ denotes the coordinate of $u$ given by (\ref{vg}).
{\bf Throughout this paper we will use coordinates (\ref{vg}) for both the target manifold $N=\Bbb H^2$ and the starting manifold $M=\Bbb H^2$.}
Recall also the identity for Riemannian curvature on $N=\Bbb H^2$
\begin{align}\label{2.best}
{\bf R}(X,Y)Z={\widetilde{\nabla} _X}{\widetilde{\nabla} _Y}Z - {\widetilde{\nabla }_Y}{\widetilde{\nabla} _X}Z - {\widetilde{\nabla}_{[X,Y]}}Z  = \left\langle {X,Z} \right\rangle Y - \left\langle {Y,Z} \right\rangle X.
\end{align}
We have a useful identity for $X,Y,Z\in u^*(TN)$
\begin{align}\label{2.4best}
{\nabla _i }\left( {{\bf R}\left( {X,Y} \right)Z} \right) = {\bf R}\left( {X,{\nabla _i }Y} \right)Z + {\bf R}\left( {{\nabla _i }X,Y} \right)Z + {\bf R}\left( {X,Y} \right){\nabla _i }Z.
\end{align}
For simplicity, denote $(X\wedge Y)Z=\left\langle {X,Z} \right\rangle Y - \left\langle {Y,Z} \right\rangle X$.

Let $H^k(\mathbb{H}^2;\Bbb R)$ be the usual Sobolev space for scalar functions defined on manifolds.  We also recall the norm of $H^k$:
$$
\|f\|^2_{H^k}=\sum^k_{l=1}\|\nabla^l f\|^2_{L^2_x},
$$
where $\nabla^l f$ is the covariant derivative.
For maps $u:\mathbb{H}^2\to \mathbb{H}^2$, we define the intrinsic Sobolev semi-norm $\mathfrak{H}^k$ by
$$
\|u\|^2_{{\mathfrak{H}}^k}=\sum^k_{i=1}\int_{\mathbb{H}^2} |\nabla^{i-1} du|^2 {\rm{dvol_h}}.
$$
The map $u:\mathbb{H}^2\to\mathbb{H}^2$ is associated with a vector-valued function $u:\mathbb{H}^2\to \Bbb R^2$ by (\ref{vg}). Indeed, the vector $(u^1(x),u^2(x))$ is defined by $\Psi(u^1(x),u^2(x))=u(x)$ for any $x\in \mathbb{H}^2$ . Let $Q:\Bbb H^2\to\Bbb H^2$ be an admissible harmonic map in Definition 1.1. Then the extrinsic Sobolev space is defined by
\begin{align}\label{h897}
{\bf{H}^k_{Q}}=\{u: u^1-Q^1(x), u^2-Q^2(x)\in H^k(\mathbb{H}^2;\Bbb R)\},
\end{align}
where $(Q^1(x),Q^2(x))\in \Bbb R^2$ is the corresponding components of $Q(x)$ under the coordinate (\ref{vg}). Denote the set of smooth maps which coincide with $Q$ outside of some compact subset of $M=\Bbb H^2$ by $\mathcal{D}$. Let $\mathcal{H}^k_{Q}$ be the completion of $\mathcal{D}$ under the metric given by
\begin{align}\label{h897}
{\rm{dist}}_{k,Q}(u,w)=\sum^2_{j=1}\|u^j-w^j\|_{H^k(\mathbb{H}^2;\Bbb R)},
\end{align}
where $u,w\in \mathcal{H}^k_{Q}$. Since $C^{\infty}_c(\mathbb{H}^2;\Bbb R)$ is dense in $H^k(\mathbb{H}^2;\Bbb R)$ (see Hebey \cite{Hebey}), $\mathcal{H}^k_{Q}$ coincides with ${\bf{H}^k_{Q}}$. And for simplicity, we write $\bf{H}^k$ without confusions.
If $u$ is a map from $\Bbb R\times\Bbb H^2$ to $\Bbb H^2$, we define the space $\bf{H}^k\times\bf{H}^{k-1}$ by
\begin{align}\label{h897}
{\bf{H}^k\times\bf{H}^{k-1}}=\left\{u:\sum^2_{j=1}\|u^j-Q^j\|_{H^k(\mathbb{H}^2;\Bbb R)}+\|\partial_tu^j\|_{H^{k-1}(\mathbb{H}^2;\Bbb R)}<\infty\right\}.
\end{align}
The distance in ${\bf{H}^k\times\bf{H}^{k-1}}$ is given by
\begin{align}\label{zvh897}
{\rm{dist}}_{\bf{H}^k\times\bf{H}^{k-1}}(u,w)=\sum^2_{j=1}\|u^j-w^j(x)\|_{H^k}+\|\partial_tu^j-\partial_t w^j\|_{H^{k-1}}.
\end{align}

\subsection{Sobolev embedding and Equivalence lemma}
The Fourier transform on hyperbolic spaces takes proper functions defined on $\mathbb{H}^2$ to functions defined on $\Bbb R\times \Bbb S^1$, see Helgason \cite{Hel} for details.
The operator $(-\Delta)^{\frac{s}{2}}$ is defined by the Fourier multiplier $\lambda\to (\frac{1}{4}+\lambda^2)^{\frac{s}{2}}$.
We now recall the Sobolev inequalities of functions in $H^k$.

\begin{lemma}\label{wusijue}
If $f\in C^{\infty}_c(\mathbb{H}^2;\Bbb R)$, then for $1<p<\infty,$ $p\le q\le \infty$, $0<\theta<1$, $1<r<2$, $r\le l<\infty$, $\alpha>1$, the following inequalities hold
\begin{align}
{\left\| f \right\|_{{L^2}}} &\lesssim {\left\| {\nabla f} \right\|_{{L^2}}} \label{{uv111}}\\
 {\left\| f \right\|_{{L^q}}} &\lesssim \left\| {\nabla f} \right\|_{{L^2}}^\theta \left\| f \right\|_{{L^p}}^{1 - \theta }\mbox{  }{\rm{when}}\mbox{  }\frac{1}{p} - \frac{\theta }{2} = \frac{1}{q} \label{uv211}\\
 {\left\| f \right\|_{{L^l}}} &\lesssim {\left\| {\nabla f} \right\|_{{L^r}}}\mbox{  }\mbox{  }\mbox{  }\mbox{  }\mbox{  }\mbox{  }\mbox{  }\mbox{  }\mbox{  }\mbox{  }{\rm{when}}\mbox{  }\frac{1}{r} - \frac{1}{2} = \frac{1}{l} \label{uv311}\\
 {\left\| f \right\|_{{L^\infty }}} &\lesssim {\left\| {{{\left( { - \Delta } \right)}^{\frac{\alpha }{2}}}f} \right\|_{{L^2}}}\mbox{  }\mbox{  }\mbox{  }{\rm{when}}\mbox{  }\alpha>1 \label{uv4}\\
 {\left\| {\nabla f} \right\|_{{L^p}}} &\sim{\left\| {{{\left( { - \Delta } \right)}^{\frac{1}{2}}}f} \right\|_{{L^p}}} \label{uv5}.
 \end{align}
\end{lemma}
For the proof, we refer to Bray \cite{B} for (\ref{uv211}), Ionescu, Pausader, Staffilani \cite{IPS} for (\ref{uv311}), Hebey \cite{Hebey} for (\ref{uv4}), see also Lawrie, Oh, Shahshahani \cite{LOS}. (\ref{uv5}) is obtained in \cite{Stri}.

We also recall the diamagnetic inequality which sometimes refers to Kato's inequality (see \cite{LOS}) and a more generalized Sobolev inequality (see [Proposition 2.2,\cite{AP}]).
\begin{lemma}
$(a)$ If $T$ is a tension field defined on $\Bbb H^2$, then in the distribution sense, one has the diamagnetic inequality
\begin{align}\label{wusijue3}
|\nabla|T||\le |\nabla T|.
\end{align}
$(b)$ Let $1<p,q<\infty$ and $\sigma_1,\sigma_2\in\Bbb R$  such that $\sigma_1-\sigma_2\ge n/p-n/q\ge0$. Then for all $f\in C^{\infty}_c(\Bbb H^n;\Bbb R)$
\begin{align*}
\|(-\Delta)^{\sigma_2}f\|_{L^q}\lesssim \|(-\Delta)^{\sigma_1}f\|_{L^p}.
\end{align*}
\end{lemma}

\begin{remark}
Lemma \ref{wusijue} and (\ref{wusijue3}) have several useful corollaries, for instance for $f\in H^2$
\begin{align}
\|f\|_{L_x^{\infty}}&\lesssim \|\nabla^2f\|_{L^2_x}\label{{uv111}6}\\
\|f\|_{L_x^{2}}&\lesssim \|\nabla^2f\|_{L^2_x}.
\end{align}
\end{remark}

The intrinsic and extrinsic formulations are equivalent in the following sense, see [Section 2, \cite{LZ}].
\begin{lemma}\label{new}
Suppose that $Q$ is an admissible harmonic map in Definition 1.1.
If $u\in \bf{H}_{Q}^k$ then for $k=2,3$
\begin{align}
\|u\|_{{\bf H}_{Q}^k}\thicksim\|u\|_{\mathfrak{H}^k},
\end{align}
in the sense that there exist continuous functions $\mathcal{P},\mathcal{Q}$ such that
\begin{align}
\|u\|_{{\bf H}^k_Q}&\le \mathcal{P}(\|u\|_{\mathfrak{H}^k})C(R_0,\|u\|_{\mathfrak{H}^2})\label{jia8}\\
\|u\|_{\mathfrak{H}^k}&\le \mathcal{Q}(\|u\|_{\bf{H}^k_Q})C(R_0,\|u\|_{{\bf{H}}_{Q}^2})\label{jia9}.
\end{align}
\end{lemma}

Lemma \ref{new} and its proof imply the following corollary, by which we can view Theorem 1.1 as a small data problem in the intrinsic sense.
The proof of Corollary \ref{new2} is presented in Section 7.
\begin{corollary}\label{new2}
If $(u_0,u_1)$ belongs to ${\bf{H}^3}\times {\bf{H}^2}$ satisfying (\ref{as3})
then for $0<\mu_1\le 1,0<\mu_2\le 1$
\begin{align}
\|\nabla du_0\|_{L^2}+\|\nabla u_1\|_{L^2}+\|du\|_{L^2}+\|u_1\|_{L^2}\le C(R_0)\mu_2+C(R_0)\mu_1.
\end{align}
\end{corollary}

\begin{lemma}\label{8.5}
We have the decay estimates for heat equations on $\Bbb H^2$:
\begin{align}
\|e^{s\Delta_{\Bbb H^2}}f\|_{L^{\infty}_x}&\lesssim e^{-\frac{s}{4}}s^{-1}\|f\|_{L^{1}_x}\label{huhu899}\\
\|e^{s\Delta_{\Bbb H^2}}f\|_{L^{2}_x}&\lesssim e^{-\frac{s}{4}}\|f\|_{L^{2}_x}\label{m8}\\
\|e^{s\Delta_{\Bbb H^2}}f\|_{L^p_x}&\lesssim s^{\frac{1}{p}-\frac{1}{r}}\|f\|_{L^{r}_x},\label{huhu89}\\
\|e^{s\Delta_{\Bbb H^2}}(-\Delta_{\Bbb H^2})^{\alpha} f\|_{L^{q}_x}&\lesssim s^{-\alpha}e^{-\delta s}\|f\|_{L^{q}_x},\label{mm8}
\end{align}
where $1\le r\le p\le\infty$, $\alpha\in[0,1]$, $1<q<\infty$, $0<\delta\ll1$.
\end{lemma}
\begin{proof}
(\ref{huhu899}) and (\ref{huhu89}) are known in the literature, see \cite{LZ,Coding}. (\ref{m8}) is a corollary of the spectral gap of $\frac{1}{4}$ for $-\Delta_{\Bbb H^2}$. The $s^{-\alpha}$ part of (\ref{mm8}) follows by interpolation between the three estimates of [Lemma 2.11,\cite{LOS}].
Thus it suffices to prove (\ref{mm8}) for $s$ large. The case of (\ref{mm8}) when $\alpha=0$ follows by directly estimating the heat kernel given in \cite{BM}. Since one has $e^{s\Delta}(-\Delta)^{\alpha}f=e^{\frac{s}{2}\Delta}e^{\frac{s}{2}\Delta}(-\Delta)^{\alpha}f$, by applying the exponential decay $L^p-L^p$ estimate to the first $e^{\frac{s}{2}\Delta}$ and the $s^{-\alpha}$ decay of $L^p\to (-\Delta)^{\alpha}L^p$ for the second $e^{\frac{s}{2}\Delta}$ proved just now, we obtain the full (\ref{mm8}).
\end{proof}

The $\Bbb R^2$ version of the following lemma was proved in [Lemma 2.5,\cite{Tao4}]. We remark that the same arguments work in the $\Bbb H^2$ case, because the proof in \cite{Tao4} only uses the decay estimate (\ref{huhu89}) and the self-ajointness of $e^{t\Delta_{\Bbb R^2}}$, which are also satisfied by $e^{t\Delta_{\Bbb H^2}}$.
\begin{lemma}\label{ktao1}
For $f\in L^2_x$ defined on $\Bbb H^2$, one has
$$\int^{\infty}_0\|e^{s\Delta_{\Bbb H^2}}f\|^2_{L^{\infty}_x}ds\lesssim \|f\|^2_{L^2_x}.
$$
\end{lemma}
Without confusion, we will always use $\Delta$ instead of $\Delta_{\Bbb H^2}$.

\subsection{The Local and conditional global well-posedness}
We quickly sketch the local well-posedness and conditional global well-posedness for (\ref{wmap1}).
The local well-posedness of (\ref{wmap1}) for $(u_0,u_1)\in{\bf{H}^3\times\bf{H}^2}$ is standard by fixed point argument. Thus we present the following lemma with a rough proof.

\begin{lemma}\label{local}
For any initial data $(u_0,u_1)\in \bf{H}^3\times\bf{H}^2$, there exists $T>0$ depending only on $\|(u_0,u_1)\|_{\bf{H}^3\times\bf{H}^2}$ such that (\ref{wmap1}) has a unique local solution $(u,\partial_tu)\in C([0,T];\bf{H}^3\times\bf{H}^2)$.
\end{lemma}
\begin{proof}
In the coordinates (\ref{vg}), (\ref{wmap1}) can be written as the following semilinear wave equation
\begin{align}\label{XV12}
\frac{{{\partial ^2}{u^k}}}{{\partial {t^2}}} - {\Delta}{u^k} + {\overline\Gamma}_{ij}^k\frac{{\partial {u^i}}}{{\partial t}}\frac{{\partial {u^j}}}{{\partial t}} - {h^{ij}}{\overline\Gamma} _{mn}^k\frac{{\partial {u^m}}}{{\partial {x^i}}}\frac{{\partial {u^n}}}{{\partial {x^j}}} = 0.
\end{align}
Notice that $\bf{H}^3$ and $\bf{H}^2$ are embedded to $L^{\infty}$  as illustrated in Remark 9.1, we can prove the local well-posedness of (\ref{XV12}) by the standard contradiction mapping argument in the complete metric space $\bf{H}^3\times\bf{H}^2$ with the metric given by
\begin{align*}
{\rm{dist}}(u,w)=\sum^2_{j=1}\|u^j-w^j\|_{H^3}+\sum^2_{j=1}\|\partial_tu^j-\partial_tw^j\|_{H^2}.
\end{align*}
Moreover we can obtain the blow-up criterion: $T_*>0$ is the lifespan of $(\ref{XV12})$ if and only if
\begin{align}\label{09ijn}
\mathop {\lim }\limits_{t \to T_*} {\left\| {(u(t,x),\partial_tu(t,x))} \right\|_{{\bf{H}^3\times\bf{H}^2}}} = \infty.
\end{align}
\end{proof}

The conditional global well-posedness is given by the following proposition. We remark that in the flat case $M=\Bbb R^d$, $1\le d\le 3$, Theorem 7.1 of Shatah, Struwe  \cite{SStruwe} gave a local theory for Cauchy problem in $H^2\times H^1$.
\begin{proposition}\label{global}
Let $(u_0,u_1)\in\bf{H}^3\times \bf{H}^2$ be the initial data of (\ref{wmap1}), $T_*$ is the maximal lifespan determined by Lemma \ref{local}. If the solution $(u,\partial_tu)$ satisfies uniformly for all $t\in[0,T_*)$
\begin{align}\label{hcxp}
\|\nabla du\|_{L^2_x}+\|du\|_{L^2_x}+\|\nabla\partial_t u\|_{L^2_x}+\|\partial_t u\|_{L^2_x}\le C_1,
\end{align}
for some $C_1>0$ independent of $t\in[0,T_*)$ then $T_*=\infty$.
\end{proposition}
\begin{proof} By the local well-posedness in Lemma \ref{local}, it suffices to obtain a uniform bound for $\|(u,\partial_t u)\|_{\bf{H}^3\times\bf{H}^2 }$ with respect to $t\in[0,T]$. By Lemma \ref{new}, it suffices to prove the intrinsic norms are uniformly bounded up to order three.
We first point out a useful inequality which can be verified by integration by parts
\begin{align}\label{V6}
\|\nabla^2 du\|^2_{L^2_x}\lesssim \|\nabla \tau(u)\|^2_{L^2_x}+ \|du\|^6_{L^6_x}+\|\nabla d u\|^2_{L^4_x}\|d u\|^2_{L^4_x}+C(\|u \|^2_{\mathfrak{H}^2}),
\end{align}
where $\tau(u)$ denotes the tension field which in the local coordinates is written as
\begin{align*}
\tau(u)=\left(\Delta u^k+h^{pq}\overline{\Gamma}^k_{ij}\frac{\partial u^{i}}{\partial x^p}\frac{\partial u^{j}}{\partial x^q}\right)\frac{\partial}{\partial y^k}.
\end{align*}
Thus (\ref{V6}), Gagliardo-Nirenberg  inequality and Young inequality further yield
\begin{align}\label{equil}
\|\nabla^2 du\|^2_{L^2_x}\lesssim \mathcal{P}(\|u \|^2_{\mathfrak{H}^2})+\|\nabla \tau(u)\|^2_{L^2_x},
\end{align}
where $\mathcal{P}(x)$ is some polynomial.
Define
\begin{align*}
 {E_3}(u,{\partial _t}u) = \frac{1}{2}\int_{{\Bbb H^2}} {{{\left| {\nabla \tau (u)} \right|}^2}} {\rm{dvo}}{{\rm{l}}_{\rm{h}}} + \frac{1}{2}{\int_{{\Bbb H^2}} {\left| {{\nabla ^2}{\partial _t}u} \right|} ^2}{\rm{dvo}}{{\rm{l}}_{\rm{h}}}.
\end{align*}
Then integration by parts yields
\begin{align*}
\frac{d}{{dt}}{E_3}(u,{\partial _t}u) &= \int_{{\Bbb H^2}} {{h^{ii}}\left\langle {{\nabla _t}{\nabla _i}\tau (u),{\nabla _i}\tau (u)} \right\rangle } {\rm{dvo}}{{\rm{l}}_{\rm{h}}}\\
&+ \int_{{\Bbb H^2}} {{h^{ii}h^{jj}}\left\langle {{\nabla _t}{\nabla _i}{\nabla _j}{\partial _t}u - \Gamma _{ij}^k{\nabla _k}{\partial _t}u,{\nabla _i}{\nabla _j}{\partial _t}u}-\Gamma_{ij}^k{\nabla _k}{\partial _t}u \right\rangle {\rm{dvo}}{{\rm{l}}_{\rm{h}}}}.
\end{align*}
Furthermore we have
\begin{align*}
 &\int_{{\Bbb H^2}} {{h^{ii}h^{jj}}\left\langle {{\nabla _t}{\nabla _i}{\nabla _j}{\partial _t}u - \Gamma _{ij}^k{\nabla _k}{\partial _t}u,{\nabla _i}{\nabla _j}{\partial _t}u}-\Gamma_{ij}^k{\nabla _k}{\partial _t}u \right\rangle  {\rm{dvo}}{{\rm{l}}_{\rm{h}}}}  \\
 &= \int_{{\Bbb H^2}} {{h^{ii}h^{jj}}\left\langle {{\nabla _i}{\nabla _j}{\nabla _t}{\partial _t}u - \Gamma _{ij}^k{\nabla _k}{\partial _t}u,{\nabla _i}{\nabla _j}{\partial _t}u}-\Gamma_{ij}^k{\nabla _k}{\partial _t}u \right\rangle  {\rm{dvo}}{{\rm{l}}_{\rm{h}}}} \\
 &+ \int_{\Bbb H^2}O\big(\left| {\nabla {\partial _t}u} \right|\left| {{\nabla ^2}{\partial _t}u} \right|\big){\rm{dvol_h}} + O\big(\int_{\Bbb H^2}\left| {du} \right|\left| {{\partial _t}u} \right|\left| {\nabla {\partial _t}u} \right|\left| {{\nabla ^2}{\partial _t}u} \right|{\rm{dvol_h}} \big)\\
 &+ \int_{\Bbb H^2}O\big(\left| {\nabla {\partial _t}u} \right|\left| {{\nabla ^2}{\partial _t}u} \right|\big){\rm{dvol_h}}+ \int_{\Bbb H^2}O\big({\left| {du} \right|^2}{\left| {{\partial _t}u} \right|^2}\left| {{\nabla ^2}{\partial _t}u} \right|\big){\rm{dvol_h}}\\
 &+ \int_{\Bbb H^2}O\big({\left| {{\partial _t}u} \right|^2}\left| {\nabla du} \right|\left| {{\nabla ^2}{\partial _t}u} \right|\big){\rm{dvol_h}}
\end{align*}
Since $u$ solves (\ref{wmap1}), $\nabla_t\partial_tu=\tau(u)$. Then by integration by parts the leading term can be expanded as
\begin{align*}
 &\int_{{\Bbb H^2}} {{h^{ii}h^{jj}}}\left\langle {{\nabla _i}{\nabla _j}{\nabla _t}{\partial _t}u,{\nabla _i}{\nabla _j}{\partial _t}u}-\Gamma_{ij}^k{\nabla _k}{\partial _t}u \right\rangle  {\rm{dvol_h}}\\
 &= \int_{{\Bbb H^2}}{h^{ii}h^{jj}}\left\langle {{\nabla _i}{\nabla _j}\tau (u),{\nabla _i}{\nabla _j}{\partial _t}u}-\Gamma_{ij}^k{\nabla _k}{\partial _t}u \right\rangle   {\rm{dvol_h}}  \\
 &=-\int_{{\Bbb H^2}} {{h^{ii}}\left\langle {{\nabla _i}\tau (u),{\nabla _t}{\nabla _i}\tau (u)} \right\rangle {\rm{dvo}}{{\rm{l}}_{\rm{h}}}}  + \int_{\Bbb H^2}O\big(\left| {\nabla \tau (u)} \right|\left| {du} \right|\left| {{\partial _t}u} \right|\left| {\tau (u)} \right|\big){\rm{dvol_h}}\\
 &+ \int_{\Bbb H^2}O\big(\left| {\nabla \tau (u)} \right|\left| {{\nabla ^2}u} \right|\left| {du} \right|\left| {{\partial _t}u} \right|\big){\rm{dvol_h}}+ \int_{\Bbb H^2}O\big(\left| {\nabla \tau (u)} \right|\left| {{\partial _t}u} \right|{\left| {du} \right|^2}\big){\rm{dvol_h}} \\
 &+ \int_{\Bbb H^2}O\big(\left| {\nabla \tau (u)} \right|\left| {\nabla {\partial _t}u} \right|{\left| {du} \right|^2}\big){\rm{dvol_h}} +\int_{\Bbb H^2}O\big( \left| {\nabla \tau (u)} \right|\left| {{\partial _t}u} \right|{\left| {du} \right|^3}\big){\rm{dvol_h}}\\
 &+ \int_{\Bbb H^2}O\big(\left| {\nabla {\partial _t}u} \right|\left| {\nabla \tau (u)} \right|\big){\rm{dvol_h}}.
 \end{align*}
 Thus we conclude
 \begin{align*}
 &\frac{d}{{dt}}{E_3}(u,{\partial _t}u)\\
 &\le {\left\| {du} \right\|_{L_x^8}}{\left\| {{\partial _t}u} \right\|_{L_x^4}}{\left\| {\nabla {\partial _t}u} \right\|_{L_x^8}}{\left\| {{\nabla ^2}{\partial _t}u} \right\|_{L_x^2}} + {\left\| {\nabla {\partial _t}u} \right\|_{L_x^2}}{\left\| {{\nabla ^2}{\partial _t}u} \right\|_{L_x^2}} \\
 &+ {\left\| {{\nabla ^2}{\partial _t}u} \right\|_{L_x^2}}\left\| {du} \right\|_{L_x^8}^2\left\| {{\partial _t}u} \right\|_{L_x^8}^2 + {\left\| {\nabla du} \right\|_{L_x^6}}\left\| {{\partial _t}u} \right\|_{L_x^6}^2{\left\| {\nabla^2 {\partial _t}u} \right\|_{L_x^2}} \\
 &+ {\left\| {\nabla \tau (u)} \right\|_{L_x^2}}{\left\| {\nabla du} \right\|_{L_x^6}}{\left\| {du} \right\|_{L_x^6}}{\left\| {{\partial _t}u} \right\|_{L_x^6}} + {\left\| {\nabla \tau (u)} \right\|_{L_x^2}}{\left\| {\nabla du} \right\|_{L_x^4}}\left\| {du} \right\|_{L_x^8}^2 \\
 &+ {\left\| {\nabla \tau (u)} \right\|_{L_x^2}}\left\| {du} \right\|_{L_x^{12}}^3{\left\| {{\partial _t}u} \right\|_{L_x^4}} + {\left\| {\nabla \tau (u)} \right\|_{L_x^2}}{\left\| {{\partial _t}u} \right\|_{L_x^6}}{\left\| {\tau (u)} \right\|_{L_x^6}}{\left\| {du} \right\|_{L_x^6}} \\
 &+ {\left\| {\nabla \tau (u)} \right\|_{L_x^2}}{\left\| {{\partial _t}u} \right\|_{L_x^6}}\left\| {du} \right\|_{L_x^8}^2
 + {\left\| {\nabla \tau (u)} \right\|_{L_x^2}}{\left\| {\nabla {\partial _t}u} \right\|_{L_x^2}} + {\left\| {\nabla \tau (u)} \right\|_{L_x^2}}{\left\| {\nabla^2{\partial _t}u} \right\|_{L_x^2}}.
\end{align*}
Hence Young's inequality, Sobolev embedding and (\ref{V6}), (\ref{equil}) give
\begin{align*}
\frac{d}{{dt}}{E_3}(u,{\partial _t}u) \le C{E_3}(u,{\partial _t}u)+C.
\end{align*}
where $C$ depends only on $C_1$ in (\ref{hcxp}).
Thus Gronwall shows
\begin{align*}
{E_3}(u,{\partial _t}u) \le e^{Ct}({E_3}({u_0},{u_1}) + C) .
\end{align*}
If $T_*<\infty$ this contradicts with (\ref{09ijn}).
\end{proof}

\subsection{Geometric identities related to Gauges}
Let $\{e_1(t,x),e_2(t,x)\}$ be an orthonormal frame for $u^*(T\mathbb{H}^2)$. Let $\phi_\alpha=(\psi^1_\alpha,\psi^2_\alpha)$ for $\alpha=0,1,2$ be the components of $\partial_{t,x}u$ in the frame $\{e_1,e_2\}$, i.e.,
$$
\phi_\alpha^j = \left\langle {{\partial _\alpha}u,{e_j}} \right\rangle.
$$
For given $\Bbb R^2$-valued function $\phi$ defined on $[0,T]\times\Bbb H^2$, associate $\phi$ with a tangent filed $e\phi$ on $u^*(TN)$ by
\begin{align}\label{poill}
\phi\leftrightarrow e\phi=\sum^2_{j=1}\phi^je_j,
\end{align}
The map $u$ induces a covariant derivative on the trivial boundle $([0,T]\times\Bbb H^2,\Bbb R^2)$ defined by
$$D_\alpha\phi=\partial_\alpha \phi+[A_\alpha]\phi,
$$
where the coefficient matrix is defined by
\begin{align*}
[{A_\alpha}]^k_j = \left\langle {{\nabla_\alpha}{e_j},{e_k}} \right\rangle.
\end{align*}
It is easy to check the torsion free identity
\begin{align}\label{pknb}
D_\alpha\phi_\beta=D_\beta\phi_\alpha,
\end{align}
and the commutator identity
\begin{align}\label{commut1}
e[D_\alpha,D_\beta]\phi=e(\partial_\alpha A_\beta-\partial_\beta A_\alpha)\phi+e[A_\alpha,A_\beta]\phi=\mathbf{R}(u)(\partial_\alpha u, \partial_\beta u)(e\phi).
\end{align}
In the two dimensional case, (\ref{commut1}) can be further simplified to
\begin{align}\label{commut}
e[D_\alpha,D_\beta]\phi=e(\partial_\alpha A_\beta-\partial_\beta A_\alpha)\phi=\mathbf{R}(u)(\partial_\alpha u, \partial_\beta u)(e\phi).
\end{align}
\noindent{\bf{Remark 2.1}}
Sometimes in the same line, we will use both the intrinsic quantities such as ${\bf{R}}(\partial_tu,\partial_su)$ and frame dependent quantities such as $\phi_i$. This will not cause trouble by remembering the correspondence (\ref{poill}). And we define a matrix valued function $\bf{a}\wedge\bf{b}$ by
\begin{align}\label{nb890km}
(\bf{a}\wedge\bf{b})\bf{c}=\left\langle {\bf{a},\bf{c}} \right\rangle\bf{b} - \left\langle {\bf{b},\bf{c}} \right\rangle \bf{a},
\end{align}
where $\bf{a},\bf{b},\bf{c}$ are vectors on $\Bbb R^2$.
It is easy to see (\ref{nb890km}) coincide with (\ref{2.best}) by letting $X=a_1e_1+a_2e_2$, $Y=b_1e_1+b_2e_2$, $Z=c_1e_1+c_2e_2$. Hence (\ref{commut}) can be written as
\begin{align}\label{nb90km}
[D_\alpha,D_\beta]\phi=(\phi_\alpha\wedge\phi_{\beta})\phi
\end{align}

\begin{lemma}
With the notions and notations given above, (\ref{wmap1}) can be written as
\begin{align}\label{jnk}
D_t\phi_t-h^{ij}D_i\phi_j+h^{ij}\Gamma^k_{ij}\phi_k=0¡£
\end{align}
\end{lemma}
\begin{proof}
In the intrinsic formulation, (\ref{wmap1}) can be written as
\begin{align*}
{\nabla _t}{\partial _t}u - \left( {{\nabla _{{x_i}}}{\partial _{{x_j}}}u - {u_*}({\nabla _{\frac{\partial}{\partial x_i}}}\frac{\partial}{\partial {x_j}})} \right){h^{ij}} = 0.
\end{align*}
Expanding $\nabla_i\partial_j u$ and $u_*(\nabla_i\partial_j)$ by the frame $\{e_i\}^2_{i=1}$ yields
\begin{align*}
&{h^{ij}}{\nabla _i}{\partial _j}u - {h^{ij}}{u_*}({\nabla _{\frac{\partial }{{\partial {x_i}}}}}\frac{\partial }{{\partial {x_j}}}) = \sum\nolimits_{l =1}^2{{h^{ij}}} {\nabla _i}\left( {\left\langle {{\partial _j}u,{e_l}} \right\rangle {e_l}} \right) - \Gamma _{i,j}^k{h^{ij}}{\partial _k}u \\
&= {h^{ij}}\left( {{\partial _i}\psi _j^p{e_p} + [{A_i}]_l^p\psi _j^l{e_p}} \right) - \Gamma _{i,j}^k{h^{ij}}\psi _k^l{e_l} = e{h^{ij}}\left( {{D_i}{\phi _j}} \right) - e\Gamma _{i,j}^k{h^{ij}}{\phi _k}¡£
\end{align*}
And $\nabla_t\partial_tu$ is expanded as
\begin{align*}
{\nabla _t}{\partial _t}u = \sum\nolimits_{l = 1}^2 {{\nabla _t}\left( {\left\langle {{\partial _t}u,{e_l}} \right\rangle {e_l}} \right)}  = \left( {{\partial _t}\phi _0^p{e_p} + [{A_0}]_l^p\phi _0^l{e_p}} \right) = e\left( {{D_0}{\phi _0}} \right).
\end{align*}
Hence (\ref{jnk}) follows.
\end{proof}

\section{Caloric Gauge}
Denote the space  $C([0,T];\bf{H}^3\times\bf{H}^2)$ by $\mathcal{X}_T$.
The caloric gauge was first introduced by Tao \cite{Tao4} for the wave maps from $\Bbb R^{2+1}$ to $\mathbb{H}^n$. We give the definition of the caloric gauge in our setting.
\begin{definition}\label{pp}
Let  $u(t,x):[0,T]\times \mathbb{H}^2\to \mathbb{H}^2$ be a solution of (\ref{wmap1}) in $\mathcal{X}_T$. Suppose that the heat flow initiated from $u_0$ converges to a harmonic map $Q:\mathbb{H}^2\to \mathbb{H}^2$. Then for a given orthonormal frame $\Xi(x)\triangleq\{\Xi_j(Q(x))\}^2_{j=1}$ which spans the tangent space $T_{Q(x)}\mathbb{H}^2$ for any $x\in \mathbb{H}^2$, by saying a caloric gauge we mean a tuple consisting of a map  $\widetilde{u}:\Bbb R^+\times [0,T]\times\mathbb{H}^2\to\Bbb H^2$ and an orthonormal frame $\Omega\triangleq\{\Omega_j(\widetilde{u}(s,t,x))\}^2_{j=1}$ such that
\begin{align}\label{muqi}
\left\{ \begin{array}{l}
{\partial _s}\widetilde{u}= \tau (\widetilde{u}) \\
{\nabla _s}{\Omega _j} = 0 \\
\mathop {\lim }\limits_{s \to \infty } {\Omega _j} = {\Xi _j} \\
\end{array} \right.
\end{align}
where the convergence of frames is defined by
\begin{align}\label{convergence}
\left\{ \begin{array}{l}
 \mathop {\lim }\limits_{s \to \infty } \widetilde{u}(s,t,x) = Q(x) \\
 \mathop {\lim }\limits_{s \to \infty } \left\langle {{\Omega _i}(s,t,x),{\Theta _j}(\widetilde{u}(s,t,x))} \right\rangle  = \left\langle {{\Xi _i}(Q(x)),{\Theta _j}(Q(x))} \right\rangle  \\
 \end{array} \right.
\end{align}
\end{definition}

The remaining part of this section is devoted to the existence of the caloric gauge.
\subsection{Warming up for the heat flows}
In this subsection, we prove the estimates needed for the existence of the caloric gauge and the bounds for connection coefficients.

The equation of the heat flow is given by
\begin{align}\label{8.29.1}
\left\{ \begin{array}{l}
 {\partial _s}u = \tau (u) \\
 u(0,x)= v(x) \\
 \end{array} \right.
\end{align}
The energy density $e$ is defined by
$$
e(u)=\frac{1}{2}|du|^2.
$$

The following lemma is due to Li, Tam \cite{LT}. (\ref{VI4}), (\ref{uu}) are proved in \cite{LZ}.
\begin{lemma}\label{8.44}
Given initial data $v:\Bbb H^2\to\Bbb H^2$ with bounded energy density, suppose that $\tau(v)\in L^p_x$ for some $p>2$ and the image of $\Bbb H^2$ under the map $v$ is contained in a compact subset of $\Bbb H^2$. Then the heat flow equation (\ref{8.29.1}) has a global solution $u$. Moreover for some $K,C>0$, we have
\begin{align}
(\partial_s-\Delta)|du|^2+2|\nabla du|^2&\le K|du|^2\label{8.4}\\
(\partial_s-\Delta)|\partial_s u|^2+2|\nabla \partial_su|^2&\le 0\label{8.3}\\
(\partial_s-\Delta)|\partial_s u|&\le 0\label{VI4}\\
(\partial_s-\Delta)(|du|e^{-Cs})&\le 0.\label{uu}
\end{align}
\end{lemma}

Consider the heat flow from $\mathbb{H}^2$ to $\mathbb{H}^2$ with a parameter
\begin{align}\label{8.29.2}
\left\{ \begin{array}{l}
 {\partial _s}\widetilde{u} = \tau (\widetilde{u}) \\
 \widetilde{u}(s,t,x) \upharpoonright_{s=0}= u(t,x) \\
 \end{array} \right.
\end{align}

We will give two types of estimates of $\nabla^k\partial_s\widetilde{u},\nabla^k\partial_x\widetilde{u}$ in the following. One is the decay of $\|\nabla^k\partial_s\widetilde{u}\|_{L^{2}_x}$ as $s\to\infty$ which can be easily proved via energy arguments. The other is the global boundedness of $\|\partial_x\widetilde{u}\|_{L^{\infty}_x}$ away from $s=0$ and the decay of $\|\nabla^k\partial_s\widetilde{u}\|_{L^{\infty}_x}$ as $s\to\infty$, both of which need additional efforts. And we will prove the decay estimates with respect to $s$ for $\|\partial_t\widetilde{u}\|_{{L^{\infty}_x}\bigcap L^2_x}$, which is the key integrability gain to compensate the loss of decay of $\partial_x\widetilde{u}$. We start with the estimate of $\|d\widetilde{u}\|_{L^{\infty}_x}$ which is the cornerstone for all other estimates.

\begin{remark}\label{ki78}
The following inequality which can be verified by Moser iteration is known in the heat flow literature:
If $v$ is a nonnegative function satisfying
\begin{align*}
\partial_t v-\Delta v\le 0,
\end{align*}
then for $t\ge1$,
\begin{align*}
v(x,t)\le \int^t_{t-1}\int_{B(x,1)}v(y,s){\rm{dvol_y}}ds.
\end{align*}
\end{remark}

Introduce the norm:
\begin{align}\label{ytgvfre}
\|u(t,x)\|_{\mathcal{X}_T}=&\|\nabla du\|_{C([0,T];L^2_x)}+\|\nabla \partial_tu\|_{C([0,T];L^2_x)}\nonumber\\
&+\| du\|_{C([0,T];L^2_x)}+\| \partial_tu\|_{C([0,T];L^2_x)}.
\end{align}
Trivial applications of Remark \ref{ki78}, (\ref{uu}) and the non-increasing of the energy along the heat flow give the bounds for $\|d\widetilde{u}\|_{L^{\infty}_x}$. See also \cite{LZ} for another proof.
\begin{lemma}\label{density}
Let $(u,\partial_tu)$ solve $(\ref{wmap1})$ in $\mathcal{X}_T$ (see (\ref{ytgvfre}) ) with $\|u\|_{\mathcal{X}_T}\le M$. If $\widetilde{u}$ is the solution to $(\ref{8.29.2})$ with initial data $u(t,x)$, then we have uniformly for $t\in[0,T]$, $s\in[1,\infty)$
\begin{align}
\left\| d\widetilde{u}(s,t,x) \right\| _{L_x^\infty }&\lesssim \left\| d u(t,x) \right\|_{L_x^2},\label{3.14a}
\end{align}
\end{lemma}

The decay of $\|\nabla^k\partial_s\widetilde{u}\|_{L^2_x}$ follows from an energy argument and the bound of the energy density provided by (\ref{3.14a}).
\begin{lemma}\label{chen2222}
Let $(u,\partial_tu)\in \mathcal{X}_T$ with $\|u\|_{\mathcal{X}_T}\le M$, then for some universal constant $\delta>0$ the solution $\widetilde{u}(s,t,x)$ to heat flow (\ref{8.29.2}) satisfies
\begin{align}
\|\partial_s\widetilde{u}(s,t,x)\|_{L^2_x}&\lesssim e^{-\delta s}MC(M),  \mbox{  }{\rm{for}}\mbox{  }s>0\label{f41}\\
\|\nabla\partial_s\widetilde{u}(s,t,x)\|_{L^2_x}&\lesssim e^{-\delta s}MC(M), \mbox{  }{\rm{for}}\mbox{  }s\ge2\label{f42}\\
\int^{\infty}_0\|\nabla\partial_s\widetilde{u}(s,t,x)\|^2_{L^2_x}ds&\lesssim MC(M).\label{f40}
\end{align}
for all $t\in[0,T]$. The constant $C(M)$ grows polynomially as $M$ grows.
\end{lemma}
\begin{proof}
First we notice that (\ref{8.3}), (\ref{huhu899}) and maximum principle yield
\begin{align}
\|\partial_s\widetilde{u}(s,t,x)\|_{L^{2}_x}&\lesssim e^{-\frac{s}{4}}\|\partial_s\widetilde{u}(0,t,x)\|_{L^{2}_x}\label{lao1}\\
\|\partial_s\widetilde{u}(s,t,x)\|_{L^{\infty}_x}&\lesssim s^{-1}e^{-\frac{s}{4}}\|\partial_s\widetilde{u}(0,t,x)\|_{L^{2}_x}.\label{lao2}
\end{align}
We introduce three energy functionals:
\begin{align*}
{\mathcal{E}_1}(\widetilde{u}) = \frac{1}{2}\int_{{\Bbb H^2}} {{{\left| {\nabla \widetilde{u}} \right|}^2}} dx,\mbox{  }{\mathcal{E}_2}(u) = \frac{1}{2}\int_{{\Bbb H^2}} {{{\left| {{\partial _s}\widetilde{u}} \right|}^2}} {\rm{dvol_h}},\mbox{  }{\mathcal{E}_3}(u) = \frac{1}{2}\int_{{\Bbb H^2}} {{{\left| {\nabla {\partial _s}\widetilde{u}} \right|}^2}} {\rm{dvol_h}}.
\end{align*}
By integration by parts and (\ref{8.29.2}), we have
\begin{align*}
\frac{d}{{ds}}{\mathcal{E}_1}(\widetilde{u}) =  - \int_{{\Bbb H^2}} {{{\left| {\tau (\widetilde{u})} \right|}^2}} {\rm{dvol_h}}.
\end{align*}
Thus the energy is decreasing with respect to $s$ and
\begin{align}\label{V1}
\|d\widetilde{u}\|^2_{L^2_x}+\int^s_0\|\partial_s \widetilde{u}\|^2_{L^2_x}\le \mathcal{E}_1(u_0).
\end{align}
The non-positive sectional curvature assumption with integration by parts yields
$$\|\nabla d\widetilde{u}(s)\|^2_{L^2_x}\le \|\tau(\widetilde{u}(s))\|^2_{L^2_x}+\|d\widetilde{u}\|^2_{L^2_x}
$$
Hence by (\ref{V1}), (\ref{8.29.2}) we conclude
\begin{align}\label{V5}
\|\widetilde{u}\|^2_{\mathfrak{H}^2}+\int^s_0\|\partial_s \widetilde{u}\|^2_{L^2_x}ds\lesssim\|u_0\|^2_{\mathfrak{H}^2}.
\end{align}
Again by (\ref{8.29.2}) and integration by parts, one has
\begin{align}
 \frac{d}{{ds}}{\mathcal{E}_2}(\widetilde{u}) &= \int_{{\Bbb H^2}} {\left\langle {{\nabla _s}{\partial _s}\widetilde{u},{\partial _s}\widetilde{u}} \right\rangle } {\rm{dvol_h}}=\int_{{\Bbb H^2}} {\left\langle {{\nabla _s}\tau (\widetilde{u}),{\partial _s}\widetilde{u}} \right\rangle } {\rm{dvol_h}} \nonumber\\
 &\le -\int_{{\Bbb H^2}} {\left\langle {\nabla {\partial _s}(\widetilde{u}),\nabla {\partial _s}\widetilde{u}} \right\rangle } {\rm{dvol_h}}+ C\int_{{\Bbb H^2}} {{{\left| {d\widetilde{u}} \right|}^2}{{\left| {{\partial _s}\widetilde{u}} \right|}^2}{\rm{dvol_h}}}.\label{j089}
\end{align}
Integrating (\ref{j089}) with respect to $s$ in $(s_1,s_2)$ for any $1<s_1<s_2$, we infer from (\ref{lao2}) and (\ref{V1}) that
\begin{align}
\mathcal{E}_2(\widetilde{u}(s_2))-{\mathcal{E}_2}(\widetilde{u}(s_1))+\int^{s_2}_{s_1}\mathcal{E}_3(\widetilde{u}(s))ds &\lesssim \left\| { d\widetilde{u}} \right\|^2_{L^{\infty}_s{L_x^2}}\int^{s_2}_{s_1}\left\| {{\partial _s}\widetilde{u}} \right\|^2_{L_x^\infty}ds\lesssim M^4e^{-\delta s_1}.
\end{align}
Then by (\ref{lao1}) we have for $1<s<s_1<s_2$ and any $t\in [0,T]$
\begin{align}\label{lao3}
\int^{s_2}_{s_1}\|\nabla\partial_s\widetilde{u}(\tau,t,x)\|^2_{L^2_x}d\tau\lesssim  M^2e^{-\delta s}.
\end{align}
Integration by parts and (\ref{8.29.2}) yield
\begin{align}
 &\frac{d}{{ds}}{\mathcal{E}_3}(\widetilde{u}(s))\nonumber\\
 &\le  -\int_{{\Bbb H^2}} \big({{{\left| {{\nabla ^2}{\partial _s}\widetilde{u}} \right|}^2}}  + C\left| {d \widetilde{u}} \right|\left| {\nabla {\partial _s}\widetilde{u}} \right|{\left| {{\partial _s}\widetilde{u}} \right|^2} + C\left| {d\widetilde{u}} \right|\left| {\nabla {\partial _s}\widetilde{u}} \right|{\left| {{\partial _s}\widetilde{u}} \right|^2} \big){\rm{dvol_h}}\nonumber\\
 &+ C\int_{\Bbb H^2}\big({\left| {d\widetilde{u}} \right|^3}\left| {{\partial _t}\widetilde{u}} \right|\left| {\nabla {\partial _t}\widetilde{u}} \right| + C\left| {{\nabla ^2}\widetilde{u}} \right|\left| {d\widetilde{u}} \right|\left| {{\partial _t}\widetilde{u}} \right|\left| {\nabla {\partial _t}\widetilde{u}} \right| + C{\left| {{\partial _t}\widetilde{u}} \right|^2}{\left| {d \widetilde{u}} \right|^4}\big){\rm{dvol_h}} \nonumber\\
 &+ C\int_{\Bbb H^2}\big({\left| {\nabla {\partial _t}\widetilde{u}} \right|^2}{\left| {d\widetilde{ u}} \right|^2} + C{\left| {d\widetilde{u}} \right|^2}\left| {{\nabla ^2}{\partial _t}\widetilde{u}} \right|\left| {{\partial _t}\widetilde{u}} \right|\big){\rm{dvol_h}}.\label{9xian}
\end{align}
By H\"older, (\ref{lao3}), (\ref{lao2}), we see for $1<s<s_1<s_2$ and any $t\in [0,T]$
\begin{align*}
&\int_{{s_1}}^{{s_2}} {\int_{{\Bbb H^2}} {\left| {d\widetilde{u}} \right|} \left| {\nabla {\partial _s}\widetilde{u}} \right|{{\left| {{\partial _s}\widetilde{u}} \right|}^2}{\rm{dvol_h}}} ds\\
&\lesssim \left\| {{\partial _s}\widetilde{u}} \right\|_{L_s^4L_x^\infty ([{s_1},{s_2}] \times {\Bbb H^2})}^2{\left\| {d\widetilde{u}} \right\|_{L_s^\infty L_x^2([{s_1},{s_2}] \times {\Bbb H^2})}}{\left\| {\nabla {\partial _s}\widetilde{u}} \right\|_{L_s^2L_x^2([{s_1},{s_2}] \times {\Bbb H^2})}}\\
&\lesssim M^4{e^{ - \delta {s_1}}}.
\end{align*}
Similarly we have from (\ref{lao3}), (\ref{lao2}), (\ref{3.14a}) that for $1<s<s_1<s_2$ and any $t\in [0,T]$
\begin{align*}
 &\int_{{s_1}}^{{s_2}} \int_{{\Bbb H^2}} {{{\left| {d\widetilde{u}} \right|}^3}} \left| {\nabla {\partial _s}\widetilde{u}} \right|\left| {{\partial _s}\widetilde{u}} \right|{\rm{dvol_hds}}\\
 &\lesssim \left\| {d\widetilde{u}} \right\|_{L_s^\infty L_x^\infty ([{s_1},{s_2}] \times {\Bbb H^2})}^3{\left\| {\nabla {\partial _s}\widetilde{u}} \right\|_{L_s^2L_x^2([{s_1},{s_2}] \times {\Bbb H^2})}}{\left\| {{\partial _s}\widetilde{u}} \right\|_{L_s^2L_x^2([{s_1},{s_2}] \times {\Bbb H^2})}} \\
 &\lesssim {M^5}{e^{ - \delta {s_1}}}.
\end{align*}
And similarly we obtain for $1<s<s_1<s_2$ and all $t\in [0,T]$
\begin{align*}
&\int_{{s_1}}^{{s_2}} \int_{{\Bbb H^2}} {\left| {\nabla d\widetilde{u}} \right|\left| {d\widetilde{u}} \right|\left| {{\partial _s}\widetilde{u}} \right|\left| {\nabla {\partial _s}\widetilde{u}} \right|} {\rm{dvol_hds}}\\
&\lesssim {\left\| {d\widetilde{u}} \right\|_{L_s^\infty L_x^\infty }}{\left\| {{\partial _s}\widetilde{u}} \right\|_{L_s^2L_x^\infty }}{\left\| {\nabla d\widetilde{u}} \right\|_{L_s^\infty L_x^2}}{\left\| {\nabla {\partial _s}\widetilde{u}} \right\|_{L_s^2L_x^2}} \\
&\lesssim {M^4}{e^{ - \delta {s_1}}},
\end{align*}
where the integrand domains are $[{s_1},{s_2}] \times {\Bbb H^2}$.
The remaining three terms in (\ref{9xian}) are easier to bound. In fact, Sobolev embedding, (\ref{lao2}) and (\ref{V5}) show
\begin{align*}
\int_{{s_1}}^{{s_2}} {\int_{{\Bbb H^2}} {{{\left| {\nabla\widetilde{ u}} \right|}^4}} {{\left| {{\partial _s}\widetilde{u}} \right|}^2}} {\rm{dvol_hds}} \le \left\| {{\partial _s}\widetilde{u}} \right\|_{L_s^2L_x^\infty}^2\left\| {{\nabla d}\widetilde{u}} \right\|_{L_s^\infty L_x^2}^4 \le {M^6}{e^{ - \delta {s_1}}}.
\end{align*}
Similarly we obtain
\begin{align*}
\int_{{s_1}}^{{s_2}} {\int_{{\Bbb H^2}} {{{\left| {d\widetilde{u}} \right|}^2}} {{\left| {\nabla {\partial _s}\widetilde{u}} \right|}^2}} {\rm{dvol_hds}} \le \left\| {\nabla {\partial _s}\widetilde{u}} \right\|_{L_s^2L_x^2}^2\left\| {d\widetilde{u}} \right\|_{L_s^\infty L_x^\infty }^2 \le {M^4}{e^{ - \delta {s_1}}}.
\end{align*}
The last remaining term in (\ref{9xian}) is absorbed by the negative term on the left. Indeed, for sufficiently small $\eta>0$
\begin{align*}
 &\int_{{s_1}}^{{s_2}} {\int_{{\Bbb H^2}} {{{\left| {d\widetilde{u}} \right|}^2}} \left| {{\nabla ^2}{\partial _s}\widetilde{u}} \right|\left| {{\partial _s}\widetilde{u}} \right|} {\rm{dvol_hds}} \\
 &\lesssim \eta \left\| {{\nabla ^2}{\partial _s}\widetilde{u}} \right\|_{L_s^2L_x^2([{s_1},{s_2}] \times {\Bbb H^2})}^2+ {\eta ^{ - 1}}\left\| {d\widetilde{u}} \right\|_{L_s^\infty L_x^\infty}^4\left\| {{\partial _s}\widetilde{u}} \right\|_{L_s^2L_x^2}^2 \\
 &\lesssim \eta \left\| {{\nabla ^2}{\partial _s}\widetilde{u}} \right\|_{L_s^2L_x^2}^2 + {\eta ^{ - 1}}{M^6}{e^{ - \delta {s_1}}}.
\end{align*}
(\ref{V5}) implies that there exists $s_0\in(1,2)$ such that
\begin{align}
\int_{{\Bbb H^2}} {{\left| {\nabla {\partial _s}u} \right|}^2}({s_0},t,x) {\rm{dvol_h}} \le {M^2}.
\end{align}
Hence applying (\ref{{uv111}}) and Gronwall inequality to (\ref{9xian}), we have for $s>s_0$
\begin{align}
&\int_{{\Bbb H^2}} {\left| {\nabla {\partial _s}\widetilde{u}} \right|}^2(s,t,x){\rm{dvol_h}}\nonumber\\
&\lesssim e^{ - \delta (s - {s_0})}\int_{{\Bbb H^2}} {{\left| {\nabla {\partial _s}\widetilde{u}} \right|}^2}({s_0},t,x) {\rm{dvol_h}} + MC(M)\int_{s_0}^s e^{ - \delta (s - \tau )}e^{ - \delta \tau }d\tau  \nonumber\\
&\lesssim MC(M)\left( e^{ - \delta (s - {s_0})}+ e^{ - \delta s}(s - {s_0})\right).\label{7si}
\end{align}
Since $s_0\in(1,2)$, we have verified $(\ref{f42})$ for $s\ge 2$. (\ref{f40}) follows directly from (\ref{f42}), (\ref{V5}) and integrating (\ref{j089}) with respect to $s$.
\end{proof}

We then prove the pointwise decay of $|\nabla\partial_s\widetilde{u}|$ with respect to $s$.  First we need the Bochner formula for high derivatives of $\widetilde{u}$ along the heat flow. The proof of following four lemmas is direct calculations with the Bochner technique. Considering that the proof is quite standard, we state the results without detailed calculations.
\begin{lemma}\label{9tian}
Let $\widetilde{u}$ be a solution to heat flow equation. Then $|\nabla\partial_s\widetilde{u}|^2$ satisfies
\begin{align}
 {\partial _s}{\left| {\nabla {\partial _s}\widetilde{u}} \right|^2} - {\Delta}{\left| {\nabla {\partial _s}\widetilde{u}} \right|^2} + 2{\left| {{\nabla ^2}{\partial _s}\widetilde{u}} \right|^2} &\lesssim {\left| {\nabla {\partial _s}\widetilde{u}} \right|^2}+ \left| {{\partial _s}\widetilde{u}} \right|{\left| {d\widetilde{u}} \right|^3}\left| {\nabla {\partial _s}\widetilde{u}} \right| \nonumber\\
 &+{\left| {\nabla {\partial _s}\widetilde{u}} \right|^2}{\left| {d\widetilde{u}} \right|^2} + \left| {{\partial _s}\widetilde{u}} \right|\left| {{\nabla }d\widetilde{u}} \right|\left| {\nabla {\partial _s}\widetilde{u}} \right|\left| {d \widetilde{u}} \right|.\label{9tian1}
\end{align}
\end{lemma}

\begin{lemma}\label{B1}
Let $\widetilde{u}$ be a solution to heat flow equation, then we have
\begin{align}
 &{\partial _s}{\left| {\nabla d\widetilde{u}} \right|^2} - {\Delta}{\left| {\nabla d\widetilde{u}} \right|^2} + 2{\left| {{\nabla ^2}d\widetilde{u}} \right|^2} \lesssim {\left| {\nabla d\widetilde{u}} \right|^2} + {\left| {d\widetilde{u}} \right|^2}{\left| {\nabla d\widetilde{u}} \right|^2} \nonumber\\
 &+ {\left| {d\widetilde{u}} \right|}^2 + {\left| {d\widetilde{u}} \right|^4}\left| {\nabla d\widetilde{u}} \right|.\label{ion1}
\end{align}
\end{lemma}

\begin{lemma}\label{8zu}
Let $\widetilde{u}$ be a solution to heat flow equation. Then $|\nabla\partial_t\widetilde{u}|^2$ satisfies
\begin{align}
&(\partial_s-\Delta)|\partial_t \widetilde{u}|^2=-2|\nabla \partial_t \widetilde{u}|^2-\mathbf{R}(\widetilde{u})(\nabla \widetilde{u},\partial_t \widetilde{u},\nabla \widetilde{u}, \partial_t \widetilde{u})\le 0.\label{10.127}\\
 &{\partial _s}{\left| {\nabla {\partial _t}\widetilde{u}} \right|^2} - {\Delta}{\left| {\nabla {\partial _t}\widetilde{u}} \right|^2} + 2{\left| {{\nabla ^2}{\partial _t}\widetilde{u}} \right|^2} \lesssim {\left| {\nabla {\partial _t}\widetilde{u}} \right|^2} +  {\left| {{\partial _s}\widetilde{u}} \right|{{\left| {{d}\widetilde{u}} \right|}^2}\left| {\nabla {\partial _t}\widetilde{u}} \right|}\nonumber \\
 &+ {{{\left| {d\widetilde{u}} \right|}^3}\left| {{\partial _t}\widetilde{u}} \right|\left| {\nabla {\partial _t}\widetilde{u}} \right|}  +  {\left| {d\widetilde{u}} \right|\left| {{\partial _t}\widetilde{u}} \right|\left| {\nabla d\widetilde{u}} \right|\left| {\nabla {\partial _t}\widetilde{u}} \right|}+  {{{\left| {\nabla {\partial _t}\widetilde{u}} \right|}^2}{{\left| {d\widetilde{u}} \right|}^2}} .
 \end{align}
 \end{lemma}

\begin{lemma}
Let $\widetilde{u}$ be a solution to heat flow equation, then
\begin{align}
&{\partial _s}{\left| {{\nabla _t}{\partial _s}\widetilde{u}} \right|^2} - \Delta {\left| {{\nabla _t}{\partial _s}\widetilde{u}} \right|^2} + 2{\left| {\nabla {\nabla _t}{\partial _s}\widetilde{u}} \right|^2} \lesssim  {\left| {{\nabla _t}{\partial _s}\widetilde{u}} \right|^2}+\left| {\nabla d\widetilde{u}} \right|\left| {{\partial _t}\widetilde{u}} \right|\left|{\nabla _t}{\partial _s}\widetilde{u}\right|\left| {{\partial _s}\widetilde{u}} \right| \nonumber\\
&+ \left| {{\nabla _t}{\partial _s}\widetilde{u}} \right|\left| {\nabla {\partial _s}\widetilde{u}} \right|\left| {{\partial _t}\widetilde{u}} \right|\left| {d\widetilde{u}} \right| + \left| {\nabla {\partial _t}\widetilde{u}} \right|\left| {{\partial _s}\widetilde{u}} \right|\left|{\nabla _t}{\partial _s}\widetilde{u}\right|\left| {d\widetilde{u}} \right|+{\left| {{\nabla _t}{\partial _s}\widetilde{u}} \right|^2}{\left| {d\widetilde{u}} \right|^2} .\label{iconm}
\end{align}
\end{lemma}

We have previously seen that the bound of $\|d\widetilde{u}\|_{L^{\infty}}$ is useful for bounding $\|\nabla\partial_s\widetilde{u}\|_{L^{2}}$. In order to bound $\|\nabla\partial_s\widetilde{u}\|_{L^{\infty}}$, it is convenient if one has a bound for $\|\nabla d\widetilde{u}\|_{L^{\infty}}$ firstly.
\begin{lemma}
If $(u(t,x),\partial_tu(t,x))$ is a solution to (\ref{wmap1}) with $\|u(t,x)\|_{\mathcal{X}_T}\le M$. Then for $s\ge 2$
\begin{align}
\|\nabla d\widetilde{u}\|_{L^{\infty}_x}\lesssim MC(M).\label{chen2}
\end{align}
\end{lemma}
\begin{proof}
The proof of (\ref{chen2}) is also based on Remark \ref{ki78}.
One can rewrite (\ref{ion1}) by Young inequality in the following form
\begin{align*}
{\partial _s}{\left| {\nabla d\widetilde{u}} \right|^2} - {\Delta}{\left| {\nabla d\widetilde{u}} \right|^2} + 2{\left| {{\nabla ^2}d\widetilde{u}} \right|^2} \le C\left( {1 + {{\left| {d\widetilde{u}} \right|}^2}} \right){\left| {\nabla d\widetilde{u}} \right|^2} + {\left| {d\widetilde{u}} \right|^2}+ {\left| {d\widetilde{u}} \right|^8}.
\end{align*}
Since for $s\ge1$, $\|d\widetilde{u}\|_{L^{\infty}}\lesssim M$, $\|\partial_s\widetilde{u}\|_{L^{\infty}}\lesssim M$, let $r(s,t,x) = {\left| {\nabla d\widetilde{u}} \right|^2} + M^2 + {M^8}$, then we have
$${\partial _s}r - {\Delta}r \le C\left( {{M^2} + 1} \right)r.$$
Let $v = {e^{ - C\left( {{M^2} + 1} \right)s}}r$. For $s\ge2$, it is obvious that $v$ satisfies
$$
 {\partial _s}v - {\Delta}v \le 0.
$$
By Remark \ref{ki78}, we deduce for $d(x,y)\le 1$, $s\ge2$
\begin{align*}
v(s,t,x)\lesssim \int^{s}_{s-1}\int_{B(x,1)} v(\tau,t,y)d\tau{\rm{dvol_y}}.
\end{align*}
Thus by $\|u\|_{\mathfrak{H}^2}\le M$ and (\ref{V5}), we conclude
$${\left| {\nabla d\widetilde{u}} \right|^2}\left( {{s},t,x} \right)\lesssim MC(M).
$$
Hence (\ref{chen2}) follows.
\end{proof}

Now we prove the decay of $\|\nabla\partial_s\widetilde{u}\|_{L^{\infty}_x}$  as $s\to\infty$.
\begin{lemma}\label{decayingt}
If $(u,\partial_tu)$ is a solution to (\ref{wmap1}) in $\mathcal{X}_T$ with $\|u(t,x)\|_{\mathcal{X}_T}\le M$. Then for some universal constant $\delta>0$
\begin{align}
\|\nabla\partial_s\widetilde{u}\|_{L^{\infty}_x}\lesssim MC(M)e^{-\delta s}, \mbox{  }{\rm{for}} \mbox{  }s\ge1.\label{koo4}
\end{align}
\end{lemma}
\begin{proof}
By (\ref{lao2}), for $s\ge1$
\begin{align}\label{kongkong}
\|\partial_s\widetilde{u}\|_{L^{\infty}_x}\lesssim e^{-\delta s}M.
\end{align}
We can rewrite (\ref{9tian1}) by Young inequality as
\begin{align}
{\partial _s}{\left| {\nabla {\partial _s}\widetilde{u}} \right|^2} - {\Delta}{\left| {\nabla {\partial _s}\widetilde{u}} \right|^2} \le (1 + {\left| {d\widetilde{u}} \right|^2}){\left| {\nabla {\partial _s}\widetilde{u}} \right|^2} + {\left| {d\widetilde{u}} \right|^6}{\left| {{\partial _s}\widetilde{u}} \right|^2} + {\left| {d\widetilde{u}} \right|^2}{\left| {\nabla d\widetilde{u}} \right|^2}{\left| {{\partial _s}\widetilde{u}} \right|^2}.
\end{align}
Let $g(s,t,x) = {\left| {d\widetilde{u}} \right|^6}{\left| {{\partial _s}\widetilde{u}} \right|^2} + {\left| {d\widetilde{u}} \right|^2}{\left| {\nabla d\widetilde{u}} \right|^2}{\left| {{\partial _s}\widetilde{u}} \right|^2}$,
then by Lemma \ref{density}, Lemma \ref{chen2222} and (\ref{kongkong}),  $g(s,t,x) \le C(M)M{e^{ - \delta s}}$ for $s\ge1$. Let $f(s,t,x) = {\left| {\nabla \partial_s \widetilde{u}} \right|^2}\left( {s,t,x} \right) + \frac{1}{\delta }C(M)M{e^{ - \delta s}}$, then
$${\partial _s}f - {\Delta}f \le C\left( {{M^2} + 1} \right)f.$$
Then $\bar v=e^{-C( {{M^2} + 1})s}f$ satisfies
$$
 {\partial _s}\bar v - {\Delta}\bar v \le0.
$$
Applying Remark \ref{ki78} to $\bar{v}$ as before implies
$${\left| {\nabla {\partial _s}\widetilde{u}} \right|^2}\left( {{s},t,x} \right) + \frac{1}{\delta }C(M)M{e^{ - \delta {s}}} \le \int_{{s} - 1}^{{s}} {\int_{{\Bbb H^2}} {{{\left| {\nabla {\partial _s}\widetilde{u}\left( {\tau,t,y} \right)} \right|}^2}{\rm{dvol_hd\tau}}}} +C(M)M{e^{ - \delta {s}}}.
$$
Therefore, (\ref{koo4}) follows from
\begin{align}\label{ingjh}
\int_{{s} - 1}^{{s}} \int_{{\Bbb H^2}} {{{\left| {\nabla {\partial _s}\widetilde{u}\left( {\tau,t,y} \right)} \right|}^2}{\rm{dvol_h}d\tau}}\lesssim MC(M)e^{-\delta s},
\end{align}
which arises from (\ref{f42}).
\end{proof}

We move to the decay for $|\partial_t\widetilde{u}|$ with respect to $s$.
\begin{lemma}
If $(u,\partial_tu)$ is a solution to (\ref{wmap1}) in $\mathcal{X}_T$ with $\|u(t,x)\|_{\mathcal{X}_T}\le M$. Then
\begin{align}
\|\partial_t\widetilde{u}\|_{L^{2}_x}&\lesssim MC(M)e^{-\delta s},\mbox{  }{\rm{for}} \mbox{  }s>0\label{xiu1}\\
\|\partial_t\widetilde{u}\|_{L^{\infty}_x}&\lesssim MC(M)e^{-\delta s},\mbox{  }{\rm{for}} \mbox{  }s\ge1\label{xiu2}\\
\int^{\infty}_0\|\nabla\partial_t\widetilde{u}\|^2_{L^{2}_x}ds&\lesssim MC(M), \label{xiu3}\\
\|\nabla\partial_t\widetilde{u}\|_{L^{\infty}_x}&\lesssim MC(M)e^{-\delta s}, \mbox{  }{\rm{for}} \mbox{  }s\ge1.\label{xiu4}
\end{align}
\end{lemma}
\begin{proof}
The maximum principle and (\ref{huhu899}) imply
\begin{align}\label{sdf8uj}
\| {{\partial _t}\widetilde{u}(s,t,x)} \|^2_{L^{\infty}_x} \le {s^{- 2}}{e^{ - \delta s}} \left\| {{\partial _t}\widetilde{u}(0,t,x)} \right\|^2_{L^2_x}.
\end{align}
Moreover further calculations with (\ref{10.127}) show
\begin{align*}
(\partial_s-\Delta)|\partial_t \widetilde{u}|\le 0.
\end{align*}
Thus maximum principle and (\ref{m8}) give
\begin{align}\label{10.26}
\|{{\partial _t}\widetilde{u}(s,t,x)} \|_{L^2_x}\lesssim e^{-\frac{1}{4}s}\|\partial_tu\|_{L^2_x}\le M.
\end{align}
Therefore, (\ref{xiu1}) and (\ref{xiu2}) follow  from (\ref{sdf8uj}) and (\ref{10.26}) respectively.
Second, we prove (\ref{xiu3}) by energy arguments. Introduce the energy functionals
$$
\mathcal{E}_4(\widetilde{u})=\frac{1}{2}\int_{\Bbb H^2}|\partial_t\widetilde{u}|^2{\rm{dvol_h}},\mbox{  }\mathcal{E}_5(\widetilde{u})=\int_{\Bbb H^2}|\nabla\partial_t\widetilde{u}|^2{\rm{dvol_h}}.
$$
Then integration by parts gives
\begin{align}\label{muxc6zb8}
\frac{d}{{ds}}{\mathcal{E}_4}\left( \widetilde{u} \right) + {\mathcal{E}_5}\left( \widetilde{u} \right) \le \int_{\Bbb H^2}{\left| {d\widetilde{u}} \right|^2}{\left| {{\partial _t}\widetilde{u}} \right|^2}{\rm{dvol_h}}.
\end{align}
Integrating this formula with respect to $s$ in $[0.\kappa)$ with $\kappa>1$ shows
$$\int_0^{\kappa} {\left\| {\nabla \partial {}_t\widetilde{u}} \right\|_{L_x^2}^2ds}  \le \left\| {\partial_t\widetilde{u}(\kappa)} \right\|_{L_x^2}^2 + \int_0^1 {\left\| {{\partial _t}\widetilde{u}} \right\|_{{L^4}}^2\left\| {d\widetilde{u}} \right\|_{{L^4}}^2} ds + {\mathcal{E}_1}\left( \widetilde{u} \right)M\int_1^{\kappa} {{e^{ - 2\delta s}}ds},
$$
where we have used (\ref{xiu1}), (\ref{xiu2}) and H\"older. By Sobolev embedding and letting $\kappa\to\infty$, we obtain
\begin{align}\label{chuyunyi}
\int_0^\infty  {\left\| {\nabla \partial_t\widetilde{u}} \right\|_{L_x^2}^2ds}  \le {M^4} + {M^2}.
\end{align}
Finally, the proof of (\ref{xiu4}) follows by the same arguments as  (\ref{koo4}) illustrated in Lemma \ref{decayingt}.
\end{proof}

\begin{lemma}\label{zhangqiling}
Let $(u,\partial_tu)$ be a solution to (\ref{wmap1}) in $\mathcal{X}_T$ with $\|u(t,x)\|_{\mathcal{X}_T}\le M$. Then
\begin{align}
 {\left\| {{s^{\frac{1}{2}}}{\nabla _t}{\partial _s}\widetilde{u}} \right\|_{L_s^\infty L_x^2}} &\lesssim MC(M)\mbox{  }{\rm{for}}\mbox{  }s\in[0,1]\label{haorenhao9}.
\end{align}
Moreover, for $s\in[1,\infty)$ and some $0<\delta\ll1$
\begin{align}
{\left\| {{\nabla _t}{\partial _s}\widetilde{u}} \right\|_{L_s^\infty L_x^\infty }} &\lesssim e^{-\delta s}MC(M).\label{haorenhao10}
\end{align}
\end{lemma}
\begin{proof}
It is easy to see $\left| {{\nabla _t}{\partial _s}\widetilde{u}} \right| \le \left| {\nabla {\partial _t}\widetilde{u}} \right|+ \left| {{h^{ii}}{\nabla _i}{\nabla _t}{\partial _i}\widetilde{u}} \right| + \left| {{\partial _t}\widetilde{u}} \right|{\left| {d\widetilde{u}} \right|^2}$,
then
\begin{align}\label{facik}
\left| {{\nabla _t}{\partial _s}\widetilde{u}} \right| \le \left| {{\nabla ^2}{\partial _t}\widetilde{u}} \right| + \left| {{\partial _t}\widetilde{u}} \right|{\left| {d\widetilde{u}} \right|^2} + \left| {\nabla {\partial _t}\widetilde{u}} \right|.
\end{align}
Integration by parts gives
\begin{align}
 &\frac{d}{{ds}}\left\| {\nabla {\partial _t}\widetilde{u}} \right\|_{{L^2}}^2\nonumber\\
 &\le  - {\int_{{\Bbb H^2}} {\left| {{\nabla ^2}{\partial _t}\widetilde{u}} \right|} ^2}{\rm{dvol_h}} + \int_{{\Bbb H^2}} {\left| {\nabla {\partial _t}\widetilde{u}} \right|\left| {{\partial _t}\widetilde{u}} \right|} \left| {d\widetilde{u}} \right|\left| {{\partial _s}\widetilde{u}} \right|{\rm{dvol_h}} \nonumber\\
 &+ \int_{{\Bbb H^2}} {\left| {{\nabla ^2}{\partial _t}\widetilde{u}} \right|} {\left| {d\widetilde{u}} \right|^2}\left| {{\partial _s}\widetilde{u}} \right|{\rm{dvol_h}} + {\int_{{\Bbb H^2}} {\left| {\nabla {\partial _t}\widetilde{u}} \right|} ^2}{\rm{dvol_h}}+ {\int_{{\Bbb H^2}} {\left| {\nabla {\partial _t}\widetilde{u}} \right|} ^2}{\left| {d\widetilde{u}} \right|^2}{\rm{dvol_h}}.\label{haoren6}
\end{align}
By Sobolev embedding, we obtain
$$\frac{d}{{ds}}\left\| {\nabla {\partial _t}\widetilde{u}} \right\|_{{L^2}}^2 \le C\left\| {\nabla {\partial _t}\widetilde{u}} \right\|_{{L^2}}^2\left( {1 + \left\| {{\partial _s}\widetilde{u}} \right\|_{{L^\infty }}^2 + \left\| {d\widetilde{u}} \right\|_{{L^\infty }}^2} \right) + \left\| {\nabla d\widetilde{u}} \right\|_{{L^2}}^4\left\| {{\partial _s}\widetilde{u}} \right\|_{{L^\infty }}^2.
$$
Thus we get
$$\left\| {\nabla {\partial _t}\widetilde{u}} \right\|_{{L^2}}^2 \le \left\| {\nabla {\partial _t}\widetilde{u}(0,t,x)} \right\|_{{L^2}}^2 + {e^{\int^s_0V(\tau)d\tau}}\int_0^s {{e^{ - \int_0^{\kappa} {V(\tau )} d\tau }}} \left\| {\nabla d\widetilde{u}} \right\|_{{L^2}}^4\left\| {{\partial _s}\widetilde{u}} \right\|_{{L^\infty }}^2d\kappa,
$$
where $V(s)=Cs+C\|d\widetilde{u}\|^2_{L^{\infty}}+C\|\partial_s\widetilde{u}\|^2_{L^{\infty}}$.
By Lemma \ref{ktao1} and Lemma \ref{8.44}
\begin{align}\label{haorenh7}
\int^{1}_0\|d\widetilde{u}\|^2_{L^{\infty}}ds+\int^1_0\|\partial_s\widetilde{u}\|^2_{L^{\infty}}ds\le M^2.
\end{align}
Hence we conclude for $s\in[0,1]$,
\begin{align}\label{jidujihao1}
\left\| {\nabla {\partial _t}\widetilde{u}} \right\|_{{L^2}}^2 \le \left\| {\nabla {\partial _t} \widetilde{u}(0,t,x)} \right\|_{{L^2}}^2 + {e^{MC(M)s}}MC(M).
\end{align}
With (\ref{haoren6}), we further deduce that
\begin{align}\label{jidujihao}
\int_0^1 {{{\left\| {{\nabla ^2}{\partial _t}\widetilde{u}} \right\|}_{{L^2}}}ds}  \lesssim MC(M).
\end{align}
Integration by parts shows,
\begin{align*}
 &\frac{d}{{ds}}\left( {\left\| {{\nabla ^2}{\partial _t}\widetilde{u}} \right\|_{L_x^2}^2s} \right) \\
 &\le  - s{\int_{{\Bbb H^2}} {\left| {{\nabla ^3}{\partial _t}\widetilde{u}} \right|} ^2}{\rm{dvol_hdt}} + \left\| {{\nabla ^2}{\partial _t}\widetilde{u}} \right\|_{L_x^2}^2 + s{\left\| {{\partial _s}\widetilde{u}} \right\|_{L_x^\infty }}{\left\| {\nabla {\partial _t}\widetilde{u}} \right\|_{L_x^4}}{\left\| {d\widetilde{u}} \right\|_{L_x^4}}{\left\| {{\nabla ^2}{\partial _t}\widetilde{u}} \right\|_{L_x^2}} \\
 &+ s{\left\| {{\partial _t}\widetilde{u}} \right\|_{L_x^\infty }}{\left\| {\nabla {\partial _s}\widetilde{u}} \right\|_{L_x^2}}{\left\| {d\widetilde{u}} \right\|_{L_x^\infty }}{\left\| {{\nabla ^2}{\partial _t}\widetilde{u}} \right\|_{L_x^2}} + s\left\| {d\widetilde{u}} \right\|_{{L^8}}^2{\left\| {\nabla {\partial _t}\widetilde{u}} \right\|_{L_x^4}}{\left\| {{\nabla ^3}{\partial _t}\widetilde{u}} \right\|_{L_x^2}}  \\
 &+ s{\left\| {d\widetilde{u}} \right\|_{L_x^\infty }}{\left\| {\nabla {\partial _t}\widetilde{u}} \right\|_{L_x^2}}{\left\| {{\partial _s}\widetilde{u}} \right\|_{L_x^\infty }}{\left\| {{\nabla ^2}{\partial _t}\widetilde{u}} \right\|_{L_x^2}} + s\left\| {d\widetilde{u}} \right\|_{{L^\infty }}^2{\left\| {\nabla {\partial _t}\widetilde{u}} \right\|_{L_x^2}}{\left\| {{\nabla ^2}{\partial _t}\widetilde{u}} \right\|_{L_x^2}} \\
 &+ s\left\| {d\widetilde{u}} \right\|_{{L^\infty }}^2{\left\| {\nabla d\widetilde{u}} \right\|_{L_x^2}}{\left\| {{\nabla ^2}{\partial _t}\widetilde{u}} \right\|_{L_x^2}} + s{\left\| {{\partial _t}\widetilde{u}} \right\|_{L_x^\infty }}{\left\| {\nabla d\widetilde{u}} \right\|_{L_x^2}}{\left\| {{\partial _s}\widetilde{u}} \right\|_{L_x^\infty }}{\left\| {{\nabla ^2}{\partial _t}\widetilde{u}} \right\|_{L_x^2}}\\
 &+ s{\left\| {\nabla d\widetilde{u}} \right\|_{L_x^2}}{\left\| {{\partial _t}\widetilde{u}} \right\|_{L_x^\infty }}{\left\| {{\nabla ^3}{\partial _t}\widetilde{u}} \right\|_{L_x^2}}{\left\| {d\widetilde{u}} \right\|_{L_x^\infty }} + s\left\| {d\widetilde{u}} \right\|_{{L^\infty }}^2\left\| {{\nabla ^2}{\partial _t}\widetilde{u}} \right\|_{{L^2}}^2.
 \end{align*}
Then Gronwall with (\ref{V5}), (\ref{haorenh7}) yields for all $s\in[0,1]$
\begin{align}\label{haoren8}
\left\| {{\nabla ^2}{\partial _t}\widetilde{u}} \right\|_{L_x^2}^2s + \int_0^s {{\int_{{\Bbb H^2}} {\left| {{\nabla ^3}{\partial _t}\widetilde{u}} \right|}^2}\tau} {\rm{dvol_h}}d\tau \le MC(M).
\end{align}
Thus by (\ref{haoren8}), (\ref{jidujihao}), (\ref{jidujihao1}), and (\ref{facik}), we conclude
\begin{align}
{\left\| {{s^{\frac{1}{2}}}{\nabla _t}{\partial _s}\widetilde{u}} \right\|_{L_s^\infty[0,1] L_x^2}} \le MC(M).
\end{align}
(\ref{haorenhao10}) follows by the same path as Lemma \ref{decayingt} with the help of (\ref{iconm}). The essential ingredient is to prove for $s_1\ge2$
\begin{align}\label{huaqian}
\int^{s_1+1}_{s_1}\|\nabla^2\partial_t\widetilde{u}\|^2_{L^2_x}ds\lesssim MC(M)e^{-\delta s_1}.
\end{align}
The remaining proof is devoted to verifying (\ref{huaqian}). By (\ref{{uv111}}) and  (\ref{haoren6}), we obtain for any $0<c\ll1$
\begin{align}
 &\frac{d}{{ds}}\left\| {\nabla {\partial _t}\widetilde{u}} \right\|_{{L^2}}^2 +c {\int_{{\Bbb H^2}} {\left| {{\nabla}{\partial _t}\widetilde{u}} \right|} ^2}{\rm{dvol_h}}+c {\int_{{\Bbb H^2}} {\left| {{\nabla^2}{\partial _t}\widetilde{u}} \right|} ^2}{\rm{dvol_h}}\nonumber\\
 &\lesssim \int_{{\Bbb H^2}} {\left| {\nabla {\partial _t}\widetilde{u}} \right|\left| {{\partial _t}\widetilde{u}} \right|} \left| {d\widetilde{u}} \right|\left| {{\partial _s}\widetilde{u}} \right|{\rm{dvol_h}}+ \int_{{\Bbb H^2}} {\left| {{\nabla ^2}{\partial _t}\widetilde{u}} \right|} {\left| {d\widetilde{u}} \right|^2}\left| {{\partial _s}\widetilde{u}} \right|{\rm{dvol_h}}\nonumber\\
 &+ {\int_{{\Bbb H^2}} {\left| {\nabla {\partial _t}\widetilde{u}} \right|} ^2}{\left| {d\widetilde{u}} \right|^2}{\rm{dvol_h}} +\frac{1}{c} {\int_{{\Bbb H^2}} {\left| {\nabla {\partial _t}\widetilde{u}} \right|} ^2}{\rm{dvol_h}}.\label{haoren68}
\end{align}
By Lemma 3.3 and (\ref{lao2}), we have for $s\ge1$
\begin{align}
&\left\| {d\widetilde{u}} \right\|_{L_x^{\infty}}\lesssim \|du\|_{L^2_x}\lesssim M\label{g8uiknlo}\\
&\left\| {{\partial _s}\widetilde{u}} \right\|_{L_x^{\infty}}\lesssim e^{-\delta s}\|\partial_s\widetilde{u}(0,t)\|_{L^2_x}\lesssim e^{-\delta s}M.\label{g8uiknlp}
\end{align}
Then by Sobolev embedding and Gronwall inequality, for $s\ge1$
\begin{align}
\left\| {\nabla {\partial _t}\widetilde{u}} \right\|_{L_x^2}^2 &\lesssim {e^{-cs}}\left\| {\nabla {\partial _t}\widetilde{u}(1,t,x)} \right\|_{L_x^2}^2
+ {e^{ - cs}}\int_1^s {{e^{c\kappa}}} \left\| {\nabla {\partial _t}\widetilde{u}(\kappa)} \right\|_{L_x^2}^2d\kappa\nonumber\\
&+ {e^{ - cs}}\int_1^s {{e^{c\kappa}}} \left\| {\nabla d\widetilde{u}(\kappa,t)} \right\|_{L_x^2}^4\left\| {{\partial _s}\widetilde{u}(\kappa,t)} \right\|_{L_x^\infty }^2d\kappa.\label{ua5vvz}
\end{align}
Hence (\ref{chuyunyi}), (\ref{V5}), (\ref{g8uiknlp}) and (\ref{ua5vvz}) give for $s\in[0,\infty)$
\begin{align}\label{fanlin}
\left\| {\nabla {\partial _t}\widetilde{u}} \right\|_{L_x^2}^2 \le MC(M),
\end{align}
where (\ref{fanlin}) when $s\in[0,1]$ follows by (\ref{jidujihao1}).
Integrating (\ref{muxc6zb8}) with respect to $s$ in $[s_1,s_2]$ for $1\le s_1\le s_2<\infty$ yields
\begin{align}\label{chanjiang1}
\int_{{s_1}}^{{s_2}} {\left\| {\nabla {\partial _t}\widetilde{u}} \right\|_{L_x^2}^2} ds \le \left\| {{\partial _t}\widetilde{u}} \right\|_{L_x^2}^2({s_2}) - \left\| {{\partial _t}\widetilde{u}} \right\|_{L_x^2}^2({s_1}) + \int_{{s_1}}^{{s_2}} {\left\| {{\partial _t}\widetilde{u}} \right\|_{L_x^2}^2} \left\| {d\widetilde{u}} \right\|_{L_x^\infty }^2ds.
\end{align}
By (\ref{xiu1}), Lemma \ref{density},
\begin{align}\label{chanjiang2}
\int_{{s_1}}^{{s_2}} {\left\| {\nabla {\partial _t}\widetilde{u}} \right\|_{L_x^2}^2} ds \lesssim MC(M){e^{ - \delta {s_1}}}.
\end{align}
Thus in any interval $[s_*,s_*+1]$ there exists $s^0_*\in[s_*,s_*+1]$ such that
\begin{align}\label{chanjiang21}
{\left\| {\nabla {\partial _t}\widetilde{u}}(s^0_*) \right\|_{L_x^2}^2} ds \lesssim MC(M){e^{ - \delta {s_*}}}.
\end{align}
Fix $s_*\ge 1$, applying Gronwall to (\ref{haoren68}) in $[s^0_*,a]$ with $a\in[s^0_*,s_*+2]$ gives
\begin{align}
\left\| {\nabla {\partial _t}\widetilde{u}} \right\|_{L_x^2}^2(a,t) &\le {e^{ - ca}}\left\| {\nabla {\partial _t}\widetilde{u}(s_*^0,t,x)} \right\|_{L_x^2}^2+ {e^{ - ca}}\int_{s^0_*}^a {{e^{cs}}}\left\| {\nabla {\partial _t}\widetilde{u}(s)} \right\|_{L_x^2}^2ds\nonumber\\
&+ {e^{ - ca}}\int_{s^0_*}^a {{e^{cs}}} \left\| {\nabla d\widetilde{u}(s)} \right\|_{L_x^2}^4\left\| {{\partial _s}\widetilde{u}(s)} \right\|_{L_x^\infty }^2ds.\label{ua6vvz}
\end{align}
Thus by  (\ref{xiu1}), Lemma \ref{density}, (\ref{chanjiang21}) and the fact that $a$ at leat ranges over all $[s_*+1,s_*+2]$, we have for $s\ge2$,
\begin{align}\label{chanjiang3}
\left\| {\nabla {\partial _t}\widetilde{u}} \right\|_{L_x^2}^2\le MC(M) e^{-\delta s}.
\end{align}
Integrating (\ref{haoren68}) with respect to $s$ again in $[s_1,s_1+1]$, we obtain (\ref{huaqian}) by (\ref{chanjiang3}) and (\ref{xiu1}), Lemma 3.3. Finally using maximum principle and Remark \ref{ki78} as Lemma  \ref{decayingt}, we get (\ref{haorenhao10}) from (\ref{ua5vvz}),
(\ref{chanjiang3}), (\ref{facik}) and Lemma \ref{decayingt}.
\end{proof}

In the remaining part of this subsection, we consider the short time behaviors of the differential fields under the heat flow.
Since the energy of the solution to the heat flow in our case will not decay to zero, we can not expect that it behaves as a solution to the linear heat equation in the large time scale. However, one can still expect that the solution to the heat flow is almost governed by the linear equation in the short time scale. We summarize these useful estimates in the following proposition.
\begin{proposition}\label{lize1}
Let $u:[0,T]\times\Bbb H^2\to \Bbb H^2$ be a solution to (\ref{wmap1}) satisfying
\begin{align*}
\|(\nabla du,\nabla\partial_t u)\|_{L^2\times L^2}+\|(du,\partial_t u)\|_{L^2\times L^2}\le M.
\end{align*}
If $\widetilde{u}:\Bbb R^+\times[0,T]\times\Bbb H^2\to \Bbb H^2$ is the solution to (\ref{8.29.2}) with initial data $u(t,x)$, then for any $\eta>0$, it holds uniformly for $(s,t)\in(0,1)\times[0,T]$ that
\begin{align*}
&s^{\frac{1}{2}}{\left\| {\nabla d\widetilde{u}} \right\|_{{L^\infty_x }}} + s^{\frac{1}{2}}{\left\| {\nabla {\partial _t}\widetilde{u}} \right\|_{{L^\infty_x }}} + {s}{\left\| {\nabla {\partial _s}\widetilde{u}} \right\|_{{L^\infty_x }}} + s^{\frac{1}{2}}{\left\| {{\partial _s}\widetilde{u}} \right\|_{{L^\infty_x }}}\\
&+ s^{\frac{1}{2}}{\left\| {\nabla {\partial _s}\widetilde{u}} \right\|_{{L^2_x}}} + {\left\| {s^{\frac{1}{2}}{\nabla ^2}{\partial _s}\widetilde{u}} \right\|_{L_{s,x}^2}}
+s{\left\| {\nabla_t {\partial _s}\widetilde{u}} \right\|_{{L^{\infty}_x}}}+s^{\eta}\|d\widetilde{u}\|_{L^{\infty}_x}\le MC(M).
\end{align*}
\end{proposition}
\begin{proof}
Since $\|\nabla d\widetilde{u}\|_{L^2_x}\le M$ shown by (\ref{V5}), Sobolev embedding implies $\|d\widetilde{u}\|_{L^p_x}\le M$ for any $p\in(2,\infty)$. Then (\ref{uu}) and (\ref{huhu89}) yield $s^{\eta}\|d\widetilde{u}\|_{L^{\infty}_x}\le M$ for any $\eta>0$ and all $(t,s)\in[0,T]\times(0,1)$.
By (\ref{ion1}), one has
$${\partial _s}\left| {\nabla d\widetilde{u}} \right| - {\Delta}\left| {\nabla d\widetilde{u}} \right|{\rm{ }} \le K\left| {\nabla d\widetilde{u}} \right| + {\left| {d\widetilde{u}} \right|^2}\left| {\nabla d\widetilde{u}} \right| + {\left| {d\widetilde{u}} \right|} + {\left| {d\widetilde{u}} \right|^4}.
$$
Furthermore we obtain
$$\left( {\partial _s} - {\Delta} \right)\left( {{e^{ - sK}}{e^{ - \int_0^s {\left\| {d\widetilde{u}(\tau )} \right\|_{L_x^\infty }^2d\tau } }}\left| {\nabla d\widetilde{u}} \right|} \right) \le {e^{ - sK}}{e^{ - \int_0^s {\left\| {d\widetilde{u}(\tau )} \right\|_{L_x^\infty }^2d\tau } }}\left( {{{\left| {d\widetilde{u}} \right|}} + {{\left| {d\widetilde{u}} \right|}^4}} \right).
$$
Then maximum principle implies for $s\in[0,2]$
\begin{align*}
{\left\| { {\nabla d\widetilde{u}}(s)} \right\|_{L_x^\infty }} &\lesssim {\big\| {{e^{{\Delta}\frac{s}{2}}}}( {{e^{ - \frac{{sK}}{2}}}{e^{ - \int_0^{\frac{s}{2}} {\left\| {d\widetilde{u}(\tau )} \right\|_{L_x^\infty }^2d\tau } }}\left| {\nabla d\widetilde{u}} \right|(\frac{s}{2})})\big\|_{L_x^\infty }} \\
&+ {\big\| {\int_{\frac{s}{2}}^s {{e^{{\Delta}(s - \tau )}} {{e^{ - K\tau }}{e^{ - \int_0^\tau  {\| {d\widetilde{u}(\tau_1)} \|_{L_x^\infty }^2d\tau_1 } }}( {{{| {d\widetilde{u}}|}}+ {{| {d\widetilde{u}} |}^4}} )(\tau )} d\tau } }\big \|_{L_x^\infty }}.
\end{align*}
By the smoothing effect of the heat semigroup, we obtain for $s\in[0,1]$
\begin{align*}
{\left\| { {\nabla d\widetilde{u}}(s)} \right\|_{L_x^\infty }} \lesssim {s^{ - \frac{1}{2}}}{\big\| { {\nabla d\widetilde{u}} (\frac{s}{2})} \big\|_{L_x^2}} + {\int_{\frac{s}{2}}^s {{{\big\| {{{| {d\widetilde{u}}|}^4}(\tau )} \big\|}_{L_x^\infty }} + \big\| {{{| {d\widetilde{u}}|}}|(\tau )} \big\|} _{L_x^\infty }}d\tau.
\end{align*}
Then Lemma \ref{8.5} and Lemma \ref{8.44} show for $s\in(0,1)$
$${\left\| \nabla d\widetilde{u}(s) \right\|_{L_x^\infty }} \le {s^{ - \frac{1}{2}}}{\big\| { {\nabla d\widetilde{u}}(\frac{s}{2})} \big\|_{L_x^2}} + \int_{\frac{s}{2}}^s {{\tau ^{ - 3/2}}( \big\| {du} \big\|^4_{L_x^{\frac{8}{3}}} + \big\| {du} \big\|_{L_x^2})} d\tau.
$$
Therefore by Sobolev inequality we conclude
\begin{align*}
{\left\| {\left| {\nabla d\widetilde{u}} \right|(s)} \right\|_{L_x^\infty }} &\le  {s^{ -\frac{1}{2}}}\mathop {\sup }\limits_{t \in [0,T]} \left( \left\| {\nabla du}(t) \right\|_{L_x^2}^4 + \left\| {du}(t) \right\|_{L_x^2} \right)\nonumber\\
&+{s^{ - \frac{1}{2}}}\mathop {\sup }\limits_{s \in [0,1]} {\left\| {\left| {\nabla d\widetilde{u}} \right|(s)} \right\|_{L_x^2}}.
\end{align*}
Thus by (\ref{V5}), we obtain for $s\in[0,1]$
\begin{align}\label{qian1}
s^{\frac{1}{2}}{\left\| {\left| {\nabla d\widetilde{u}} \right|(s)} \right\|_{L_x^\infty }} \le MC(M).
\end{align}
By (\ref{9xian}) we have
\begin{align}
&\frac{d}{{ds}}\left( {s{\mathcal{E}_3}(\widetilde{u}(s))} \right) \nonumber\\
&\lesssim {\mathcal{E}_3}(\widetilde{u}(s)) - \int_{{\Bbb H^2}} {s{{\left| {{\nabla ^2}{\partial _s}\widetilde{u}} \right|}^2}{\rm{dvol_h}}}+ \int_{\Bbb H^2}s\left( \left| {d\widetilde{u}} \right|\left| {\nabla {\partial _s}\widetilde{u}} \right|{{\left| {{\partial _s}\widetilde{u}} \right|}^2}  \right) {\rm{dvol_h}}\nonumber \\
&+ \int_{{\Bbb H^2}} s\left( \left| {{\nabla d}\widetilde{u}} \right|\left| {d\widetilde{u}} \right|\left| {{\partial _s}\widetilde{u}} \right|\left| {\nabla {\partial _s}\widetilde{u}} \right| + {{\left| {{\partial _s}\widetilde{u}} \right|}^2}{{\left| {d\widetilde{u}} \right|}^4} + {{\left| {\nabla {\partial _s}\widetilde{u}} \right|}^2}{{\left| {d\widetilde{u}} \right|}^2}  \right){\rm{dvol_h}} \nonumber\\
&+\int_{{\Bbb H^2}}s\big( {{\left| {d\widetilde{u}} \right|}^3}\left| {{\partial _s}\widetilde{u}} \right|\left| {\nabla {\partial _s}\widetilde{u}} \right|+ {{\left| {d\widetilde{u}} \right|}^2}\left| {{\nabla ^2}{\partial _s}\widetilde{u}} \right|\left| {{\partial _s}\widetilde{u}} \right|\big){\rm{dvol_h}}.\label{hu897}
\end{align}
The terms in the right hand side can be bounded by Sobolev and H\"older as follows
\begin{align*}
\int_{{\Bbb H^2}} {s\left| {d\widetilde{u}} \right|} \left| {\nabla {\partial _s}\widetilde{u}} \right|{\left| {{\partial _s}\widetilde{u}} \right|^2}{\rm{dvol_h}}&\le {s^{\frac{1}{2}}}{\left\| {d\widetilde{u}} \right\|_{L_x^\infty }}{\left\| {\nabla {\partial _s}\widetilde{u}} \right\|_{L_x^2}}{s^{\frac{1}{2}}}\left\| {{\partial _s}\widetilde{u}} \right\|_{L_x^4}^2 \\
\int_{{\Bbb H^2}} {s{{\left| {d\widetilde{u}} \right|}^3}\left| {{\partial _s}\widetilde{u}} \right|\left| {\nabla {\partial _s}\widetilde{u}} \right|} {\rm{dvol_h}} &\le s\left\| {d\widetilde{u}} \right\|_{L_x^{12}}^3{\left\| {\nabla {\partial _s}\widetilde{u}} \right\|_{L_x^2}}{\left\| {{\partial _s}\widetilde{u}} \right\|_{L_x^4}} \\
\int_{{\Bbb H^2}} s\left| {{\nabla d}\widetilde{u}} \right|\left| {d\widetilde{u}} \right|\left| {{\partial _s}\widetilde{u}} \right|\left| {\nabla {\partial _s}\widetilde{u}} \right|{\rm{dvol_h}} &\le {{\left\| {s{\nabla d}\widetilde{u}} \right\|}_{L_x^\infty }}{{\left\| {\nabla {\partial _s}\widetilde{u}} \right\|}_{L_x^2}}{{\left\| {d\widetilde{u}} \right\|}_{L_x^4}}{{\left\| {{\partial _s}\widetilde{u}} \right\|}_{L_x^4}}  \\
\int_{{\Bbb H^2}} s{{\left| {{\partial _s}\widetilde{u}} \right|}^2}{{\left| {d\widetilde{u}} \right|}^4}{\rm{dvol_h}} &\le \left\| {{\partial _s}\widetilde{u}} \right\|_{L_x^2}^2s\left\| {d\widetilde{u}} \right\|_{L_x^{\infty}}^4  \\
\int_{{\Bbb H^2}} s{{\left| {\nabla {\partial _s}\widetilde{u}} \right|}^2}{{\left| {d\widetilde{u}} \right|}^2}{\rm{dvol_h}} &\le \left\| {\nabla {\partial _s}\widetilde{u}} \right\|_{L_x^2}^2s\left\| {d\widetilde{u}} \right\|_{L_x^\infty }^2.
\end{align*}
The highest order term can be absorbed by the negative term, indeed we have
\begin{align*}
 \int_{{\Bbb H^2}} s{{\left| {d\widetilde{u}} \right|}^2}\left| {{\nabla ^2}{\partial _s}\widetilde{u}} \right|\left| {{\partial _s}\widetilde{u}} \right|{\rm{dvol_h}} &\le  \frac{s}{{2C}}\int_{\Bbb H^2} {{{\left| {{\nabla ^2}{\partial _s}\widetilde{u}} \right|}^2}} {\rm{dvol_h} + C\int_{{\Bbb H^2}} {s{{\left| {{\partial _s}\widetilde{u}} \right|}^2}{{\left| {d\widetilde{u}} \right|}^4}{\rm{dvol_h}}}}  \\
 &\le \frac{s}{{2C}}\int_{{\Bbb H^2}} {{{\left| {{\nabla ^2}{\partial _s}\widetilde{u}} \right|}^2}} {\rm{dvol_h}} + C\left\| {{\partial _s}\widetilde{u}} \right\|_{L_x^2}^2s\left\| {d\widetilde{u}} \right\|_{L_x^{\infty}}^4.
 \end{align*}
Recall the fact $\left| {d\widetilde{u}} \right|(s) \le e^{{\Delta s}}\left| {d{u}} \right|$ when $s\in[0,1]$, $\left| {{\partial _s}\widetilde{u}} \right|(s) \le {e^{{\Delta s}}}\left| {\tau({u})} \right|$,
the terms involved above are bounded by smoothing effect
\begin{align}
 s\left\| {d\widetilde{u}} \right\|_{L_x^\infty }^2 + {s^{\frac{1}{4}}}{\left\| {{\partial _s}\widetilde{u}} \right\|_{L_x^4}} \le \left\| {d{u}} \right\|_{L_x^2}^2 + {\left\| {\tau({u})} \right\|_{L_x^2}}.
\end{align}
Thus integrating (\ref{hu897}) with respect to $s$ in $[0,s]$ with (\ref{qian1}) gives for $s\in[0,1]$
\begin{align}\label{ki6ll1}
s{\mathcal{E}_3}(\widetilde{u}(s)) + \int_0^s {\int_{\Bbb H^2}} {s{{\left| {{\nabla ^2}{\partial _s}\widetilde{u}} \right|}^2}} {\rm{dvol_h}} ds \lesssim \int_0^s {\left\| {\nabla {\partial _s}\widetilde{u}} \right\|_{L_x^2}^2ds'}.
\end{align}
Therefore by (\ref{f40}), we conclude for $s\in[0,1]$
\begin{align}\label{qian2}
\int_{{\Bbb H^2}} s{{\left| {\nabla {\partial _s}\widetilde{u}} \right|}^2}{\rm{dvo}}{{\rm{l}}_{\rm{h}}} \le MC(M).
\end{align}
By (\ref{9tian1}), we deduce
$$
{\partial _s}\left| {\nabla {\partial _s}\widetilde{u}} \right| - {\Delta}\left| {\nabla {\partial _s}\widetilde{u}} \right|\le \left| {\nabla {\partial _s}\widetilde{u}} \right|{\left| {d\widetilde{u}} \right|^2} + \left| {{\partial _s}\widetilde{u}} \right|{\left| {d\widetilde{u}} \right|^3} + \left| {{\partial _s}\widetilde{u}} \right|\left| {{\nabla d}\widetilde{u}} \right|\left| {d \widetilde{u}} \right|.
$$
Then as above considering the equation of ${e^{ - \int_0^s {\| {d\widetilde{u}(\tau )}\|_{L_x^\infty }^2d\tau } }}\left| {\nabla {\partial _s}\widetilde{u}} \right|$, we obtain by maximum principle that
\begin{align*}
&{\| { {\nabla {\partial _s}\widetilde{u}}(s)} \|_{L_x^\infty }}\\
&\le {s^{ - \frac{1}{2}}}{\| {{\nabla {\partial _s}\widetilde{u}}(\frac{s}{2})} \|_{L_x^2}} + {\int_{\frac{s}{2}}^s {{{\| {\left| {{\partial _s}\widetilde{u}} \right|{{\left| {d\widetilde{u}} \right|}^3}(\tau )} \|}_{L_x^\infty }} + \|{\left| {{\partial _s}\widetilde{u}} \right|\left| {{\nabla d}\widetilde{u}} \right|\left| {d\widetilde{u}} \right|(\tau )} \|} _{L_x^\infty }}d\tau.
\end{align*}
Hence (\ref{qian1}) and (\ref{qian2}) give
\begin{align*}
{\left\| {\left| {\nabla {\partial _s}\widetilde{u}} \right|(s)} \right\|_{L_x^\infty }} \le& {s^{ - 1}}M + \left( {\mathop {\sup }\limits_{s \in [0,1]} s{{\left\| {{\partial _s}\widetilde{u}} \right\|}_{L_x^\infty }}{{\left\| {d\widetilde{u}} \right\|}_{L_x^\infty }}} \right)\int_{\frac{s}{2}}^s {{\tau ^{ - 1}}\left\| {d\widetilde{u}} \right\|_{L_x^\infty }^2} d\tau\\
&+ \left( {\mathop {\sup }\limits_{s \in [0,1]} s{{\left\| {\nabla d\widetilde{u}} \right\|}_{L_x^\infty }}{{\left\| {d\widetilde{u}} \right\|}_{L_x^\infty }}} \right)\int_{\frac{s}{2}}^s {{\tau ^{ - 1}}{{\left\| {d\widetilde{u}} \right\|}_{L_x^\infty }}} d\tau\\
&\le {s^{ - 1}}M + {s^{ - 1}}M^2\int_{\frac{s}{2}}^s {\left( {\left\| {d\widetilde{u}} \right\|_{L_x^\infty }^2 + {{\left\| {d\widetilde{u}} \right\|}_{L_x^\infty }}} \right)} d\tau.
\end{align*}
Consequently, we have by Lemma \ref{ktao1},
\begin{align}\label{qian3}
{\left\| {\left| {\nabla \partial_s \widetilde{u}} \right|(s)} \right\|_{L_x^\infty }} \le MC(M)s^{- 1}.
\end{align}
By Lemma \ref{8zu}, one deduces
\begin{align*}
{\partial _s}\left| {\nabla {\partial _t}\widetilde{u}} \right| - {\Delta}\left| {\nabla {\partial _t}\widetilde{u}} \right| &\le K\left| {\nabla {\partial _t}\widetilde{u}} \right|+\left| {\nabla {\partial _t}\widetilde{u}} \right||d\widetilde{u}|^2+|\partial_s\widetilde{u}||d\widetilde{u}|^2\\
&+ {\left| {d\widetilde{u}} \right|^3}\left| {{\partial _t}\widetilde{u}} \right| + \left| {d\widetilde{u}} \right|\left| {{\partial _t}\widetilde{u}} \right|\left| {\nabla d\widetilde{u}} \right|.
\end{align*}
Considering the equation of ${e^{ - \int_0^s \big({\left\| {d\widetilde{u}(\tau )} \right\|_{L_x^\infty }^2-K\big)d\tau } }}\left| {\nabla {\partial _t}\widetilde{u}} \right|$, we have by maximum principle that
\begin{align*}
\left\| {\nabla {\partial _t}\widetilde{u}} \right\|_{L^{\infty}_x} &\le {s^{ - \frac{1}{2}}}{\| {\nabla {\partial _t}\widetilde{u}} \|_{L_x^2}} +\int_{\frac{s}{2}}^s {{{(s - \tau )}^{ - \frac{1}{2}}}} \| {{\left| {d\widetilde{u}} \right|}^3}\left| {{\partial _t}\widetilde{u}} \right|(\tau )\|_{{L^2_x}}d\tau\\
&+ \int_{\frac{s}{2}}^s\||d\widetilde{u}|\left| {{\partial _t}\widetilde{u}} \right|\left| {\nabla d\widetilde{u}} \right|(\tau )\|_{L^{\infty}_x}+\||\partial_s\widetilde{u}||d\widetilde{u}|^2\|_{L^{\infty}_x}d\tau  \\
&\le {s^{ - \frac{1}{2}}}{\left\| {\nabla \partial_t\widetilde{u}} \right\|_{L_x^2}} + \mathop {\sup }\limits_{s \in [0,1]} \left( {{{\left\| {{\partial _t}\widetilde{u}} \right\|}_{L_x^4}}\left\| {d\widetilde{u}} \right\|_{L_x^{12}}^3} \right)\int_{\frac{s}{2}}^s {{{(s - \tau )}^{ - \frac{1}{2}}}d\tau }  \\
&+ \mathop {\sup }\limits_{s \in [0,1]} \left( {{s^{\frac{1}{2}}}{{\left\| {\nabla d\widetilde{u}} \right\|}_{L_x^\infty }}} \right)\int_{\frac{s}{2}}^s {{\tau ^{ - \frac{1}{2}}}{{\left\| {{\partial _t}\widetilde{u}} \right\|}_{L_x^\infty }}{{\left\| {d\widetilde{u}} \right\|}_{L_x^\infty }}d\tau }\\
&+ \mathop {\sup }\limits_{s \in [0,1]} \left( s\left\| d\widetilde{u}\right\|^2_{L_x^{\infty} } s^{\frac{1}{2}}\left\| \partial_s\widetilde{u} \right\|_{L_x^{\infty} }  \right)\int_{\frac{s}{2}}^s \tau ^{ - \frac{3}{2}}d\tau.
\end{align*}
Hence we deduce by Lemma \ref{ktao1}
\begin{align}\label{qian4}
{\left\| {\left| {\nabla \partial_t u} \right|(s)} \right\|_{L_x^\infty }} \le MC(M){s^{ - \frac{1}{2}}}.
\end{align}
The bounds for $|\nabla_t\partial_s\widetilde{u}|$ follows by the same arguments as (\ref{qian1}) with help of Lemma \ref{zhangqiling} and  (\ref{iconm}).
\end{proof}

We summarize the long time and short time behaviors as a proposition.
\begin{proposition}\label{sl}
Let $u:[0,T]\times\Bbb H^2\to \Bbb H^2$ be a solution to (\ref{wmap1}) satisfying
\begin{align*}
\|(\nabla du,\nabla\partial_t u)\|_{L^2\times L^2}+\|(du,\partial_t u)\|_{L^2\times L^2}\le M,
\end{align*}
If $\widetilde{u}:\Bbb R^+\times[0,T]\times\Bbb H^2\to \Bbb H^2$ is the solution to (\ref{8.29.2}) with initial data $u(t,x)$, then for any $\eta>0$, it holds uniformly for $t\in[0,T]$ that
\begin{align*}
&\left\| {d\widetilde{u}} \right\|_{L_s^\infty[1,\infty) L_x^{\infty}}+\left\| {\nabla d\widetilde{u}} \right\|_{L_s^\infty[1,\infty) L_x^{\infty}}+\left\| {\nabla d\widetilde{u}} \right\|_{L_s^\infty L_x^2} + {\left\| {{{\nabla\partial _t}\widetilde{u}}} \right\|_{L_s^\infty L_x^2 }}\\
&+{\| {{s^{\frac{1}{2}}}\left| {\nabla d\widetilde{u}} \right|}\|_{L_s^\infty[0,1] L_x^\infty }}
+ {\| {e^{\delta s}\left| {{\partial _s}\widetilde{u}} \right|}\|_{L_s^\infty L_x^2 }}+
{\left\| {{s}\left| {{\nabla_t}{\partial _s}\widetilde{u}} \right|} \right\|_{L^{\infty}_s[0,1]L_x^{\infty}}}\\
&+ {\| {{s^{\frac{1}{2}}}\left| {\nabla {\partial _t}\widetilde{u}} \right|} \|_{L_s^\infty[0,1] L_x^\infty }}+ {\left\| {s\left| {\nabla {\partial _s}\widetilde{u}} \right|} \right\|_{L_s^\infty[0,1] L_x^\infty }}
+ {\| {{s^{\frac{1}{2}}}e^{\delta s}\left| {{\partial _s}\widetilde{u}} \right|}\|_{L_s^\infty L_x^\infty }}\\
&+{\| {{s^{\frac{1}{2}}}\left| {\nabla {\partial _s}\widetilde{u}} \right|}\|_{L_s^\infty[0,1] L_x^2}}
+{\| {{s^{\frac{1}{2}}}\left| {\nabla_t {\partial _s}\widetilde{u}} \right|}\|_{L_s^\infty[0,1] L_x^2}}+\left\|s^{\eta} {d\widetilde{u}} \right\|_{L_s^\infty(0,1) L_x^{\infty}}\\
&{\left\| {{s^{\frac{1}{2}}}{e^{\delta s}}\left| {\nabla {\partial _t}\widetilde{u}} \right|} \right\|_{L_s^\infty L_x^\infty }} + {\left\| {s{e^{\delta s}}\left| {\nabla {\partial _s}\widetilde{u}} \right|} \right\|_{L_s^\infty L_x^\infty }}+ {\left\| se^{\delta s}{\left| {{\nabla_t\partial _s}\widetilde{u}} \right|} \right\|_{L_s^\infty L_x^\infty }}\\
&{\left\| {{e^{\delta s}}\left| {\nabla {\partial _t}\widetilde{u}} \right|} \right\|_{L_s^\infty L_x^2 }} + {\left\| {s^{\frac{1}{2}}{e^{\delta s}}\left| {\nabla {\partial _s}\widetilde{u}} \right|}\right\|_{L_s^\infty L_x^2}}+ {\left\|s^{\frac{1}{2}}e^{\delta s} { {{\nabla_t\partial _s}\widetilde{u}}} \right\|_{L_s^\infty L_x^2 }}\le MC(M).
\end{align*}
\end{proposition}

\begin{lemma}\label{fotuo1}
If $(u,\partial_tu)$ solves (\ref{wmap1}) and $\|u(t,x)\|_{\mathcal{X}_T}\le M$, then we have
\begin{align}
\left\| {{\nabla}^2d\widetilde{u}} \right\|_{L^2_x}&\le \max(s^{-\frac{1}{2}},1)MC(M)\label{fotuo2}\\
\left\| {{\nabla}^2d\widetilde{u}} \right\|_{L^{\infty}_x}&\le \max(s^{-1},1)MC(M)\label{fotuo3}\\
se^{\delta' s}\|\nabla ^2\partial _s\widetilde{u}\|_{L^2_x}&\lesssim MC(M)\label{fotuo4}\\
s^{\frac{3}{2}}e^{\delta' s}\|\nabla ^2\partial _s\widetilde{u}\|_{L^\infty_x}&\lesssim MC(M)\label{fotuo5}
\end{align}
\end{lemma}
\begin{proof}
The Bochner formula for $\left| {{\nabla}^2d\widetilde{u}} \right|^2$ is as follows
\begin{align}
&{\partial _s}{| {{\nabla}^2d\widetilde{u}}|^2} - \Delta {| {{\nabla}^2d\widetilde{u}}|^2} + 2{| { {\nabla^3 }d\widetilde{u}}|^2} \lesssim |\nabla^2d\widetilde{u}|^2(|d\widetilde{u}|^2+1)+|\nabla d\widetilde{u}|^2|\nabla^2d\widetilde{u}||d\widetilde{u}|\nonumber\\
&+|d\widetilde{u}|^3|\nabla d\widetilde{u}||\nabla^2d\widetilde{u}|+|\nabla d\widetilde{u}||\nabla^2 d\widetilde{u}|^2.\label{piiguuuu7}
\end{align}
Interpolation by parts and $\tau(\widetilde{u})=\partial_s\widetilde{u}$ give
\begin{align}\label{hupke3}
\left\| {{\nabla}^2d\widetilde{u}} \right\|^2_{L^2_x}\lesssim \left\| \nabla\partial_s\widetilde{u} \right\|^2_{L^2_x}+\left\| \nabla d\widetilde{u} \right\|^3_{L^2_x}
+\left\| \nabla d\widetilde{u} \right\|^2_{L^2_x}\|du\|^2_{L^{\infty}_x}
\end{align}
Then Proposition \ref{sl} yields (\ref{fotuo2}). (\ref{piiguuuu7}) shows $| {{\nabla}^2d\widetilde{u}}|$ satisfies
\begin{align}
&{\partial _s}{| {{\nabla}^2d\widetilde{u}}|} - \Delta {| {{\nabla}^2d\widetilde{u}}|}\lesssim |\nabla^2d\widetilde{u}|(|d\widetilde{u}|^2+1)+|\nabla d\widetilde{u}|^2|d\widetilde{u}|+|d\widetilde{u}|^3|\nabla d\widetilde{u}|+|\nabla d\widetilde{u}||\nabla^2 d\widetilde{u}|.
\end{align}
Let $f={| {{\nabla}^2d\widetilde{u}}|}e^{-\int^{s}_0(\|d\widetilde{u}\|^2_{L^{\infty}}+\|\nabla d\widetilde{u}\|_{L^{\infty}}+1)d\kappa}$. Then for $s\in[0,1]$, by Duhamel principle and smoothing effect, Lemma \ref{ktao1},
\begin{align}
&\|f(s,x)\|_{L^{\infty}_x}\lesssim s^{-\frac{1}{2}}\|f(\frac{s}{2},x)\|_{L^{2}_x}+\int^s_{\frac{s}{2}}(s-\tau)^{-\frac{1}{2}}\||\nabla d\widetilde{u}|^2|d\widetilde{u}|+|d\widetilde{u}|^3|\nabla d\widetilde{u}|\|_{L^2_x}d\tau.
\end{align}
Then (\ref{fotuo3}) when $s\in[0,1]$ follows by Lemma \ref{ktao1} and Proposition \ref{sl}. (\ref{fotuo2}) gives $\| {{\nabla}^2d\widetilde{u}}\|_{L^2_x}\le MC(M)$ for all $s\ge1$. Meanwhile Proposition \ref{sl} shows $\|\nabla d\widetilde{u}\|_{L^{\infty}}+\| d\widetilde{u}\|_{L^{\infty}}\le MC(M)$ when $s\ge1$. Then if let $Z\triangleq e^{-C_1(M)s}(e^{-C_1(M) s}|\nabla^2 d\widetilde{u}|+C_1(M))$, then $(\partial_s-\Delta)Z\le0$.
Applying Remark \ref{ki78} to $Z$ gives
\begin{align}
\|\nabla^2 d\widetilde{u}(s,x)\|^2_{L^{\infty}_x}\lesssim \int^{s}_{s-1}\|\nabla^2 d\widetilde{u}(\tau,x)\|^2_{L^{2}_x}d\tau+MC(M).
\end{align}
Then (\ref{fotuo3}) when $s\ge1$ follows by (\ref{hupke3}), (\ref{ingjh}) and Proposition \ref{sl}. The Bochner formula for $\left| {{\nabla}^2\partial_s{\widetilde{u}}} \right|^2$ is as follows
\begin{align}
&{\partial _s}{\left| {{\nabla}^2{\partial _s}\widetilde{u}} \right|^2} - \Delta {\left| {{\nabla}^2{\partial _s}\widetilde{u}} \right|^2} + 2{\left| { {\nabla^3 }{\partial _s}\widetilde{u}} \right|^2} \lesssim |\nabla^2\partial_s\widetilde{u}|^2(|d\widetilde{u}|^2+1)+|\partial_s\widetilde{u}|^2|\nabla^2\partial_s\widetilde{u}||\nabla d\widetilde{u}|\nonumber\\
&+|\partial_s\widetilde{u}||d\widetilde{u}||\nabla \partial_s\widetilde{u}||\nabla^2\partial_s\widetilde{u}|+|\nabla^2\partial_s\widetilde{u}|^2|d\widetilde{u}||\partial_s\widetilde{u}|+
|\nabla^2d\widetilde{u}||d\widetilde{u}||\partial_s\widetilde{u}||\nabla^2\partial_s\widetilde{u}|\nonumber\\
&+|\nabla\partial_s\widetilde{u}||\nabla^2\partial_s\widetilde{u}||d\widetilde{u}||\nabla d\widetilde{u}|
+|\nabla\partial_s\widetilde{u}||\nabla^2\partial_s\widetilde{u}||d\widetilde{u}||\nabla d\widetilde{u}|+|\nabla d\widetilde{u}||\nabla^2 \partial_s\widetilde{u}|^2.\label{ftuo4}
\end{align}
Then one has
\begin{align*}
&\frac{d}{{ds}}\int_{{\Bbb H^2}} {{s^2}{{\left| {{\nabla ^2}{\partial _s}\widetilde{u}} \right|}^2}{\rm{dvol_h}}}  \\
&\le \int_{{\Bbb H^2}} 2{s{{\left| {{\nabla ^2}{\partial _s}\widetilde{u}} \right|}^2}}  - 2{s^2}{\left| {{\nabla ^3}{\partial _s}\widetilde{u}} \right|^2} + {s^2}{\left| {{\nabla ^2}{\partial _s}\widetilde{u}} \right|^2}|d\widetilde{u}{|^2}{\rm{dvol_h}} \\
&+ \int_{{\Bbb H^2}} {{s^2}} |{\nabla ^2}d\widetilde{u}||d\widetilde{u}||{\partial _s}\widetilde{u}||{\nabla ^2}{\partial _s}\widetilde{u}| + {s^2}|\nabla {\partial _s}\widetilde{u}||{\nabla ^2}{\partial _s}\widetilde{u}||d\widetilde{u}||\nabla d\widetilde{u}|{\rm{dvol_h}} \\
&+ \int_{{\Bbb H^2}} {{s^2}|{\partial _s}\widetilde{u}||d\widetilde{u}||\nabla {\partial _s}\widetilde{u}||{\nabla ^2}{\partial _s}\widetilde{u}|}  + {s^2}|{\nabla ^2}{\partial _s}u{|^2}|d\widetilde{u}||{\partial _s}\widetilde{u}|{\rm{dvol_h}} \\
&+ \int_{{\Bbb H^2}} {{s^2}|{\partial _s}\widetilde{u}{|^2}|{\nabla ^2}{\partial _s}\widetilde{u}||\nabla d\widetilde{u}| + } {s^2}|{\nabla ^2}{\partial _s}\widetilde{u}|^2|\nabla d\widetilde{u}|{\rm{dvol_h}} \\
&+ \int_{{\Bbb H^2}} {{s^2}} |\nabla {\partial _s}\widetilde{u}||{\nabla ^2}{\partial _s}\widetilde{u}||d\widetilde{u}||\nabla d\widetilde{u}|{\rm{dvol_h}}
+\int_{{\Bbb H^2}} {{s^2}{{\left| {{\nabla ^2}{\partial _s}\widetilde{u}} \right|}^2}{\rm{dvol_h}}}.
\end{align*}
Integrating the above formula in $s\in[s_1,\tau]$ with any $0<s_1<\tau<2$, by Sobolev embedding, Gagliardo-Nirenberg  and Young inequality, we obtain
\begin{align*}
 &\int_{{\Bbb H^2}} {{\tau ^2}{{\left| {{\nabla ^2}{\partial _s}\widetilde{u}} \right|}^2}(\tau ,t){\rm{dvol_h}}}  - \int_{{\Bbb H^2}} {s_1^2{{\left| {{\nabla ^2}{\partial _s}\widetilde{u}} \right|}^2}({s_1},t){\rm{dvol_h}}}  \\
 &\lesssim \int_{{s_1}}^\tau  {\int_{{\Bbb H^2}} {s{{\left| {{\nabla ^2}{\partial _s}\widetilde{u}} \right|}^2}} }  - {s^2}{\left| {{\nabla ^3}{\partial _s}\widetilde{u}} \right|^2} + \left\| {d\widetilde{u}} \right\|_{L_s^\infty L_x^4}^2{s^2}{\left| {{\nabla ^2}{\partial _s}\widetilde{u}} \right|^2}{\rm{dvol_h}}ds \\
 &+ \int_{{s_1}}^\tau  {\int_{{\Bbb H^2}} {{s^3}|d\widetilde{u}|^2|{\partial _s}\widetilde{u}|^2} } {\left| {{\nabla ^2}d\widetilde{u}} \right|^2} + {s^3}|\nabla {\partial _s}\widetilde{u}|^2|d\widetilde{u}|^2|\nabla d\widetilde{u}|^2{\rm{dvol_h}}ds \\
 &+ \int_{{s_1}}^\tau  {\int_{{\Bbb H^2}} {{s^3}|{\partial _s}\widetilde{u}|^2|d\widetilde{u}|^2|\nabla {\partial _s}\widetilde{u}|^2} }  + {s^2}|d\widetilde{u}|^2|{\partial _s}\widetilde{u}|^2{\rm{dvol_h}}ds \\
 &+ \int_{{s_1}}^\tau  {\int_{{\Bbb H^2}} {{s^3}|\nabla d\widetilde{u}|^2|{\partial _s}\widetilde{u}|^2} }  + {s^2}|\nabla d\widetilde{u}|^2{\rm{dvol_h}}ds \\
 &+ \int_{{s_1}}^\tau  {\int_{{\Bbb H^2}} {{s^3}} |\nabla {\partial _s}\widetilde{u}|^2|d\widetilde{u}|^2|\nabla d\widetilde{u}|^2{\rm{dvol_h}}} ds.
 \end{align*}
Thus letting $s_1\to0$, for $\tau\in(0,2)$, we deduce from (\ref{ki6ll1}), (\ref{fotuo3}) and Proposition \ref{sl} that
\begin{align*}
\|s\nabla ^2\partial _s\widetilde{u}\|_{L^2_x}\lesssim MC(M),
\end{align*}
from which (\ref{fotuo4}) when $s\in(0,1)$ follows.
Integrating (\ref{ftuo4}) with respect to $x$ in $\Bbb H^2$, one obtains by (\ref{{uv111}}) and Proposition \ref{sl} especially the $L^{\infty}_x$ bounds for $|d\widetilde{u}|+|\nabla d\widetilde{u}|$ that for $s\ge1$ and any $0<c\ll 1$
\begin{align}
&\frac{d}{{ds}}\left\| {{\nabla ^2}{\partial _s}\widetilde{u}} \right\|_{L_x^2}^2+c\left\| {{\nabla ^2}{\partial _s}\widetilde{u}} \right\|_{L_x^2}^2\lesssim {\left\| {{\partial _s}\widetilde{u}} \right\|_{L_x^\infty }}{\left\| {\nabla {\partial _s}\widetilde{u}} \right\|_{L_x^2}}{\left\| {{\nabla ^2}{\partial _s}\widetilde{u}} \right\|_{L_x^2}}+\frac{1}{c}\left\| {{\nabla ^2}{\partial _s}\widetilde{u}} \right\|_{L_x^2}^2\nonumber\\
&+ {\left\| {{\partial _s}\widetilde{u}} \right\|_{L_x^2}}{\left\| {{\nabla ^2}{\partial _s}\widetilde{u}} \right\|_{L_x^2}} + {\left\| {\nabla {\partial _s}\widetilde{u}} \right\|_{L_x^2}}{\left\| {{\nabla ^2}{\partial _s}\widetilde{u}} \right\|_{L_x^2}}+{\left\| {{\nabla ^2}{\partial _s}\widetilde{u}} \right\|_{L_x^2}}\left\| {{\partial _s}\widetilde{u}} \right\|_{L_x^4}^2.\label{kioplmmnn}
\end{align}
Meanwhile integrating (\ref{hu897}) with respect to $s$ in $(s',\infty)$, we obtain from the exponential decay of $|\partial_s\widetilde{u}|+|\nabla\partial_s\widetilde{u}|$ in Proposition \ref{sl} that for $s'\ge1$
\begin{align}\label{oi98nhbgfg}
\int^{\infty}_{s'}\|\nabla ^2\partial _s\widetilde{u}\|_{L^2_x}d\tau\lesssim e^{-\delta s'}MC(M).
\end{align}
Hence Gronwall inequality gives if choosing $0<c<\delta$ then for $s\ge1$ one has
\begin{align}
\left\| {{\nabla ^2}{\partial _s}\widetilde{u}} (s)\right\|_{L_x^2}^2&\le e^{-cs}MC(M)+e^{-cs}\int^{s}_{1}e^{c\tau}\|\nabla ^2\partial _s\widetilde{u}\|^2_{L^2_x}d\tau
+e^{-cs}\int^{s}_{1}e^{c\tau}e^{-\delta \tau}d\tau\nonumber\\
&\lesssim MC(M).\label{po67gvg}
\end{align}
Applying Gronwall inequality to (\ref{kioplmmnn}) again in $(\frac{s}{2},s)$, we deduce from (\ref{po67gvg}) that
\begin{align}
\left\| {{\nabla ^2}{\partial _s}\widetilde{u}} (s)\right\|_{L_x^2}^2\le e^{-\frac{c}{2}s}MC(M)+e^{-cs}\int^{s}_{\frac{s}{2}}e^{c\tau}\|\nabla ^2\partial _s\widetilde{u}\|^2_{L^2_x}d\tau
+e^{-cs}\int^{s}_{\frac{s}{2}}e^{c\tau}e^{-\delta \tau}d\tau.
\end{align}
Thus (\ref{fotuo4}) follows by (\ref{oi98nhbgfg}).
Finally (\ref{fotuo5}) follows by (\ref{fotuo4}) and applying Remark \ref{ki78} to (\ref{ftuo4}) as before.
\end{proof}

\subsection{The existence of caloric gauge}
As a preparation for the existence of the caloric gauge, we prove that the heat flows initiated from $u(t,x)$ with different $t$ converge to the same harmonic map as $u_0$.
\begin{lemma}\label{lhu880}
If $(u,\partial_tu)$ is a solution to (\ref{wmap1}) in $\mathcal{X}_T$, then there exists a harmonic map $\widetilde{Q}$ such that as $s\to\infty$,
$$
\mathop {\lim }\limits_{s \to \infty } \mathop {\sup }\limits_{(x,t) \in {\mathbb{H}^2} \times [0,T]}
dist_{\Bbb H^2}(\widetilde{u}(s,x,t),\widetilde{Q}(x))=0.
$$
\end{lemma}
\begin{proof}
The global existence of $\widetilde{u}$ is due to Lemma \ref{8.44}, the embedding ${\bf{H}^2}\hookrightarrow L^{\infty}$, ${\bf{H}^1}\hookrightarrow L^{p}$ for $p\in[2,\infty)$ and diamagnetic inequality.
Then (\ref{8.3}), maximum principle and (\ref{huhu899}) show
\begin{align}\label{10.112}
\left\| {{\partial _s}\widetilde{u}(s,t,x)} \right\|^2_{L^{\infty}_x} \le {s^{ - 1}}{e^{ - \frac{1}{4}s}}\int_{{\mathbb{H}^2}} {{{\left| {{\partial _s}\widetilde{u}(0,t,x)} \right|}^2}{\rm{dvol_h}}}.
\end{align}
Thus (\ref{8.29.2}) yields
$$
\mathop {\sup }\limits_{(x,t) \in {\mathbb{H}^2}\times[0,T]} \left| {{\partial _s}\widetilde{u}(s,t,x)} \right| \le {s^{ -\frac{1}{2}}}{e^{ - \frac{1}{8}s}}\int_{{\mathbb{H}^2}} {{{\left| {{\partial _t}u(t,x)} \right|}^2}{\rm{dvol_h}}}\le C{s^{ - 1}}{e^{ - \frac{1}{8}s}}.
$$
Therefore for any $1<s_0<s_1<\infty$ it holds
$${d_{{\mathbb{H}^2}}}(\widetilde{u}({s_0},t,x),\widetilde{u}({s_1},t,x)) \lesssim \int_{{s_0}}^{{s_1}} {{e^{ - \frac{1}{8}s}}ds},$$
which implies $\widetilde{u}(s,t,x)$ converges to some map $\widetilde{Q}(t,x)$ uniformly on $(t,x)\in[0,T]\times \mathbb{H}^2$. By [Theorem 5.2,\cite{LT}], for any fixed $t$, $\widetilde{Q}(t,x)$ is a harmonic map form $\Bbb H^2\to\Bbb H^2$. It suffices to verify $\widetilde{Q}(t,x)$ is indeed independent of $t$.
By (\ref{10.127}), maximum principle and (\ref{huhu899}),
\begin{align}\label{sdf}
\mathop {\sup }\limits_{x \in {\mathbb{H}^2}} {\left| {{\partial _t}\widetilde{u}(s,t,x)} \right|^2} \le {s^{ - 1}}{e^{ - \frac{1}{4}s}}\int_{{\Bbb H^2}} {{{\left| {{\partial _t}\widetilde{u}(0,t,x)} \right|}^2}{\rm{dvol_h}}}.
\end{align}
As a consequence, for $0\le t_1<t_2\le T$ one has
$${d_{{\Bbb H^2}}}(\widetilde{u}(s,{t_1},x),\widetilde{u}(s,{t_2},x)) \le \int_{{t_1}}^{{t_2}} {\left| {{\partial _t}\widetilde{u}(s,t,x)} \right|} dt \le C{s^{ - \frac{1}{2}}}{e^{ - s/8}}({t_2} - {t_1}).
$$
Let $s\to\infty$, we get ${d_{{\mathbb{H}^2}}}(\widetilde{Q}({t_1},x),\widetilde{Q}({t_2},x)) = 0$, thus finishing the proof.
\end{proof}

\begin{lemma}\label{z8vcxzvb}
Let $Q$ be an admissible harmonic map in Definition 1.1, and $\mu_1,\mu_2$ be sufficiently small.
If $(u,\partial_tu)$ is a solution to (\ref{wmap1}) in $\mathcal{X}_T$, then $\widetilde{u}(s,t,x)$ uniformly converges to $Q$ as $s\to\infty$.
\end{lemma}
\begin{proof}
By Lemma \ref{lhu880}, it suffices to prove $Q=\widetilde{Q}$. In the coordinate (\ref{vg}), the harmonic map equation can be written as
\begin{align}
 \Delta {\widetilde{Q}^l} + {h^{ij}}\overline \Gamma  _{pq}^l\frac{{\partial {\widetilde{Q}^p}}}{{\partial {x_i}}}\frac{{\partial {\widetilde{Q}^q}}}{{\partial {x_j}}} &= 0\label{cxvgn} \\
 \Delta {Q^l} + {h^{ij}}\overline \Gamma  _{pq}^l\frac{{\partial {Q^p}}}{{\partial {x_i}}}\frac{{\partial {Q^q}}}{{\partial {x_j}}} &= 0. \label{cxvgn2}
\end{align}
Denote the heat flow initiated from $u_0$ by $U(s,x)$, then by (\ref{{uv111}}),
\begin{align*}
&\|U^1(s,x)-Q^1(x)\|_{L^2}+\|U^2(s,x)-Q^2(x)\|_{L^2}\\
&\lesssim \|\nabla(U^1-Q^1)\|_{L^2} +\|\nabla(U^2-Q^2)\|_{L^2}.
\end{align*}
[Lemma 2.3,\cite{LZ}] shows that for $k=1,2,l=1,2,$
\begin{align*}
\|\nabla^l U^k\|_{L^2}\lesssim C(\|U\|_{\mathfrak{H}^2},R_0,\|Q\|_{\mathfrak{H}^2})\|\nabla^{l-1}dU\|_{L^2}.
\end{align*}
By energy arguments, one obtains $\|\nabla dU\|_{L^2}\le C(\|\nabla du_0\|_{L^2},\|du_0\|_{L^2})$ and the energy decreases along the heat flow, see (\ref{V5}).
Thus we have by Sobolev embedding and Corollary \ref{new2} that
\begin{align}
&\|U^1(s,x)-Q^1(x)\|_{L^2}+\|U^2(s,x)-Q^2(x)\|_{L^2}+ \|dU\|_{L^2}\nonumber\\
&\le C(R_0)\mu_2+C(R_0)\mu_1\label{hbvcjin}\\
&\|U^1(s,x)\|_{L^{\infty}}+\|U^2(s,x)\|_{L^{\infty}}\le C(R_0).\label{pokeryu}
\end{align}
Hence letting $s\to\infty$, we have for some constant $C(R_0)$
\begin{align}
\|\widetilde{Q}^1\|_{L^{\infty}}+\|\widetilde{Q}^2\|_{L^{\infty}}\le C,\mbox{  }\|\nabla\widetilde{Q}^1\|_{L^{2}}+\|\nabla\widetilde{Q}^2\|_{L^{2}}\le \mu_1C(R_0)\label{ktu8n2}
\end{align}
Multiplying the difference between (\ref{cxvgn}) and (\ref{cxvgn2}) with $-{Q^l} + {\widetilde{Q}^l}$, we have by integration by parts that
\begin{align*}
 &{\left\| {\nabla \left( {{Q^l} - {\widetilde{Q}^l}} \right)} \right\|_{{L^2}}} \le \left\langle {{h^{ij}}\left( {\overline \Gamma  _{pq}^l(Q) - \overline \Gamma  _{pq}^l(\widetilde Q)} \right)\frac{{\partial {Q^p}}}{{\partial {x_i}}}\frac{{\partial {Q^q}}}{{\partial {x_j}}}, - {Q^l} + {{\widetilde Q}^l}} \right\rangle \\
 &+ \left\langle {h^{ij}}\overline \Gamma  _{pq}^l(\widetilde Q)\left( \frac{\partial {Q^p}}{\partial {x_i}} - \frac{\partial {\widetilde Q}^p}{\partial {x_i}} \right)\frac{{\partial {Q^q}}}{{\partial {x_j}}}, - {Q^l} + {\widetilde{Q}^l}\right\rangle  \\
 &+ \left\langle {{h^{ij}}\overline \Gamma  _{pq}^l(\widetilde{Q})\frac{{\partial {\widetilde{Q}^p}}}{{\partial {x_i}}}\left( {\frac{{\partial {Q^q}}}{{\partial {x_i}}} - \frac{{\partial {\widetilde{Q}^q}}}{{\partial {x_j}}}} \right), - {Q^l} + {\widetilde{Q}^l}} \right\rangle.
 \end{align*}
Thus using the explicit formula for ${\overline \Gamma _{pq}^l}$, by (\ref{hbvcjin}), (\ref{pokeryu}), (\ref{ktu8n2}) we get
\begin{align*}
&\left\| {\nabla \left( {{Q^l} - {\widetilde{Q}^l}} \right)} \right\|_{{L^2}}^2\\
&\lesssim \left( {\left\| {{Q^l} - {\widetilde{Q}^l}} \right\|_{{L^2}}^2 + \left\| {\nabla \left( {{Q^l} - {\widetilde{Q}^l}} \right)} \right\|_{{L^2}}^2} \right)\left( {\sum\limits_{k = 1}^2 {\left\| {\nabla {\widetilde{Q}^k}} \right\|_{{L^2}}^2 + \left\| {\nabla {Q^k}} \right\|_{{L^2}}^2} } \right).
\end{align*}
Therefore, we conclude for some constant $C(R_0)$ which is independent of $\mu_1,\mu_2$ provided $0\le \mu_1,\mu_2\le1$
\begin{align*}
&\sum\limits_{l = 1}^2 {\left\| {\nabla \left( {{Q^l} - {\widetilde{Q}^l}} \right)} \right\|_{{L^2}}^2}\\
&\le C(R_0)\left( {\left\| {dQ} \right\|_{{L^2}}^2 + \left\| {d\widetilde{Q}} \right\|_{{L^2}}^2} \right)\left( {\sum\limits_{l = 1}^2 {\left\| {\nabla \left( {{Q^l} - {\widetilde{Q}^l}} \right)} \right\|_{{L^2}}^2 + \left\| {{Q^l} - {\widetilde{Q}^l}} \right\|_{{L^2}}^2} } \right).
\end{align*}
Let $\mu_1$, $\mu_2$ be sufficiently small, (\ref{{uv111}}) gives
\begin{align*}
&\sum\limits_{l = 1}^2 {\left\| {\nabla \left( {{Q^l} - {\widetilde{Q}^l}} \right)} \right\|_{{L^2}}^2}  + \left\| {{Q^l} - {\widetilde{Q}^l}} \right\|_{{L^2}}^2\\
&\le \left( {{\mu _1} + {\mu _2}} \right)\left( {\sum\limits_{l = 1}^2 {\left\| {\nabla \left( {{Q^l} - {\widetilde{Q}^l}} \right)} \right\|_{{L^2}}^2 + \left\| {{Q^l} - {\widetilde{Q}^l}} \right\|_{{L^2}}^2} } \right).
\end{align*}
Hence $\widetilde{Q}=Q$.
\end{proof}

Now we are ready to prove the existence of the caloric gauge in Definition \ref{pp}.
\begin{proposition}\label{3.3}
Given any solution $(u,\partial_tu)$ of (\ref{wmap1}) in $\mathcal{X}_T$ with $(u_0,u_1)\in \bf H_Q^3\times\bf H_Q^2$. For any fixed frame $\Xi\triangleq\{\Xi_1(Q(x)),\Xi_2(Q(x))\}$, there exists a unique corresponding caloric gauge defined in Definition \ref{pp}.
\end{proposition}
\begin{proof}
We first show the existence part. Choose an arbitrary orthonormal frame $E_0(t,x)\triangleq\{\texttt{e}_i(t,x)\}^2_{i=1}$ such that $E_0(t,x)$ spans the tangent space $T_{u(t,x)}{\mathbb{H}^2}$ for each $(t,x)\in [0,T]\times \mathbb{H}^2$. The desired frame does exist, in fact we have a global orthonormal frame for $\mathbb{H}^2$ defined by (\ref{frame}). Then evolving (\ref{8.29.2}) with initial data $u(t,x)$, we have from Lemma \ref{z8vcxzvb} that $\widetilde{u}(s,t,x)$ converges to $Q$ uniformly for $(t,x)\in[0,T]\times \mathbb{H}^2$ as $s\to\infty$. Meanwhile, we evolve $E_0$ in $s$ according to
\begin{align}\label{11.2}
\left\{ \begin{array}{l}
{\nabla _s}{\Omega _i}(s,t,x) = 0 \\
{\Omega _i}(s,t,x)\upharpoonright_{s=0} = {\texttt{e}_i}(t,x) \\
\end{array} \right.
\end{align}
Denote the evolved frame as $E_s\triangleq \{\Omega_i(s,t,x)\}^2_{i=1}$.  We claim that there exists some orthonormal frame $E_{\infty}\triangleq\{\texttt{e}_i(\infty,t,x)\}^2_{i=1}$ which spans $T_{Q(x)}\Bbb H^2$ for each $(t,x)\in [0,T]\times\Bbb H^2$ such that
\begin{align}\label{pl}
\mathop {\lim }\limits_{s \to \infty }{\Omega_i(s,t,x)}  = \texttt{e}_i(\infty,t,x).
\end{align}
Indeed, by the definition of the convergence of frames given in (\ref{convergence}) and the fact $\widetilde{u}(s,t,x)$ converges to $Q(x)$, it suffices to show for some scalar function $c_i:[0,T]\times \mathbb{H}^2\to \Bbb R$
\begin{align}\label{aw}
\mathop {\lim }\limits_{s \to \infty } \left\langle {{\Omega_i}(s,t,x),{\Theta _i}(\widetilde{u}(s,t,x))} \right\rangle = c_i(t,x).
\end{align}
By direct calculations,
$$
\left| \nabla _s \Theta_i(\widetilde{u}(s,t,x))\right| \lesssim \left| {{\partial _s}\widetilde{u}} \right|.
$$
then (\ref{10.112}) and $\nabla_s\Omega=0$ imply that for $s>1$
$$\big|{\partial _s}\left\langle {{\Omega_i}(s,t,x),{\Theta _i}(\widetilde{u}(s,t,x))} \right\rangle \big|\lesssim Me^{-\delta s}.
$$
Hence $(\ref{aw})$ holds for some $c_i(t,x)$, thus verifying (\ref{pl}). It remains to adjust the initial frame $E_0$ to make the limit frame $E_{\infty}$ coincide with the given frame $\Xi$. This can be achieved by the gauge transform invariance illustrated in Section 2.1. Indeed, since for any $U:[0,T]\times\mathbb{H}^2\to SO(2)$, and the solution $\widetilde{u}(s,t,x)$ to (\ref{8.29.2}), one has $\nabla_s U(t,x)\Omega(s,t,x)=U(t,x)\nabla_s\Omega(s,t,x)$, then the following gauge symmetry holds
\begin{align*}
 {E_0}\triangleq\left\{ {\texttt{e}_i(t,x)} \right\}^2_{i=1} &\mapsto {{E'}_0}\triangleq\left\{ {U(t,x){\texttt{e}_i}(t,x)} \right\}^2_{i=1} \\
 {E_s}\triangleq\left\{ {{\Omega_i}(s,t,x)} \right\}^2_{i=1} &\mapsto {{E'}_s}\triangleq\left\{ {U(t,x){\Omega_i}(s,t,x)} \right\}^2_{i=1}.
\end{align*}
Therefore choosing $U(t,x)$ such that $U(t,x)E_{\infty}=\Xi$, where $E_{\infty}$ is the limit frame obtained by (\ref{pl}), suffices for our purpose. The uniqueness of the gauge follows from the identity
$$
\frac{d}{{ds}}\left\langle {{\Phi _1} - {\Phi_2},{\Phi_1} - {\Phi_2}} \right\rangle  = 0,
$$
where $(\Phi _1)$ and $(\Phi _2)$ are two caloric gauges satisfying (\ref{muqi}).
\end{proof}

\subsection{Expressions for the connection coefficients}
The following lemma gives the expressions for the connection coefficients matrix $A_{x,t}$ by differential fields. The proof of Lemma \ref{po87bg} is almost the same as
[Lemma 3.6,\cite{LZ}], thus we omit it.
\begin{lemma}\label{po87bg}
Suppose that $\Omega(s,t,x)$ is the caloric gauge constructed in Proposition \ref{3.3}, then we have for $i=1,2$
\begin{align}
&\mathop {\lim }\limits_{s \to \infty } [{A_i}]^j_k(s,t,x) =\left\langle {{\nabla _i}{\Xi_k }(x),{\Xi_j}(x)} \right\rangle\label{kji}\\
&\mathop {\lim }\limits_{s \to \infty } {A_t}(s,t,x) = 0\label{kji22}
\end{align}
Particularly let $\Xi(x)=\Theta(Q(x))$ in Proposition \ref{3.3}, denote $A^{\infty}_i$ the limit coefficient matrix, i.e.,  $[A^{\infty}_i]^k_j=\left\langle {{\nabla _i}{\Xi_k }(Q(x)),{\Xi_j}(Q(x))} \right\rangle$,
then we have for $i=1,2$, $s>0$,
\begin{align}
&{A_i}(s,t,x)\sqrt{h^{ii}(x)} = \int_s^\infty \sqrt{ h^{ii}(x)}{\mathbf{R}(\widetilde{u}(\kappa))\left( {{\partial _s}\widetilde{u}(\kappa),{\partial _i}\widetilde{u}(\kappa)} \right)} d\kappa + { \sqrt{h^{ii}(x)}}A^{\infty}_i.\label{edf}\\
&{A_t}(s,t,x)=\int^{\infty}_s\phi_s\wedge\phi_td\kappa,\label{edf22}
\end{align}
\end{lemma}

\begin{remark}\label{3sect}
For convenience, we rewrite (\ref{edf}) as $A_i(s,t,x)=A^{\infty}_i(s,t,x)+A^{con}_i(s,t,x),$
where $A^{\infty}_i$ denotes the limit part,
and $A^{con}_i$ denotes the controllable part, i.e.,
\begin{align*}
A^{con}_i=\int^{\infty}_s\phi_s\wedge\phi_id\kappa.
\end{align*}
Similarly, we split $\phi_i$ into $\phi_i=\phi^{\infty}_i+\phi^{con}_i$, where $\phi^{con}_i=\int^{\infty}_s\partial_s\phi_id\kappa,$
and
$$\phi^{\infty}_i={\left( {\left\langle {{\partial _i}Q(x),{\Xi _1}(Q(x))} \right\rangle
 ,\left\langle {{\partial _i}Q(x),{\Xi _2}(Q(x))} \right\rangle } \right)^t}.$$
\end{remark}

\section{Derivation of the master equation for the heat tension field}
Recall that the heat tension filed $\phi_s$ satisfies
\begin{align}\label{heat}
\phi_s=h^{ij}D_i\phi_j-h^{ij}\Gamma^k_{ij}\phi_k.
\end{align}
And we define the wave tension filed as Tao by
\begin{align}\label{wm}
\mathfrak{W} = {D_t}{\phi _t} - {h^{ij}}{D_i}{\phi _j} + {h^{ij}}\Gamma _{ij}^k{\phi _k}.
\end{align}
In fact (\ref{heat}) is the gauged equation for the heat flow equation, and (\ref{wm}) is the gauged equation for the wave map (\ref{wmap1}), see Lemma 2.7.
The evolution of $\phi_s$ with respect to $t$ is given by the following lemma.

\begin{lemma}\label{asdf}
The heat tension field $\phi_s$ satisfies
\begin{align}
{D_t}{D_t}{\phi _s} - {h^{ij}}{D_i}{D_j}{\phi _s} + {h^{ij}}\Gamma _{ij}^k{D_k}{\phi _s} &= {\partial _s}\mathfrak{W} + {h^{ij}}\mathbf{R}({\partial _s}\widetilde{u},{\partial _i}\widetilde{u})\left( {\partial_j\widetilde{u}} \right) \nonumber\\
&+ \mathbf{R}({\partial _t}\widetilde{u},{\partial _s}\widetilde{u})\left( {\partial_t\widetilde{u}} \right).\label{heating}
\end{align}
\end{lemma}
\begin{proof}
By the torsion free identity and the commutator identity, we have
\begin{align*}
&{D_t}{D_t}{\phi _s} = {D_t}{D_s}{\phi _t} = {D_s}{D_t}{\phi _t} + \mathbf{R}({\partial _t}\widetilde{u},{\partial _s}\widetilde{u})\left( { \partial_t\widetilde{u}} \right) \\
&= {D_s}\left( {\mathfrak{W} + {h^{ij}}{D_i}{\phi _j} - {h^{ij}}\Gamma _{ij}^k{\phi _k}} \right) +\mathbf{R}({\partial _t}\widetilde{u},{\partial _s}\widetilde{u})\left( {\partial_t\widetilde{u}} \right) \\
&= {\partial _s}\mathfrak{W} + {h^{ij}}{D_s}{D_i}{\phi _j} - {h^{ij}}\Gamma _{ij}^k{D_s}{\phi _k} + \mathbf{R}({\partial _t}\widetilde{u},{\partial _s}\widetilde{u})\left( {\partial_t\widetilde{u}} \right) \\
&= {\partial _s}\mathfrak{W} + {h^{ij}}{D_i}{D_j}{\phi _s} - {h^{ij}}\Gamma _{ij}^k{D_k}{\phi _s} + {h^{ij}}\mathbf{R}({\partial _s}\widetilde{u},{\partial _i}\widetilde{u})\left( {{\partial_j\widetilde{u}}} \right) + \mathbf{R}({\partial _t}\widetilde{u},{\partial _s}\widetilde{u})\left( {\partial_t\widetilde{u}} \right).
\end{align*}
Thus (\ref{heating}) is verified.
\end{proof}

The evolution of $\mathfrak{W}$ with respect to $s$ is given by the following lemma.
\begin{lemma}\label{ab1}
Under orthogonal coordinates, the wave tension field $\mathfrak{W}$ satisfies
\begin{align*}
{\partial _s}\mathfrak{W} =& \Delta \mathfrak{W} + 2h^{ii}{A_i}{\partial _i}\mathfrak{W} +h^{ii} {A_i}{A_i}\mathfrak{W} + h^{ii}{\partial _i}{A_i}\mathfrak{W} - {h^{ii}}\Gamma _{ii}^k{A_k}\mathfrak{W} + {h^{ii}}\left( {\mathfrak{W} \wedge {\phi _i}} \right){\phi _i}\\
& + 3{h^{ii}}({\partial _t}\widetilde{u} \wedge {\partial _i}\widetilde{u}){\nabla _t}{\partial _i}\widetilde{u}.
\end{align*}
\end{lemma}
\begin{proof}
In the following calculations, we always use the convention in Remark 2.1.
By $\mathfrak{W}=D_t\phi_t-\phi_s$, we have from commutator equality that
$${\partial _s}\mathfrak{W}= {D_s}({D_t}{\phi _t} - {\phi _s}) = {D_t}{D_t}{\phi _s} - {D_s}{\phi _s} + {\bf R}({\partial _s}\widetilde{u},{\partial _t}\widetilde{u})\left( {{\partial_t\widetilde{u}}} \right).
$$
Further applications of the torsion free identity and commutator identity show
\begin{align*}
&{D_t}{D_t}{\phi _s} - {D_s}{\phi _s} \\
&= {D_t}{D_t}\left( {{h^{ij}}{D_i}{\phi _j} - {h^{ij}}\Gamma _{ij}^k{\phi _k}} \right) - {D_s}\left( {{h^{ij}}{D_i}{\phi _j} - {h^{ij}}\Gamma _{ij}^k{\phi _k}} \right) \\
&= {h^{ij}}{D_t}{D_t}{D_i}{\phi _j} - {h^{ij}}\Gamma _{ij}^k{D_t}{D_t}{\phi _k} - \left( {{h^{ij}}{D_s}{D_i}{\phi _j} - {h^{ij}}\Gamma _{ij}^k{D_s}{\phi _k}} \right) \\
&= {h^{ij}}{D_t}\left( {{D_i}{D_j}{\phi _t} + \mathbf{R}({\partial _t}\widetilde{u},{\partial _i}\widetilde{u})({\partial _j}\widetilde{u})} \right) - {h^{ij}}\Gamma _{ij}^k\left( {{D_k}{D_t}{\phi _t} + \mathbf{R}({\partial _t}\widetilde{u},{\partial _k}\widetilde{u})({\partial _t}\widetilde{u})} \right) \\
&- \left( {{h^{ij}}{D_i}{D_j}{\phi _s} - {h^{ij}}\Gamma _{ij}^k{D_k}{\phi _s} + {h^{ij}}\mathbf{R}({\partial _s}\widetilde{u},{\partial _i}\widetilde{u})({\partial _j}\widetilde{u})} \right) \\
&= {h^{ij}}{D_t}{D_i}{D_j}{\phi _t} - {h^{ij}}\Gamma _{ij}^k{D_k}{D_t}{\phi _t} - {h^{ij}}{D_i}{D_j}{\phi _s} + {h^{ij}}\Gamma _{ij}^k{D_k}{\phi _s} - {h^{ij}}\mathbf{R}({\partial _s}\widetilde{u},{\partial _i}\widetilde{u})(\partial_j\widetilde{u}) \\
&+ {h^{ij}}{\nabla _t}\left( {\mathbf{R}({\partial _t}\widetilde{u},{\partial _i}\widetilde{u})({\partial _j}\widetilde{u})} \right) - {h^{ij}}\Gamma _{ij}^k\mathbf{R}({\partial _t}\widetilde{u},{\partial _k}\widetilde{u})\left( {{\partial _t}\widetilde{u}} \right).
\end{align*}
The leading term can be written as
\begin{align*}
{h^{ij}}{D_t}{D_i}{D_j}{\phi _t}& = {h^{ij}}{D_i}{D_t}{D_j}{\phi _t} + {h^{ij}}\left( {\mathbf{R}({\partial _t}\widetilde{u},{\partial _i}\widetilde{u})e\left( {{D_j}{\phi _t}} \right)} \right) \\
&= {h^{ij}}{D_i}{D_j}{D_t}{\phi _t} + {h^{ij}}\left( {\mathbf{R}({\partial _t}\widetilde{u},{\partial _i}\widetilde{u}){\nabla _j}{\partial _t}\widetilde{u}
} \right) + {h^{ij}}{\nabla _i}\left( {\mathbf{R}({\partial _t}\widetilde{u},{\partial _j}\widetilde{u}){\partial _t}\widetilde{u}} \right).
\end{align*}
Thus we conclude as
\begin{align*}
 {\partial _s}\mathfrak{W} &={h^{ij}}{D_i}{D_j}({D_t}{\phi _t} - {\phi _s}) - {h^{ij}}\Gamma _{ij}^k{D_k}({D_t}{\phi _t} - {\phi _s}) + {h^{ij}}\left( {\mathbf{R}({\partial _t}\widetilde{u},{\partial _i}\widetilde{u}){\nabla _j}{\partial _t}\widetilde{u}} \right) \\
 &+ {h^{ij}}{\nabla _i}\left( {\mathbf{R}({\partial _t}\widetilde{u},{\partial _j}\widetilde{u}){\partial _t}\widetilde{u}} \right)
 - {h^{ij}}\mathbf{R}({\partial _s}\widetilde{u},{\partial _i}\widetilde{u})(\partial_j\widetilde{u}) + {h^{ij}}{\nabla _t}\left( {\mathbf{R}({\partial _t}\widetilde{u},{\partial _i}\widetilde{u})({\partial _j}\widetilde{u})} \right) \\
 &- {h^{ij}}\Gamma _{ij}^k\mathbf{R}({\partial _t}\widetilde{u},{\partial _k}\widetilde{u})\left( {{\partial _t}\widetilde{u}} \right) + \mathbf{R}({\partial _s}\widetilde{u},{\partial _t}\widetilde{u}){\partial _t}\widetilde{u}.
\end{align*}
Using $\mathfrak{W}=D_t\phi_t-\phi_s$ and (\ref{2.4best}) yields
\begin{align}
 {\partial _s}\mathfrak{W}& = \Delta \mathfrak{W} + 2h^{ii}{A_i}{\partial _i}\mathfrak{W}+ h^{ii}{A_i}{A_i}\mathfrak{W} + h^{ii}{\partial _i}{A_i}\mathfrak{W} - {h^{ij}}\Gamma _{ij}^k{A_k}\mathfrak{W}\nonumber\\
 &+ \left\{ { - {h^{ii}}\left( {{\partial _s}\widetilde{u} \wedge {\partial _i}\widetilde{u}} \right){\partial _i}\widetilde{u} + {h^{ii}}({\nabla _t}{\partial _t}\widetilde{u } \wedge {\partial _i}\widetilde{u}){\partial _i}\widetilde{u}} \right\}\nonumber \\
 &+ {h^{ii}}({\partial _t}\widetilde{u} \wedge {\nabla _t}{\partial _i}\widetilde{u}){\partial _i}\widetilde{u} + {h^{ii}}({\partial _t}\widetilde{u }\wedge {\partial _i}\widetilde{u}){\nabla _t}{\partial _i}\widetilde{u}\nonumber \\
 &+ {h^{ii}}({\nabla _i}{\partial _t}\widetilde{u} \wedge {\partial _i}\widetilde{u}){\partial _t}\widetilde{u} + {h^{ii}}({\partial _t}\widetilde{u} \wedge {\partial _i}\widetilde{u}){\nabla _i}{\partial _t}\widetilde{u}\nonumber\\
 &+ \left\{ {{h^{ii}}({\partial _t}\widetilde{u} \wedge {\nabla _i}{\partial _i}\widetilde{u}){\partial _t}\widetilde{u} - {h^{ii}}\Gamma _{ii}^k({\partial _t}\widetilde{u} \wedge {\partial _k}\widetilde{u}){\partial _t}\widetilde{u} + ({\partial _s}\widetilde{u} \wedge {\partial _t}\widetilde{u}){\partial _t}\widetilde{u}} \right\}.\label{wanxiao4}
\end{align}
Recalling the facts that $\mathfrak{W}$ is the gauged field for $\nabla_t\partial_t \widetilde{u}-\tau(\widetilde{u})$ and $\partial_s\widetilde{u}=\tau(\widetilde{u})$, we have
\begin{align}
- {h^{ii}}\left( {{\partial _s}\widetilde{u} \wedge {\partial _i}\widetilde{u}} \right){\partial _i}\widetilde{u} + {h^{ii}}({\nabla _t}{\partial _t}\widetilde{u} \wedge {\partial _i}\widetilde{u}){\partial _i}\widetilde{u} &= {h^{ii}}\left( {({\nabla _t}{\partial _t}\widetilde{u} - {\partial _s}\widetilde{u}) \wedge {\partial _i}\widetilde{u}} \right){\partial _i}\widetilde{u}\nonumber\\
&= {h^{ii}}\left( {\mathfrak{W} \wedge {\phi _i}} \right){\phi _i}.\label{wanxiao1}
\end{align}
Meanwhile, $\partial_s\widetilde{u}=\tau(\widetilde{u})$ also implies
\begin{align}
 &{h^{ii}}({\partial _t}\widetilde{u} \wedge {\nabla _i}{\partial _i}\widetilde{u}){\partial _t}\widetilde{u} - {h^{ii}}\Gamma _{ii}^k({\partial _t}\widetilde{u} \wedge {\partial _k}\widetilde{u}){\partial _t}\widetilde{u} + ({\partial _s}\widetilde{u} \wedge {\partial _t}\widetilde{u}){\partial _t}\widetilde{u }\nonumber\\
 &= {h^{ii}}\left( {{\partial _t}\widetilde{u} \wedge \left( {\tau (\widetilde{u}) - {\partial _s}\widetilde{u}} \right)} \right){\partial _t}\widetilde{u} = 0.\label{wanxiao2}
\end{align}
Bianchi identity gives
\begin{align}\label{wanxiao3}
{h^{ii}}({\partial _t}\widetilde{u} \wedge {\nabla _t}{\partial _i}\widetilde{u}){\partial _i}\widetilde{u} + {h^{ii}}({\nabla _i}{\partial _t}\widetilde{u} \wedge {\partial _i}\widetilde{u}){\partial _t}\widetilde{u} &=  - {h^{ii}}({\partial _i}\widetilde{u }\wedge {\partial _t}\widetilde{u}){\nabla _t}{\partial _i}\widetilde{u}\nonumber\\
&= {h^{ii}}({\partial _t}\widetilde{u} \wedge {\partial _i}\widetilde{u}){\nabla _t}{\partial _i}\widetilde{u}.
\end{align}
By (\ref{wanxiao1}), (\ref{wanxiao2}) and (\ref{wanxiao3}), (\ref{wanxiao4}) can be further simplified as
\begin{align*}
{\partial _s}\mathfrak{W} =& \Delta \mathfrak{W} + 2h^{ii}{A_i}{\partial _i}\mathfrak{W} + h^{ii}{A_i}{A_i}\mathfrak{W} +h^{ii} {\partial _i}{A_i}\mathfrak{W} - {h^{ii}}\Gamma _{ii}^k{A_k}\mathfrak{W}\\
& + {h^{ii}}\left( {\mathfrak{W} \wedge {\phi _i}} \right){\phi _i}+ 3{h^{ii}}({\partial _t}\widetilde{u} \wedge {\partial _i}\widetilde{u}){\nabla _t}{\partial _i}\widetilde{u}.
\end{align*}
\end{proof}

\begin{lemma}\label{xuejin}
Let $Q$ be an admissible harmonic map in Definition 1.1. Fix the frame $\Xi$ in Remark \ref{3sect} by taking $\Xi(Q(x))=\Theta(Q(x))$ given by (\ref{vg}).
Recall the definitions of $A^{\infty}_i$ in Lemma \ref{po87bg}. Then
\begin{align}
&|A_i^{\infty}|\lesssim|dQ|, |\sqrt{h^{ii}}\phi_i^{\infty}|\lesssim |dQ|\label{kulun1}\\
&|{h^{ii}}\left( {{\partial _i}{A^{\infty}_i} - \Gamma _{ii}^k{A^{\infty}_k}} \right)|\lesssim  |dQ|^2.\label{kulun}
\end{align}
\end{lemma}
\begin{proof}
Recall the definition
$$[A_i^\infty ]_k^j = \left\langle {{\nabla _i}{\Theta _k},{\Theta _j}} \right\rangle ,{\Theta _1} = {e^{{Q^2}(x)}}\frac{\partial }{{\partial {y_1}}},{\Theta _2} = \frac{\partial }{{\partial {y_2}}}.
$$
Since $A_i$ is skew-symmetric, it suffices to consider the $[A_i]^1_{2}$ terms. Direct calculation gives
\begin{align*}
[A_1^\infty ]_2^1 = \left\langle {{\nabla _1}{\Theta _2},{\Theta _1}} \right\rangle&  = {e^{{Q^2}(x)}}\frac{{\partial {Q^{_k}}}}{{\partial {x_1}}}\left\langle {{\nabla _{\frac{\partial }{{\partial {y_k}}}}}\frac{\partial }{{\partial {y_2}}},\frac{\partial }{{\partial {y_1}}}} \right\rangle  = {e^{ - {Q^2}(x)}}\frac{{\partial {Q^{_k}}}}{{\partial {x_1}}}\overline{\Gamma}_{k2}^1 \\
&=  - {e^{ - {Q^2}(x)}}\frac{{\partial {Q^1}}}{{\partial {x_1}}},
\end{align*}
and similarly we obtain
\begin{align*}
[A_1^\infty ]_1^2 ={e^{ - {Q^2}(x)}}\frac{{\partial {Q^1}}}{{\partial {x_1}}};\mbox{  }
[A_2^\infty ]_1^2 =  - [A_2^\infty ]_2^1 = {e^{ - {Q^2}(x)}}\frac{{\partial {Q^1}}}{{\partial {x_2}}}.
\end{align*}
Thus one has
\begin{align}
 &{h^{ii}}\left( {{\partial _i}[A_i^\infty ]_2^1 - \Gamma _{ii}^k[A_k^\infty ]_2^1} \right)\nonumber \\
 &=  - \left( {\frac{{{\partial ^2}{Q^1}}}{{\partial {x_2}^2}}{e^{2{x_2}}} + \frac{{{\partial ^2}{Q^1}}}{{\partial {x_1}^2}} - \frac{{\partial {Q^1}}}{{\partial {x_1}}}\frac{{\partial {Q^2}}}{{\partial {x_1}}}{e^{2{x_2}}} - \frac{{\partial {Q^1}}}{{\partial {x_2}}}\frac{{\partial {Q^2}}}{{\partial {x_2}}} - \frac{{\partial {Q^1}}}{{\partial {x_2}}}} \right){e^{ - {Q^2}(x)}} \label{chumen}
\end{align}
Writing the harmonic map equation for $Q$ in the coordinate (\ref{vg}) shows for $l=1,2$
$${h^{ii}}\frac{{{\partial ^2}{Q^l}}}{{\partial {x_i}^2}} - {h^{ii}}\Gamma _{ii}^k{\partial _k}{Q^l} + {h^{ii}}\bar \Gamma _{pq}^l\frac{{\partial {Q^p}}}{{\partial {x_i}}}\frac{{\partial {Q^q}}}{{\partial {x_i}}} = 0.
$$
Let $l=1$ in the above equation, we have
$${e^{2{x_2}}}\frac{{{\partial ^2}{Q^1}}}{{\partial {x_1}^2}} + \frac{{{\partial ^2}{Q^1}}}{{\partial {x_2}^2}} - \frac{{\partial {Q^1}}}{{\partial {x_2}}} - 2{e^{2{x_2}}}\frac{{\partial {Q^1}}}{{\partial {x_1}}}\frac{{\partial {Q^2}}}{{\partial {x_1}}} - 2\frac{{\partial {Q^1}}}{{\partial {x_2}}}\frac{{\partial {Q^2}}}{{\partial {x_2}}} = 0,
$$
which combined with (\ref{chumen}) yields
\begin{align}\label{chunmen2}
{h^{ii}}\left( {{\partial _i}A_i^\infty  - \Gamma _{ii}^kA_k^\infty } \right) = \left( {{e^{2{x_2}}}\frac{{\partial {Q^1}}}{{\partial {x_1}}}\frac{{\partial {Q^2}}}{{\partial {x_1}}} + \frac{{\partial {Q^1}}}{{\partial {x_2}}}\frac{{\partial {Q^2}}}{{\partial {x_2}}}} \right){e^{ - {Q^2}(x)}}.
\end{align}
Writing the energy density in coordinates (\ref{vg}), we obtain
\begin{align*}
{\left| {dQ} \right|^2} &= {h^{ij}}\left\langle {\frac{{\partial {Q^k}}}{{\partial {x_i}}}\frac{\partial }{{\partial {y_k}}},\frac{{\partial {Q^k}}}{{\partial {x_j}}}\frac{\partial }{{\partial {y_k}}}} \right\rangle \\
&= {e^{2{x_2}}}{\left| {\frac{{\partial {Q^1}}}{{\partial {x_1}}}} \right|^2}{e^{ - 2{Q_2}}} + {e^{2{x_2}}}{\left| {\frac{{\partial {Q^2}}}{{\partial {x_1}}}} \right|^2} + {\left| {\frac{{\partial {Q^1}}}{{\partial {x_2}}}} \right|^2}{e^{ - 2{Q_2}}} + {\left| {\frac{{\partial {Q^2}}}{{\partial {x_2}}}} \right|^2}.
\end{align*}
Thus (\ref{kulun}) follows by (\ref{chunmen2}) and Young inequality. (\ref{kulun1}) is much easier and follows immediately by the same arguments.
\end{proof}

Now we separate the main term in the equation of $\phi_s$. Recall the limit of $A_{s,t,x}$ given in (\ref{kji}), (\ref{kji22}), one can easily see the main term of (\ref{heating}) is a magnetic wave equation. Precisely, we have the following lemma.
\begin{lemma}\label{hushuo}
Fix the frame $\Xi$ in Proposition 3.3 by letting $\Xi_i(x)=\Theta_i(Q(x))$, $i=1,2$. Then
the heat tension filed $\phi_s$ satisfies
\begin{align*}
&(\partial^2_t-\Delta){\phi _s} + W{\phi _s}\\
&=  - 2{A_t}{\partial _t}{\phi _s} - {A_t}{A_t}{\phi _s} - {\partial _t}{A_t}{\phi _s} + {\partial _s}w + \mathbf{R}({\partial _t}\widetilde{u},{\partial _s}\widetilde{u})({\partial _t}\widetilde{u}) + 2{h^{ii}}A_i^{con}{\partial _i}{\phi _s} \\
&+ {h^{ii}}A_i^{con}A_i^\infty {\phi _s} + {h^{ii}}A_i^\infty A_i^{con}{\phi _s} + {h^{ii}}A_i^{con}A_i^{con}{\phi _s} + {h^{ii}}\left( {{\partial _i}A_i^{con} - \Gamma _{ii}^kA_k^{con}} \right){\phi _s} \\
&+ {h^{ii}}\left( {{\phi _s} \wedge \phi _i^\infty } \right)\phi _i^{con} + {h^{ii}}\left( {{\phi _s} \wedge \phi _i^{con}} \right)\phi _i^\infty  + {h^{ii}}\left( {{\phi _s} \wedge \phi _i^{con}} \right)\phi _i^{con},
\end{align*}
where $A_{x}^{\infty}$, $A_{x}^{con}$ are defined in Remark \ref{3sect}, and $W$ is given by
\begin{align}\label{iuo9}
W\varphi  =  -2 {h^{ii}}A_i^\infty {\partial _i}\varphi  -{h^{ii}}A_i^\infty A_i^\infty \varphi  - {h^{ii}}\left( {\varphi  \wedge \phi _i^\infty } \right)\phi _i^\infty-h^{ii}(\partial_iA^{\infty}_{i}-\Gamma^k_{ii}A^{\infty}_k).
\end{align}
Furthermore, $-\Delta+W$ is a self-adjoint operator in $L^2(\Bbb H^2;\Bbb C^2)$. And it is strictly positive if $0<\mu_1\ll1$.
\end{lemma}
\begin{proof}
By (\ref{heating}), expanding $D_{x,t}$ as $\partial_{t,x}+A_{t,x}$ implies
\begin{align}
&\partial _t^2{\phi _s} - {\Delta}{\phi _s}\nonumber\\
& =  - 2{A_t}{\partial _t}{\phi _s} - {A_t}{A_t}{\phi _s} - {\partial _t}{A_t}{\phi _s} + {h^{ii}}{A_i}{A_i}{\phi _s} + {h^{ii}}\left( {{\partial _i}{A_i} - \Gamma _{ii}^k{A_k}} \right){\phi _s}  \nonumber\\
&+2 {h^{ii}}{A_i}{\partial _i}{\phi _s}+ {\partial _s}\mathfrak{W} + {h^{ii}}\mathbf{R}({\partial _s}\widetilde{u},{\partial _i}\widetilde{u})({\partial _i}\widetilde{u}) + \mathbf{R}({\partial _t}\widetilde{u},{\partial _s}\widetilde{u})({\partial _t}\widetilde{u}).\label{huta}
\end{align}
By Remark \ref{3sect}, $A_i=A^{\infty}_i+A^{con}_i$, $\phi_i=\phi^{\infty}_{i}+\phi^{con}_i$. Then fixing $\Xi$ to be $(\Theta_1(Q),\Theta_2(Q))$, we have
(\ref{huta}) reduces to
\begin{align*}
 &\partial _t^2{\phi _s} - {\Delta}{\phi _s} - 2{h^{ii}}A_i^\infty {\partial _i}{\phi _s} - {h^{ii}}A_i^\infty A_i^\infty {\phi _s} - {h^{ii}}\left( {{\phi _s} \wedge \phi _i^\infty } \right)\phi _i^\infty-h^{ii}(\partial_iA^{\infty}_i-\Gamma^k_{ii}A^{\infty}_k)\phi_s \\
 &=  - 2{A_t}{\partial _t}{\phi _s} - {A_t}{A_t}{\phi _s} - {\partial _t}{A_t}{\phi _s} + {\partial _s}w +(\phi _t\wedge\phi_s)\phi_t + {h^{ii}}A_i^{con}{\partial _i}{\phi _s} + {h^{ii}}A_i^{con}A_i^\infty {\phi _s}\\
 &+ {h^{ii}}A_i^\infty A_i^{con}{\phi _s}+ {h^{ii}}A_i^{con}A_i^{con}{\phi _s} + {h^{ii}}\left( {{\partial _i}A_i^{con} - \Gamma _{ii}^kA_k^{con}} \right){\phi _s} + {h^{ii}}\left( {{\phi _s} \wedge \phi _i^\infty } \right)\phi _i^{con}\\
 &+ {h^{ii}}\left( {{\phi _s} \wedge \phi _i^{con}} \right)\phi _i^\infty  + {h^{ii}}\left( {{\phi _s} \wedge \phi _i^{con}} \right)\phi _i^{con}.
\end{align*}
Then from the non-negativeness of the sectional curvature for the target $N=\Bbb H^2$ and the skew-symmetry of the connection matrix $A^{\infty}_i$, we have $W$ is a nonnegative symmetric operator in  $L^2(\Bbb H^2;\Bbb C^2)$ by direct calculations, see Lemma \ref{symm} in Section 7.
The self-adjointness of $W$ follows from Kato's perturbation theorem.  In fact, there exists a self-adjoint realization denote by $((\Delta_{col}),D(\Delta_{col}))$ of $(\Delta,C^{\infty}_c(\Bbb H^2,\Bbb C^2))$. It is known that $D(\Delta_{col})$ consists of functions $f\in L^2$ whose Laplacian $\Delta f$ in distribution sense belong to $L^2$, see for instance \cite{Stri}. Write $W$ as $W=V_1+V_2\nabla$, then $V_1$ and $V_2$ are of exponential decay as $d(x,0)\to\infty$ by Lemma  \ref{xuejin} and Definition 1.1. For any fixed $\varepsilon>0$, take $R>0$ sufficiently large such that
$${\left\| {{V_1}(x)} \right\|_{L_{d(x,0) \ge R}^\infty }} \le \varepsilon ,\mbox{  }{\left\| {{V_2}(x)} \right\|_{L_{d(x,0) \ge R}^\infty }} \le \varepsilon,
$$
then for any $f\in C^{\infty}_c(\Bbb H^2,\Bbb C^2)$,
\begin{align}\label{huli}
{\left\| {{V_1}(x)f + {V_2}\nabla f} \right\|_{L_{d(x,0) \ge R}^2}} \le \varepsilon {\left\| f \right\|_{{L^2}}} + \varepsilon {\left\| {\nabla f} \right\|_{{L^2}}}.
\end{align}
For this $R$, the compactness of Sobolev embedding in bounded domains implies there exists $C(\varepsilon, R)$ such that
\begin{align}\label{huli2}
{\left\| {{V_1}(x)f + {V_2}\nabla f} \right\|_{L_{d(x,0) \le R}^2}} \le C(\varepsilon ,R){\left\| f \right\|_{{L^2}}} + \varepsilon {\left\| {\Delta f} \right\|_{{L^2}}}.
\end{align}
Hence by (\ref{huli}) and (\ref{huli2}), one has for any $\varepsilon>0$ there exists $C(\varepsilon)$ such that
\begin{align}
{\left\| {{V_1}(x)f + {V_2}\nabla f} \right\|_{{L^2}}} \le C(\varepsilon){\left\| f \right\|_{{L^2}}} + \varepsilon {\left\| {\Delta f} \right\|_{{L^2}}}.
\end{align}
Since $C^{\infty}_c(\Bbb H^2,\Bbb C^2)$ is a core of $\Delta_{col}$, Kato's compact perturbation theorem shows $-\Delta+W$ is self-adjoint in  $L^2$ with domain $D(\Delta_{col})$.
\end{proof}

\section{Bootstrap for the heat tension filed}

\subsection{Strichartz estimates for wave equation with magnetic potential}
Theorem 5.2 and Remark 5.5 of Anker, Pierfelice \cite{AP} obtained the Strichartz estimates for linear wave/Klein-Gordon equation: Let $((p,q),(\widetilde{p},\widetilde{q}))$ be a $(\sigma,\tilde{\sigma})$ admissible couple, i.e.,
\begin{align*}
&\left\{ {({p^{ - 1}},{q^{ - 1}}) \in (0,\frac{1}{2}] \times (0,\frac{1}{2}):\frac{1}{p} > \frac{1}{2}(\frac{1}{2} - \frac{1}{q})} \right\} \cup \left\{ {\left( {0,\frac{1}{2}} \right)} \right\}\\
&\sigma  \ge \frac{3}{2}\left( {\frac{1}{2} - \frac{1}{q}} \right),\tilde \sigma  \ge \frac{3}{2}\left( {\frac{1}{2} - \frac{1}{{\tilde q}}} \right).
\end{align*}
If $u$ solves $\partial_t^2u-\Delta u=g$ with initial data $(f_0,f_1)$, then
$$ {\left\| {\widetilde D_x^{ - \sigma  + \frac{1}{2}}u} \right\|_{L_t^pL_x^q}} + {\left\| {\widetilde D_x^{ - \sigma  - \frac{1}{2}}{\partial _t}u} \right\|_{L_t^pL_x^q}} \lesssim {\left\| {\widetilde D_x^{\frac{1}{2}}f_0} \right\|_{{L^2}}} + {\left\| {\widetilde D_x^{ - \frac{1}{2}}f_1} \right\|_{{L^2}}} + {\left\| {\widetilde D_x^{\tilde \sigma  - \frac{1}{2}}g} \right\|_{L_t^{\tilde p'}L_x^{\tilde q'}}}.
$$
where $\widetilde D=(-\Delta-\frac{1}{4}+\kappa^2)$ for some $\kappa>\frac{1}{2}$.

Let $\rho(x)=e^{-d(x,0)}$. The endpoint and non-endpoint Strichartz estimates for magnetic wave equations in the small potential case were obtained in the first author's work [Corollary. 1.1. Proposition 3.1 \cite{Lize1}]. We recall this for reader's convenience.
Consider the magnetic wave equation on $\Bbb H^2$,
\begin{align}\label{wavem}
\left\{ \begin{array}{l}
 \partial _t^2f - \Delta f + {B_0}(x)f + \sum^2_{i=1}{h^{ii}}{B_i}(x){\partial _i}f = F \\
 f(0,x) = {f_0}(x),{\partial _t}f(0,x) = {f_1}(x) \\
 \end{array} \right.
\end{align}

\begin{lemma}[\cite{Lize1}]\label{poi9nmb}
Assume that $B_0,B_1,B_2$ in (\ref{wavem}) satisfy for some $\varrho>0$
\begin{align}
\|B_0\|_{L^2\cap e^{-r\varrho}L^{\infty}}+\sum^2_{i=1}\|\sqrt{h^{ii}}B_i\|_{L^2\cap e^{-r\varrho}L^{\infty}}\le \mu_1.
\end{align}
And assume that the Schr\"odinger operator $H=-\Delta+B_0+h^{ii}B_i\partial_i$ is symmetric. If $0<\mu_1\ll 1$, $u$ solves (\ref{wavem}), then for any $0<\sigma\ll\varrho$, $p\in(2,6)$
\begin{align*}
&{\left\| {{\rho ^{\sigma}}\nabla f} \right\|_{L_t^2L_x^2}}+{\left\| (-\Delta)^{\frac{1}{4}}f \right\|_{L_t^2L_x^{p}}} +
{\left\| {{\partial _t}f} \right\|_{L_t^\infty L_x^2}} + {\left\| {\nabla f} \right\|_{L_t^\infty L_x^2}}\\
&\lesssim {\left\| {\nabla {f_0}} \right\|_{{L^2}}} + {\left\| {{f_1}} \right\|_{{L^2}}} + {\left\| F \right\|_{L_t^1L_x^2}}.
\end{align*}
\end{lemma}

Hence by Lemma \ref{xuejin}, Lemma \ref{hushuo} and Lemma \ref{poi9nmb}, we have:
\begin{proposition}\label{tianxia}
Let $W$ be defined above and $0<\mu_1\ll 1$, $0<\sigma\ll \varrho\ll 1$,  then we have the weighted and endpoint Strichartz estimates for the magnetic wave equation: If $f$ solves the equation
\begin{align*}
\left\{ \begin{array}{l}
\partial _t^2f - {\Delta}f + Wf = F \\
f(0,x) = {f_0},{\partial _t}f(0,x) = {f_1} \\
\end{array} \right.
\end{align*}
then it holds for any $p\in(2,6)$, $0<\sigma\ll \varrho$
\begin{align}
&{\left\| {{{\left| D \right|}^{\frac{1}{2}}}f} \right\|_{L_t^2L_x^{p}}} +{\left\| {{\rho ^{\sigma}}\nabla f} \right\|_{L_t^2L_x^2}}+
{\left\| {{\partial _t}f} \right\|_{L_t^\infty L_x^2}} + {\left\| {\nabla f} \right\|_{L_t^\infty L_x^2}}+{\left\| {{\rho ^{\sigma }}\nabla f} \right\|_{L_t^2L_x^2}}\nonumber\\
&\lesssim {\left\| {\nabla {f_0}} \right\|_{{L^2}}} + {\left\| {{f_1}} \right\|_{{L^2}}} + {\left\| F \right\|_{L_t^1L_x^2}}.\label{gutu4}
\end{align}
\end{proposition}

\begin{remark}\label{tataru0}
For all $\sigma\in \Bbb R$, $p\in(1,\infty)$, $\|\widetilde D^{\sigma} f\|_{p}$ is equivalent to $\|(-\Delta)^{\sigma/2}f\|_{p}$.
Tataru \cite{Tataru4} shows for all $p\in(1,\infty)$, $\|\Delta f\|_{p}$ is equivalent to $\|\nabla^2f\|_{p}+\|\nabla f\|_{p}+\|f\|_{p}$.
\end{remark}

\subsection{Setting of Bootstrap}

We fix the constants $\mu_1,\varepsilon_1, \varrho, \sigma$ to be
\begin{align}\label{dozuoki}
0<\mu_2<\mu_1\ll\varepsilon_1\ll 1,\mbox{ }0<\sigma\ll\varrho\ll 1.
\end{align}
Let $L>0$ be sufficiently large say $L=100$.
Define $\omega:\Bbb R^+\to \Bbb R^+$ and $a:\Bbb R^+\to \Bbb R^+$ by
$$\omega(s) = \left\{ \begin{array}{l}
 {s^{\frac{1}{2}}}\mbox{  }{\rm{when}}\mbox{  }0 \le s \le 1 \\
 {s^L}\mbox{  }\mbox{  }{\rm{when}}\mbox{  }s \ge 1 \\
 \end{array} \right.,a(s) = \left\{ \begin{array}{l}
 s^{\frac{3}{4}}\mbox{  }\mbox{  }{\rm{when}}\mbox{  }0 \le s \le 1 \\
 {s^L}\mbox{  }\mbox{  }{\rm{when}}\mbox{  }s \ge 1 \\
 \end{array} \right.
$$

\begin{proposition}\label{oures}
Assume that $\mathcal{A}$ is the set of $T\in[0,T_*)$ such that for any $2<q<6+2\gamma$, $p\in(2,6)$ with some fixed $0<\gamma\ll1$,
\begin{align}
{\left\| {(du,\partial_tu)} \right\|_{L_t^\infty L_x^2([0,T] \times {\Bbb H^2})}}
+ {\left\| (\nabla{\partial _t}u,\nabla du) \right\|_{L_t^\infty L_x^2([0,T] \times {\Bbb H^2})}}\nonumber\\
+{\left\| {{\partial _t}u} \right\|_{L_t^2L^q_x([0,T] \times {\Bbb H^2})}} &\le {\varepsilon _1}.\label{boot2}\\
{\left\| {\omega (s)|D{|^{ - \frac{1}{2}}}{\partial _t}{\phi _s}} \right\|_{L_s^\infty L_t^2L_x^p}} + {\left\| {\omega (s){\partial _t}{\phi _s}} \right\|_{L_s^\infty L_t^\infty L_x^2}}\nonumber\\
+{\left\| {\omega (s)\nabla {\phi _s}} \right\|_{L_s^\infty L_t^\infty L_x^2}} + {\left\| {\omega (s){{\left| D \right|}^{\frac{1}{2}}}{\phi _s}} \right\|_{L_s^\infty L_t^2L_x^p}} &\le {\varepsilon _1}.\label{boot5}
\end{align}
Then for all $T\in \mathcal{A}$ we have
\begin{align}
&{\left\| {\omega (s){{\left| D \right|}^{ - \frac{1}{2}}}{\partial _t}{\phi _s}} \right\|_{L_s^\infty L_t^2L_x^p([0,T] \times {\Bbb H^2})}} + {\left\| {\omega (s){{\left| D \right|}^{\frac{1}{2}}}{\phi _s}} \right\|_{L^{\infty}_sL_t^2L_x^p([0,T] \times {\Bbb H^2})}}\nonumber\\
&+ {\left\| {\omega (s){\partial _t}{\phi _s}} \right\|_{L_s^\infty L_t^\infty L_x^2([0,T] \times {\Bbb H^2})}}
+ {\left\| {\omega (s)\nabla {\phi _s}} \right\|_{L_s^\infty L_t^\infty L_x^2([0,T] \times {H^2})}} \le \varepsilon _1^2. \label{huojiq}
\end{align}
and for any $r\in(2,6+2\gamma]$ it holds that
\begin{align}
{\left\| {(du,\partial_tu)} \right\|_{L_t^\infty L_x^2([0,T] \times {\Bbb H^2})}} + {\left\| {(\nabla{\partial _t}u,\nabla du)} \right\|_{L_t^\infty L_x^2([0,T] \times {\Bbb H^2})}}&\le {\varepsilon^2 _1}\label{boot8q}\\
{\left\| {{\partial _t}u} \right\|_{L_t^2L_x^r([0,T] \times {\Bbb H^2})}} &\le {\varepsilon^2_1}.\label{boot9q}
\end{align}
Moreover we have
\begin{align}
 {\left\| {du} \right\|_{L_t^\infty L_x^2([0,{T}] \times {\Bbb H^2})}}& + {\left\| {{\partial _t}u} \right\|_{L_t^\infty L_x^2([ 0,{T}] \times {\Bbb H^2})}} + {\left\| {\nabla du} \right\|_{L_t^\infty L_x^2([0,{T}] \times {\Bbb H^2})}} \nonumber\\
 &+ {\left\| {\nabla {\partial _t}u} \right\|_{L_t^\infty L_x^2([0,{T}] \times {\Bbb H^2})}} + {\left\| {{\partial _t}u} \right\|_{L_t^2L_x^6([ 0,{T}] \times {\Bbb H^2})}} \le \varepsilon _1^2.\label{hua12xde}
\end{align}
\end{proposition}
The proof of Proposition \ref{oures} will be divided into several lemmas below. (\ref{huojiq}) is proved in Proposition 5.11. (\ref{boot8q}), (\ref{boot9q}) and (\ref{hua12xde}) are proved in Proposition 5.13 and Corollary 5.15 respectively.

The bootstrap programm we apply here is based on the design of \cite{LOS,Tao7}. The essential refinement is we add a spacetime bound $\|\partial_tu\|_{L^2_tL^p_x}$ to the primitive bootstrap assumption.  The most important original ingredient in this part is we use the weighted Strichartz estimates in Section 5.1 to control the one order derivative terms of $\phi_s$.

\begin{proposition}\label{aaop}
Assume (\ref{boot2}) holds, then we have for any $\eta>0$
\begin{align}
{\big\| {{A_t}} \big\|_{L_t^\infty L_x^\infty }} &\le \varepsilon_1\label{aaaw811}\\
{\big\| {{h^{ii}} {\partial _i}{A_i}(s)} \big\|_{L_t^\infty L_x^\infty }} &\le \varepsilon_1 \max(1,{s^{ -\eta}})\label{butterfly}\\
{\big\| {\sqrt {{h^{ii}}} {\partial _t}{A_i}(s)} \big\|_{L_t^\infty L_x^\infty }} &\le \varepsilon_1 {s^{ - \frac{1}{2}}}\label{u82}\\
{\big\| {\sqrt {{h^{ii}}} {A_i}(s)} \big\|_{L_t^\infty L_x^\infty }} &\le \varepsilon_1 \label{u81}.
\end{align}
\end{proposition}
\begin{proof}
By the commutator identity and the facts $|\partial_t\widetilde{u}|\le e^{s\Delta}|\partial_tu|$, $|\partial_s\widetilde{u}|\le e^{s\Delta}|\partial_su|$, (\ref{aaaw811}) is bounded by Lemma \ref{ktao1},
\begin{align*}
{\left\| {{A_t}} \right\|_{L_t^\infty L_x^\infty }} &\le {\big\| {\int_s^\infty  {{{\big\| {{\phi _t}} \big\|}_{L_x^\infty }}{{\big\| {{\phi _s}} \big\|}_{L_x^\infty }}} d\kappa} \big\|_{L_t^\infty }} \le \mathop {\sup }\limits_{t \in [0,T]} {\big\| {{\phi _t}} \big\|_{L_s^2L_x^\infty }}{\big\| {{\phi _s}} \big\|_{L_s^2L_x^\infty }} \\
&\le \mathop {\sup }\limits_{t \in [0,T]} {\big\| {{\partial _t}u} \big\|_{L_x^2}}{\big\| {{\partial _s}u} \big\|_{L_x^2}} \le {\varepsilon _1}.
\end{align*}
By the commutator identity,
\begin{align*}
 &{\big\| {\sqrt {{h^{ii}}} {\partial _t}{A_i}} \big\|_{L_t^\infty L_x^\infty }} \le \int_s^\infty  {{{\big\| {\sqrt {{h^{ii}}} {\partial _t}\left( {{\phi _i} \wedge {\phi _s}} \right)} \big\|}_{L_t^\infty L_x^\infty }}} d\kappa\\
 & \le \int_s^\infty  {{{\big\| {\sqrt {{h^{ii}}} {\partial _t}{\phi _i}} \big\|}_{L_t^\infty L_x^\infty }}\big\| {{\phi _s}} \big\|_{L_t^\infty L^{\infty}_x }}  d\kappa+ \int_s^\infty  {{{\big\| {\sqrt {{h^{ii}}} {\phi _i}} \big\|}_{L_t^\infty L_x^\infty }}\big\| {{\partial _t}{\phi _s}} \big\|_{L_t^\infty L^{\infty}_x}}  d\kappa.
\end{align*}
Using the relation between the induced derivative $D_{i,t}$ and the covariant derivative on $u^*(TN)$, one obtains
$| {\sqrt {{h^{ii}}} {\partial _t}{\phi _i}}| \le | {\nabla {\partial _t}\widetilde{u}}| + | {\sqrt {{h^{ii}}} {A_t}{\phi _i}}|+ | {\sqrt {{h^{ii}}} {A_i}{\phi _t}}|$ and
similarly
$| {{\partial _t}{\phi _s}}| \le | {{\nabla _t}{\partial _s}\widetilde{u}}| + | {A_t}{\phi _s}|$.
Hence it suffices to prove
\begin{align}
{\int_s^\infty  {\left\| {\left| {d\widetilde{u}} \right|\left| {{\nabla _t}{\partial _s}\widetilde{u}} \right|} \right\|} _{L_t^\infty L_x^\infty }}d\kappa + {\int_s^\infty  {\left\| {\left| {{\partial _s}\widetilde{u}} \right|\left| {\nabla {\partial _t}\widetilde{u}} \right|} \right\|} _{L_t^\infty L_x^\infty }}d\kappa&\le \varepsilon_1s^{-\frac{1}{2}}\label{po987}\\
\int_s^\infty  {{{\| {\sqrt {{h^{ii}}} {A_i}{\phi _t}{\phi _s}}\|}_{L_t^\infty L_x^\infty }}}d\kappa+\int_s^\infty  {{{\| {\sqrt {{h^{ii}}} {A_t}{\phi _i}{\phi _s}}\|}_{L_t^\infty L_x^\infty }}}d\kappa&\le \varepsilon_1s^{-\frac{1}{2}} \label{pojn89}
\end{align}
For $s\in(0,1]$, Proposition \ref{sl} and $|d\widetilde{u}|\le e^{s\Delta}|du|$ give
\begin{align}\label{0918}
{\left\| {\left| {d\widetilde{u}} \right|\left| {{\nabla _t}{\partial _s}\widetilde{u}} \right|} \right\|_{L_t^\infty L_x^\infty }} + {\left\| {\left| {{\partial _s}\widetilde{u}} \right|\left| {\nabla {\partial _t}\widetilde{u}} \right|} \right\|_{L_t^\infty L_x^\infty }} \le \varepsilon_1 {s^{ - \frac{1}{2}}}{s^{ -1}} + \varepsilon_1 {s^{ -\frac{1}{2}}}{s^{ - 1}}.
\end{align}
For $s\ge1$, we have by Proposition \ref{sl}
\begin{align}\label{01918}
{\left\| {\left| {d\widetilde{u}} \right|\left| {{\nabla _t}{\partial _s}\widetilde{u}} \right|} \right\|_{L_t^\infty L_x^\infty }} + {\left\| {\left| {{\partial _s}\widetilde{u}} \right|\left| {\nabla {\partial _t}\widetilde{u}} \right|} \right\|_{L_t^\infty L_x^\infty }} \le \varepsilon_1 {e^{ - \delta s}}.
\end{align}
Therefore (\ref{01918}) and (\ref{0918}) yield for all $s\in(0,\infty)$
$${\left\| {\left| {d\widetilde{u}} \right|\left| {{\nabla _t}{\partial _s}\widetilde{u}} \right|} \right\|_{L_t^\infty L_x^\infty }} + {\left\| {\left| {{\partial _s}\widetilde{u}} \right|\left| {\nabla {\partial _t}\widetilde{u}} \right|} \right\|_{L_t^\infty L_x^\infty }} \le \varepsilon_1 {s^{ -3/2}}.
$$
Hence we obtain (\ref{po987}). (\ref{pojn89}) and (\ref{u81}) can be proved similarly. By (\ref{christ}) and direct calculations similar to Lemma \ref{xuejin},
\begin{align}\label{ipu83}
|h^{ii}\partial_iA^{\infty}_i|\lesssim |\nabla dQ|+|dQ|.
\end{align}
And the same route as (\ref{u82}) shows for any $\eta>0$
\begin{align}\label{ipu82}
|h^{ii}\partial_iA^{con}_i|\le \varepsilon_1 s^{-\eta}.
\end{align}
Thus (\ref{butterfly}) follows by (\ref{ipu82}), (\ref{ipu83})
\end{proof}

\begin{lemma}\label{aoao1}
Assume (\ref{boot2}) and (\ref{boot5}) hold, then we have
\begin{align}
\left\| \sqrt{h^{pp}}|\partial_p(h^{ii} {\partial _i}{A_i}(s))| \right\|_{L_t^\infty L_x^{\infty}} &\le \varepsilon_1 \max(s^{ -1},1)\label{1q2}
\end{align}
\end{lemma}
\begin{proof}
By Remark \ref{3sect}, it suffices to bound $A^{\infty}$ and $A^{con}$ part separately. Direct calculations as Lemma \ref{xuejin} and (\ref{christ}) yield the bound for the $A^{\infty}$ part is
\begin{align*}
\sqrt{h^{pp}}|\partial_p(h^{ii} {\partial _i}{A^{\infty}_i}(s))| \le |\nabla^2 dQ|+ |\nabla dQ|+|dQ|.
\end{align*}
Thus (\ref{as4}) shows the $A^{\infty}$ part is bounded by
\begin{align*}
\|\sqrt{h^{pp}}|\partial_p(h^{ii} {\partial _i}{A^{\infty}_i}(s))|\|_{L^{\infty}_x}\le \varepsilon_1.
\end{align*}
By (\ref{christ}) and direct calculations,
\begin{align*}
&\sqrt{h^{pp}}|\partial_p(h^{ii} {\partial _i}({\phi_i\wedge\phi_s})(s))|\\
&\le |\nabla^2\partial_s\widetilde{u}| |d\widetilde{u}|+\sqrt{h^{ii}h^{pp}}|A_iA_p||\partial_s\widetilde{u}| |d\widetilde{u}|
+|\partial_s\widetilde{u}| |\nabla^2d\widetilde{u}|+
|\nabla\partial_s\widetilde{u}||\nabla d\widetilde{u}|\\
&+\sqrt{h^{ii}}|A_i||\nabla\partial_s\widetilde{u}||d\widetilde{u}|+
\sqrt{h^{ii}}|A_i||\partial_s\widetilde{u}||\nabla d\widetilde{u}|+
\sqrt{h^{ii}}|A_i\|\nabla\partial_s\widetilde{u}| |d\widetilde{u}|\\
&+\sqrt{h^{pp}h^{ii}}|\partial_pA_i||\nabla\partial_s\widetilde{u}| |d\widetilde{u}|+\sqrt{h^{pp}h^{ii}}|\partial_pA_i||\partial_s\widetilde{u}| |\nabla d\widetilde{u}|.
\end{align*}
Thus the $A^{con}$ part follows by Lemma \ref{fotuo1} and interpolation.
\end{proof}

\begin{proposition}\label{bootstrap}
Suppose that (\ref{boot2}), (\ref{boot5}) hold.
Then we have for $p\in(2,6)$
\begin{align}
{\big\| {a(s){{\left\| {{\partial _t}{\phi _s}} \right\|}_{L_t^2L_x^p}}} \big\|_{L_s^\infty}}&\le {\varepsilon _1}\label{huojikn1} \\
 {\big\| {a(s){{\left\| {\nabla {\phi _s}} \right\|}_{L_t^2L_x^p}}} \big\|_{L_s^\infty}} &\le {\varepsilon _1}\label{huojikn}.
\end{align}
Generally we have for $\theta\in[0,2]$
\begin{align}
{\big\| {\omega_\theta (s){{\left( { - \Delta } \right)}^\theta }{\phi _s}} \big\|_{L_s^\infty L_t^2L_x^p}}&\le {\varepsilon _1}\label{uojiknmpx3}\\
\big\| \omega_1(s)|D|{\partial _t}{\phi _s} \big\|_{L_s^\infty L_t^2L_x^p}  &\le {\varepsilon _1},\label{9o0o}
\end{align}
where $\omega_\theta (s)=s^{\theta+\frac{1}{4}}$ when $s\in[0,1]$ and $\omega_\theta (s)=s^{L}$ when $s\ge1$.
\end{proposition}
\begin{proof}
By (\ref{991}) and Duhamel principle we have
\begin{align}
{\| {{{\left( { - \Delta } \right)}^{\frac{1}{2}}}{\phi _s}(s)}\|_{L_t^2L_x^p}} &\le {\| {{{\left( { - \Delta } \right)}^{\frac{1}{2}}}{e^{\frac{s}{2}\Delta }}{\phi _s}(\frac{s}{2})} \|_{L_t^2L_x^p}}\nonumber\\
&+ {\big\| {\int_{\frac{s}{2}}^s {{{\left( { - \Delta } \right)}^{\frac{1}{2}}}{e^{(s - \tau )\Delta }}{h^{ii}}{A_i}{\partial _i}{\phi _s}(\tau )} d\tau }\big\|_{L_t^2L_x^p}} \label{wulaso}\\
&+ {\big\| {\int_{\frac{s}{2}}^s {{{\left( { - \Delta } \right)}^{\frac{1}{2}}}{e^{(s - \tau )\Delta }}G(\tau)}d\tau } \big\|_{L_t^2L_x^p}}.\label{wulasoo}
 \end{align}
 where $G(\tau)={{h^{ii}}\left( {{\partial _i}{A_i}} \right){\phi _s} - {h^{ii}}\Gamma _{ii}^k{A_k}{\phi _s} + {h^{ii}}{A_i}{A_i}{\phi _s} + {h^{ii}}\left( {{\phi _s} \wedge {\phi _i}} \right){\phi _i}}.$
For (\ref{wulaso}), the smoothing effect and (\ref{u81}) show
\begin{align*}
 &s^{\frac{3}{4}}{\big\| {\int_{\frac{s}{2}}^s {{{\left( { - \Delta } \right)}^{\frac{1}{2}}}{e^{(s - \tau )\Delta }}{h^{ii}}{A_i}{\partial _i}{\phi _s}(\tau )} d\tau } \big\|_{L_t^2L_x^p}} \\
 &\lesssim s^{\frac{3}{4}}\int_{\frac{s}{2}}^s {{{{{(s - \tau )}^{-\frac{1}{2}}}}}{{\left\| {{h^{ii}}{A_i}{\partial _i}{\phi _s}(\tau )} \right\|}_{L_t^2L_x^p}}} d\tau  \\
 &\lesssim s^{\frac{3}{4}}\int_{\frac{s}{2}}^s {{{{{(s - \tau )}^{-\frac{1}{2}}}}}{{\left\| {\nabla {\phi _s}(\tau )} \right\|}_{L_t^2L_x^p}}{{\big\| {\sqrt {{h^{ii}}} {A_i}} \big\|}_{L_t^\infty L_x^\infty }}d} \tau \\
 &\lesssim s^{\frac{3}{4}}\varepsilon_1\int_{\frac{s}{2}}^s {{{{{(s - \tau )}^{-\frac{1}{2}}}}}{{\left\| {\nabla {\phi _s}(\tau )} \right\|}_{L_t^2L_x^p}}d} \tau.
 \end{align*}
Thus we conclude when $s\in[0,1]$
\begin{align}\label{w2}
&s^{\frac{3}{4}}{\big\| {\int_{\frac{s}{2}}^s {{{\left( { - \Delta } \right)}^{\frac{1}{2}}}{e^{(s - \tau )\Delta }}{h^{ii}}{A_i}{\partial _i}{\phi _s}(\tau )} d\tau } \big\|_{L_t^2L_x^p}}\nonumber\\
&\le \varepsilon_1{\left\| {s^{\frac{3}{4}}{{\left\| {\nabla {\phi _s}(s)} \right\|}_{L_t^2L_x^p}}} \right\|_{L_s^\infty }}.
\end{align}
Similarly we have for (\ref{wulasoo}) that
\begin{align*}
&s^{\frac{3}{4}}\int_{\frac{s}{2}}^s (s - \tau )^{-\frac{1}{2}}\|G(\tau)\|_{L_t^2L_x^p} d\tau  \\
&\le s^{\frac{3}{4}}\int_{\frac{s}{2}}^s {{{{{(s - \tau )}^{-\frac{1}{2}}}}}{{\left\| {{h^{ii}}{\partial _i}{A_i}} \right\|}_{L_t^\infty L_x^\infty }}{{\left\| {{\phi _s}} \right\|}_{L_t^2L_x^p}}} d\tau  + s^{\frac{3}{4}}\int_{\frac{s}{2}}^s {{{\left\| {{A_2}} \right\|}_{L_t^\infty L_x^\infty }}{{\left\| {{\phi _s}} \right\|}_{L_t^2L_x^p}}} d\tau  \\
&+ s^{\frac{3}{4}}\int_{\frac{s}{2}}^s {{{{{(s - \tau )}^{-\frac{1}{2}}}}}\left( {{{\left\| {{h^{ii}}{A_i}{A_i}} \right\|}_{L_t^\infty L_x^\infty }} + {{\left\| {{h^{ii}}{\phi _i}{\phi _i}} \right\|}_{L_t^\infty L_x^\infty }}} \right)} {\left\| {{\phi _s}} \right\|_{L_t^2L_x^p}}d\tau.
\end{align*}
Thus by Proposition \ref{aaop} and Proposition \ref{sl}, we have for all $s\in[0,1]$
\begin{align}\label{wulaso2}
(\ref{wulasoo})\lesssim  {\left\| {{s^{\frac{1}{2}}}{{\left\| {{\phi _s}(s)} \right\|}_{L_t^2L_x^p}}} \right\|_{L_s^\infty }}.
\end{align}
For $s\ge1$, we also have by Duhamel principle
\begin{align*}
&{s^L}{\big\| {{{\left( { - \Delta } \right)}^{\frac{1}{2}}}{\phi _s}(s)} \big\|_{L_t^2L_x^p}}\\
&\le {s^L}{\big\| {{{\left( { - \Delta } \right)}^{\frac{1}{2}}}{e^{\frac{s}{2}\Delta }}{\phi _s}(\frac{s}{2})} \big\|_{L_t^2L_x^6}} + {s^L}{\big\| {\int_{\frac{s}{2}}^s {{{\left( { - \Delta } \right)}^{\frac{1}{2}}}{e^{(s - \tau )\Delta }}G_1(\tau )} d\tau } \big\|_{L_t^2L_x^p}},
\end{align*}
where $G_1$ is the inhomogeneous term. The linear term is bounded by
\begin{align*}
 {s^L}{\big\| {{{\left( { - \Delta } \right)}^{\frac{1}{2}}}{e^{\frac{s}{2}\Delta }}{\phi _s}(\frac{s}{2})} \big\|_{L_t^2L_x^p}} \le {s^L}{e^{ - \frac{1}{16} s}}{\big\| {{\phi _s}(\frac{s}{2})} \big\|_{L_t^2L_x^p}}.
\end{align*}
By Proposition \ref{aaop} and smoothing effect, the first term in $G_1$ is bounded as
\begin{align*}
&{s^L}{\big\| {\int_{\frac{s}{2}}^s {{{\left( { - \Delta } \right)}^{\frac{1}{2}}}{e^{(s - \tau )\Delta }}{h^{ii}}{A_i}{\partial _i}{\phi _s}(\tau )} d\tau } \big\|_{L_t^2L_x^p}}\\
&\le {s^L}\int_{\frac{s}{2}}^s {{{{{\left( {s - \tau } \right)}^{-\frac{1}{2}}}}}{e^{ -\delta (s - \tau )}}{{\big\| {\nabla {\phi _s}(\tau )} \big\|}_{L_t^2L_x^p}}{{\big\| {\sqrt {{h^{ii}}} {A_i}} \big\|}_{L_t^\infty L_x^\infty }}} d\tau  \\
&\le \varepsilon_1{s^L}\int_{\frac{s}{2}}^s {{e^{ -\delta (s - \tau )}}{\tau ^{ - L}}{{{{\left( {s - \tau } \right)}^{-\frac{1}{2}}}}}{{\big\| {{\tau ^L}\nabla {\phi _s}(\tau )} \big\|}_{L_t^2L_x^p}}d} \tau.
\end{align*}
The other terms in $G_1$ can be estimated similarly, thus we obtain for $s\ge1$
\begin{align}\label{wulaso3}
{s^L}{\left\| {{{\left( { - \Delta } \right)}^{\frac{1}{2}}}{\phi _s}(s)} \right\|_{L_t^2L_x^p}} \le \varepsilon_1{\left\| {{s^L}{{\left\| {\nabla {\phi _s}(s)} \right\|}_{L_t^2L_x^p}}} \right\|_{L_s^\infty (s \ge 1)}} + {\left\| {{s^L}{{\left\| {{\phi _s}(\tau )} \right\|}_{L_t^2L_x^p}}} \right\|_{L_s^\infty (s \ge 1)}}.
\end{align}
Combing (\ref{wulaso}), (\ref{wulasoo}), with (\ref{wulaso3}) gives corresponding estimates in (\ref{huojikn}) for $\nabla\phi_s$. It suffices to prove the remaining estimates in (\ref{huojikn}) for $\partial_t\phi_s$. Denote the inhomogeneous term in (\ref{9923}) by $G_3$, then Duhamel principle gives
\begin{align*}
s^{\frac{3}{4}}{\left\| {{\partial _t}{\phi _s}(s)} \right\|_{L_t^2L_x^p}} \le s^{\frac{3}{4}}{\big\| {{e^{\Delta \frac{s}{2}}}{\partial _t}{\phi _s}(\frac{s}{2})} \big\|_{L_t^2L_x^p}} + s^{\frac{3}{4}}{\big\| {\int_{\frac{s}{2}}^s {{e^{\Delta (s - \tau )}}G_3(\tau )d\tau } } \big\|_{L_t^2L_x^p}}.
\end{align*}
The first term of $G_3$ is bounded by
\begin{align*}
 s^{\frac{3}{4}}{\int_{\frac{s}{2}}^s {\big\| {{e^{\Delta (s - \tau )}}{h^{ii}}\left( {{\partial _t}{A_i}} \right){\partial _i}{\phi _s}(\tau )} \big\|} _{L_t^2L_x^p}}d\tau\le s^{\frac{3}{4}}\int_{\frac{s}{2}}^s {\big\| {\sqrt {{h^{ii}}}{\partial _t}{A_i}} \big\|_{L_t^\infty L_x^\infty }}\big\|{\nabla}{\phi_s} \big\|_{L_t^2 L_x^p } d\tau.
\end{align*}
This is acceptable by Proposition \ref{aaop}. The second term in $G_3$ is bounded as
\begin{align}
&s^{\frac{3}{4}}{\int_{\frac{s}{2}}^s {\big\| {{e^{\Delta (s - \tau )}}2{h^{ii}}{A_i}{\partial _i}{\partial _t}{\phi _s}(\tau )} \big\|} _{L_t^2L_x^p}}d\tau\nonumber\\
&\le s^{\frac{3}{4}}{\int_{\frac{s}{2}}^s {\big\| {{e^{\Delta (s - \tau )}}\sqrt {{h^{ii}}} {\partial _i}\left( {\sqrt {{h^{ii}}} {A_i}{\partial _t}{\phi _s}} \right)} \big\|} _{L_t^2L_x^p}}d\tau+ s^{\frac{3}{4}}{\int_{\frac{s}{2}}^s {\big\| {{e^{\Delta (s - \tau )}}{h^{ii}}{\partial _i}{A_i}{\partial _t}{\phi _s}} \big\|} _{L_t^2L_x^p}}d\tau\nonumber\\
&\triangleq I+II.
\end{align}
$I$ is bounded by the smoothing effect, boundedness of Riesz transform and Proposition \ref{aaop}
\begin{align*}
I &\le {s^{\frac{3}{4}}}\int_{\frac{s}{2}}^s {{{\big\| {{{\left( { - \Delta } \right)}^{\frac{1}{2}}}{e^{\Delta (s - \tau )}}\big( {\sqrt {{h^{ii}}} {A_i}{\partial _t}{\phi _s}} \big)} \big\|}_{L_t^2L_x^p}}} d\tau  \\
&\le {s^{\frac{3}{4}}}\int_{\frac{s}{2}}^s {{{{{\big( {s - \tau } \big)}^{-\frac{1}{2}}}}}{{\big\| {\sqrt {{h^{ii}}} {A_i}} \big\|}_{L_t^\infty L_x^\infty }}{{\big\| {{\partial _t}{\phi _s}} \big\|}_{L_t^2L_x^p}}} d\tau\\
&\le {s^{\frac{3}{4}}}\int_{\frac{s}{2}}^s {{{{{\left( {s - \tau } \right)}^{-\frac{1}{2}}}}}} {\varepsilon _1}{\big\| {{\partial _t}{\phi _s}} \big\|_{L_t^2L_x^p}}d\tau.
\end{align*}
$II$ is estimated as the first term of $G_3$ above. The third term of $G_3$ is bounded as
\begin{align}
&{s^{\frac{3}{4}}}\int_{\frac{s}{2}}^s {{{\big\| {{e^{\Delta (s - \tau )}}{h^{ii}}\left( {{\partial _i}{\partial _t}{A_i}} \right){\phi _s}} \big\|}_{L_t^2L_x^p}}} d\tau \nonumber \\
&\le {s^{\frac{3}{4}}}{\int_{\frac{s}{2}}^s {\big\| {{e^{\Delta (s - \tau )}}\sqrt {{h^{ii}}} {\partial _i}\big( {\sqrt {{h^{ii}}} {\partial _t}{A_i}{\phi _s}} \big)} \big\|} _{L_t^2L_x^p}}d\tau\nonumber \\
&+ {s^{\frac{3}{4}}}\int_{\frac{s}{2}}^s {{{\big\| {{e^{\Delta (s - \tau )}}{h^{ii}}{\partial _t}{A_i}{\partial _i}{\phi _s}} \big\|}_{L_t^2L_x^p}}} d\tau.\label{woshi}
\end{align}
The remaining arguments are almost the same as $I$ and $II$. And the rest nine terms in $G_3$ can be estimated as above as well. Hence the desired estimates in (\ref{huojikn1}) for $\partial_t\phi_s$ when $s\in(0,1]$ is verified. It suffices to prove (\ref{huojikn1}) for $\partial_t\phi_s$ when $s\ge1$. The proof for this part is exactly close to the estimates of $\nabla\phi_s$ when $s\ge1$ and that of $I,II$.  (\ref{9o0o}) follows by the same arguments as (\ref{huojikn1}) by applying smoothing effect of the heat semigroup.
By interpolation, in order to verify (\ref{uojiknmpx3}), it suffices to prove
\begin{align}
\left\| \omega_{1}(s)  (-\Delta){\phi _s} \right\|_{L_s^\infty L_t^2L_x^p}\le {\varepsilon _1}.
\end{align}
By (\ref{991}), Duhamel principle and the smoothing effect we have
\begin{align*}
 &\|(- \Delta){\phi _s}(s)\|_{L_t^2L_x^p} \le s^{-1}e^{-\frac{\delta}{2}s}\|{\phi _s}(\frac{s}{2})\|_{L_t^2L_x^p}\\
 &+ \int^{s}_{\frac{s}{2}} (s-\tau)^{-\frac{1}{2}}e^{-\delta(s-\tau)}\big(\|\nabla (h^{ii}A_i\partial_i\phi_s)\|_{{L_t^2L_x^p}}+\|\nabla G\|_{{L_t^2L_x^p}}\big)d\tau.
\end{align*}
Then by Lemma \ref{aoao1}, Proposition \ref{aaop}, (\ref{huojikn}), (\ref{boot2}), (\ref{boot5}), one obtains
\begin{align*}
\left\| \omega_{1}(s)  (-\Delta){\phi _s} \right\|_{L_s^\infty L_t^2L_x^p}\le {\varepsilon _1}\left\| \omega_{1} (s) \nabla^2{\phi _s} \right\|_{L_s^\infty L_t^2L_x^p}+\varepsilon_1.
\end{align*}
Thus (\ref{uojiknmpx3}) follows by Remark \ref{tataru0}.
\end{proof}

\begin{lemma}
Assume that  (\ref{boot2}), (\ref{boot5}) hold, then for $q\in(2,6+2\gamma]$
\begin{align}
{\left\| {{\phi _t}(s)} \right\|_{L_s^\infty L_t^2L_x^q}} &\le {\varepsilon _1}\label{time2}\\
{\left\| {{A_t}} \right\|_{L_t^1L_x^\infty }} &\le \varepsilon _1^2\label{aaop11}
\end{align}
\end{lemma}
\begin{proof}
First notice that $\phi_t$ satisfies $(\partial_s-\Delta)|\phi_t|\le 0$, thus for any fixed $(t,s,x)$ one has the pointwise estimate
\begin{align*}
|\phi_t(s,t,x)|\le |\phi_t(0,t,x)|=|\partial_t u(t,x)|.
\end{align*}
Hence (\ref{time2}) follows by (\ref{boot2}).
From commutator identity we have
\begin{align}\label{qx1}
{\left\| {{A_t}} \right\|_{L_t^1L_x^\infty }} \le \int_0^\infty  {{{\left\| {{\partial _t}u} \right\|}_{L_t^2L_x^\infty }}{{\left\| {{\partial _s}u} \right\|}_{L_t^2L_x^\infty }}} ds.
\end{align}
Sobolev inequality implies for $p_*$ slightly less than 6
\begin{align}\label{yu0cv}
\|\phi_s\|_{L^{\infty}_x}\le \||D|^{\frac{1}{2}}\phi_s\|_{L^{p_*}_x}.
\end{align}
And since $|\partial_t \widetilde{u}|$ satisfies $(\partial_s-\Delta)|\partial_t \widetilde{u}|\le 0$, then
\begin{align}\label{0yu0cv}
\|\phi_t(s)\|_{L^{\infty}_x}\lesssim s^{-1/{p_*}}e^{-\delta s}\|\phi(\frac{s}{2})\|_{L^{p_*}_x}.
\end{align}
By (\ref{0yu0cv}), (\ref{yu0cv}) and (\ref{boot5}),
\begin{align}\label{0yu0cv9}
\int^{1}_{0}\|\phi_t(s)\|_{L^2_tL^{\infty}_x}\|\phi_s(s)\|_{L^2_tL^{\infty}_x}&\lesssim \int^1_0s^{-\frac{1}{2}-\frac{1}{p_*}}\|\phi_t(\frac{s}{2})\|_{L^2_tL^{p_*}_x}s^{\frac{1}{2}}\||D|^{\frac{1}{2}}\phi_s(s)\|_{L^2_tL^{p_*}_x} ds.\\
\int^{\infty}_{1}\|\phi_t(s)\|_{L^2_tL^{\infty}_x}\|\phi_s(s)\|_{L^2_tL^{\infty}_x}&\lesssim \int^{\infty}_1s^{-4L}\|\phi_t(\frac{s}{2})\|_{L^2_tL^{p_*}_x}\||D|^{\frac{1}{2}}\phi_s(s)\|_{L^2_tL^{p_*}_x}ds.
\end{align}
Thus (\ref{aaop11}) is obtained by (\ref{boot2}) and (\ref{boot5}).
\end{proof}

\begin{lemma}\label{gdie1}
Assume that (\ref{boot2}) and (\ref{boot5}) hold, then
for $p\in(2,6+2\gamma]$ with $0<\gamma\ll1$, $\phi_t$ satisfies
\begin{align}
{\left\| {\omega (s)|D|{\phi _t}(s)} \right\|_{L_s^\infty L_t^2L_x^p}} &\le {\varepsilon _1}\label{time}\\
{\left\| {\omega_{\frac{3}{4}} (s)\Delta{\phi _t}(s)} \right\|_{L_s^\infty L_t^2L_x^p}} &\le {\varepsilon _1}\label{timeling}
\end{align}
\end{lemma}
\begin{proof}
By Duhamel principle and (\ref{yfcvbn})
\begin{align*}
{s^{\frac{1}{2}}}{\big\| {{{\left( { - \Delta } \right)}^{\frac{1}{2}}}{\phi _t}(s)} \big\|_{L_t^2L_x^p}} &\le {s^{\frac{1}{2}}}{\big\| {{{\left( { - \Delta } \right)}^{\frac{1}{2}}}{e^{\frac{s}{2}\Delta }}{\phi _t}(\frac{s}{2})} \big\|_{L_t^2L_x^p}} \\
&+ {s^{\frac{1}{2}}}\int_{\frac{s}{2}}^s {{{\big\| {{{\left( { - \Delta } \right)}^{\frac{1}{2}}}{e^{(s - \tau )\Delta }}\mathcal{G}(\tau )} \big\|}_{L_t^2L_x^p}}} d\tau,
\end{align*}
where $\mathcal{G}$ denotes the inhomogeneous terms. By smoothing effect and Proposition \ref{aaop}, the first term in $\mathcal{G}$ is bounded by
\begin{align*}
 &{s^{\frac{1}{2}}}\int_{\frac{s}{2}}^s {{{\big\| {{{\left( { - \Delta } \right)}^{\frac{1}{2}}}{e^{(s - \tau )\Delta }}{h^{ii}}{A_i}{\partial _i}{\phi _t}} \big\|}_{L_t^2L_x^p}}} d\tau \\
 &\le {s^{\frac{1}{2}}}\int_{\frac{s}{2}}^s {{{{{(s - \tau )}^{-\frac{1}{2}}}}}{{\big\| {\nabla {\phi _t}} \big\|}_{L_t^2L_x^p}}} {\big\| {\sqrt {{h^{ii}}} {A_i}} \big\|_{L_t^\infty L_x^\infty }}d\tau  \\
 &\le \varepsilon_1{s^{\frac{1}{2}}}\int_{\frac{s}{2}}^s {{{{{(s - \tau )}^{-\frac{1}{2}}}}}{{\big\| {\nabla {\phi _t}} \big\|}_{L_t^2L_x^p}}} d\tau.
\end{align*}
The large time estimates follow by the same route.
Similar estimates for the rest terms in $\mathcal{G}$ and (\ref{time2}) yield (\ref{time}). By Duhamel principle and smoothing effect, we have
\begin{align*}
\| \Delta\phi_t\|_{L^2_tL^p_x}
\lesssim s^{-\frac{1}{2}}e^{-\delta \frac{s}{2}}\|\nabla\phi_t\|_{L^2_tL^p_x}+\int_{\frac{s}{2}}^s (s-\tau)^{-\frac{1}{2}}e^{-\delta(s-\tau)}\|\nabla \mathcal{G} \|_{L_t^2L_x^p} d\tau.
\end{align*}
Then Lemma \ref{aoao1}, Proposition \ref{aaop}, (\ref{time}), (\ref{boot2}), (\ref{boot5}) give
\begin{align*}
\| \omega_{\frac{3}{4}}(s)\Delta\phi_t\|_{L^{\infty}_sL^2_tL^p_x}
\lesssim \epsilon_1+\epsilon_1\| \omega_{\frac{3}{4}}(s)\nabla^2\phi_t\|_{L^{\infty}_sL^2_tL^p_x}
\end{align*}
Thus (\ref{timeling}) follows by Remark \ref{tataru0}.
\end{proof}

\begin{lemma}
Suppose that (\ref{boot2}) and (\ref{boot5}) hold, then
the wave map tension field satisfies
\begin{align}
{\left\|s^{-\frac{1}{2}} \mathfrak{W}(s) \right\|_{L^\infty_s L_t^1L_x^2}}\le \varepsilon^2_1\label{time5}\\
{\left\| \nabla \mathfrak{W}(s) \right\|_{L^\infty_s L_t^1L_x^2}}\le \varepsilon^2_1\label{time6}\\
{\left\| s^{\frac{1}{2}}\Delta \mathfrak{W}(s) \right\|_{L^\infty_s L_t^1L_x^2}}\le \varepsilon^2_1\label{time7}\\
{\left\| \omega(s)\partial_s \mathfrak{W}(s) \right\|_{L^\infty_s L_t^1L_x^2}}\le \varepsilon^2_1.\label{time8}
\end{align}
\end{lemma}
\begin{proof}
Recall the equation for $\mathfrak{W}$ evolving along $s$:
\begin{align}
{\partial _s}\mathfrak{W} &= \Delta \mathfrak{W} + 2{h^{ii}}{A_i}{\partial _i}\mathfrak{W} + {h^{ii}}{A_i}{A_i}\mathfrak{W} + {h^{ii}}{\partial _i}{A_i}\mathfrak{W}  - {h^{ii}}\Gamma _{ii}^k{A_k}\mathfrak{W} + {h^{ii}}\left( {\mathfrak{W}  \wedge {\phi _i}} \right){\phi _i}\nonumber\\
&+ 3{h^{ii}}({\partial _t}\widetilde{u} \wedge {\partial _i}{\widetilde{u}}){\nabla _t}{\partial _i}\widetilde{u}.\label{gurenjim}
\end{align}
Since $\mathfrak{W}(0,s,x)$=0 for all $(s,x)\in\Bbb R^+\times \Bbb H^2$, Duhamel principle gives
$$(-\Delta)^k\mathfrak{W}(s,t,x)= \int_0^s {{e^{(s - \tau )\Delta }}{{\left( { - \Delta } \right)}^k}G_2(\tau )d\tau },
$$
where $G_2$ denotes the inhomogeneous term. \\
{\bf{Step One.}} In this step, we consider short time behavior, and all the integrand domain of $L^{\infty}_s$ is restricted in $s\in[0,1]$.
By (\ref{u81}), (\ref{butterfly}),
\begin{align*}
 \int_0^s {{{\left\| {h^{ii}{A_i}{A_i}\mathfrak{W}} \right\|}_{L_t^1L_x^2}}} d\kappa &\le \int_0^s {{{\left\| {h^{ii}{A_i}{A_i}} \right\|}_{L_t^\infty L_x^\infty }}} {\left\| \mathfrak{W} \right\|_{L_t^1L_x^2}}d\kappa \le {s^{\frac{3}{2}}}\varepsilon _1^2{\left\| {\mathfrak{W}{s^{ - \frac{1}{2}}}} \right\|_{L_s^\infty L_t^1L_x^2}} \\
 \int_0^s {{{\left\| {{h^{ii}}{\partial _i}{A_i}\mathfrak{W}} \right\|}_{L_t^1L_x^2}}} d\kappa &\le \int_0^s {{{\left\| {{h^{ii}}{\partial _i}{A_i}} \right\|}_{L_t^\infty L_x^\infty }}} {\left\| \mathfrak{W} \right\|_{L_t^1L_x^2}}d\kappa \le s^{\frac{1}{2}}\varepsilon _1^2{\left\| {\mathfrak{W}{s^{ - \frac{1}{2}}}} \right\|_{L_s^\infty L_t^1L_x^2}}.
\end{align*}
By Proposition \ref{sl},
\begin{align}
\int_0^s {{{\left\| {{h^{ii}}\left( {\mathfrak{W} \wedge {\phi _i}} \right){\phi _i}} \right\|}_{L_t^1L_x^2}}} d\kappa \le \int_0^s {{{\left\| {{h^{ii}}{\phi _i}{\phi _i}} \right\|}_{L_t^\infty L_x^\infty }}} {\left\|\mathfrak{W} \right\|_{L_t^1L_x^2}}d\kappa \le s\varepsilon _1^2{\left\| {\mathfrak{W}{s^{ - \frac{1}{2}}}} \right\|_{L_s^\infty L_t^1L_x^2}}.
\end{align}
By (\ref{time}), (\ref{u81}) and Proposition \ref{sl},
\begin{align*}
 &\int_0^s {{{\left\| {{h^{ii}}\left( {{\partial _t}\widetilde{u} \wedge {\partial _i}\widetilde{u}} \right){\nabla _i}{\partial _t}\widetilde{u}} \right\|}_{L_t^1L_x^2}}} d\kappa \\
 &\le \int_0^s {{{\left\| {d\widetilde{u}} \right\|}_{L_t^\infty L_x^6}}{{\left\| {\nabla {\partial _t}\widetilde{u}} \right\|}_{L_t^2L_x^6}}} {\left\| {{\partial _t}\widetilde{u}} \right\|_{L_t^2L_x^6}}d\kappa \\
 &\le \int_0^s {{{\left\| {d\widetilde{u}} \right\|}_{L_t^\infty L_x^6}}{{\left\| {\nabla {\phi _t}} \right\|}_{L_t^2L_x^6}}} {\left\| {{\partial _t}\widetilde{u}} \right\|_{L_t^2L_x^6}}d\kappa + \int_0^s {{{\left\| {d\widetilde{u}} \right\|}_{L_t^\infty L_x^6}}{{\left\| {\sqrt {{h^{ii}}} {A_i}{\phi _t}} \right\|}_{L_t^2L_x^6}}} {\left\| {{\partial _t}\widetilde{u}} \right\|_{L_t^2L_x^6}}d\kappa \\
 &\le {s^{\frac{1}{2}}}\varepsilon _1^2.
 \end{align*}
By the smoothing effect and the boundedness of Riesz transform, we have
\begin{align*}
 &\int_0^s {{{\left\| {{e^{(s - \kappa)\Delta }}{h^{ii}}{A_i}{\partial _i}\mathfrak{W}} \right\|}_{L_t^1L_x^2}}} d\kappa \\
 &\le \int_0^s {{{\left\| {{e^{(s - \kappa)\Delta }}{h^{ii}}{\partial _i}\left( {{A_i}\mathfrak{W}} \right)} \right\|}_{L_t^1L_x^2}}} d\kappa
 + \int_0^s {{{\left\| {{e^{(s - \kappa)\Delta }}{h^{ii}}{\partial _i}{A_i}\mathfrak{W}} \right\|}_{L_t^1L_x^2}}} d\kappa \\
 &\le \int_0^s {{{(s - \kappa)}^{ - \frac{1}{2}}}{{\left\| {\sqrt {{h^{ii}}} {A_i}\mathfrak{W}} \right\|}_{L_t^1L_x^2}}} d\kappa + \int_0^s {{{\left\| {{h^{ii}}{\partial _i}{A_i}\mathfrak{W}} \right\|}_{L_t^1L_x^2}}} d\kappa \\
 &\le {s^{\frac{1}{2}}}\varepsilon _1^2{\left\| {\mathfrak{W}{s^{ - \frac{1}{2}}}} \right\|_{L_s^\infty L_t^1L_x^2}}.
\end{align*}
Hence we conclude (\ref{time5}) for $s\in[0,1]$ by choosing $\varepsilon_1$ sufficiently small. In order to prove (\ref{time6}), we use the following Duhamel principle instead to apply (\ref{time5}),
$${\left( { - \Delta } \right)^{\frac{1}{2}}}\mathfrak{W}(s) = {\left( { - \Delta } \right)^{\frac{1}{2}}}{e^{\frac{s}{2}\Delta }}\mathfrak{W}(\frac{s}{2}) + \int_{\frac{s}{2}}^s {{{\left( { - \Delta } \right)}^{\frac{1}{2}}}{e^{(s - \tau )\Delta }}{G_2}(\tau )d\tau }.
$$
Then (\ref{time6}) follows by (\ref{time5}) and the smoothing effect. Again by Duhamel principle and the smoothing effect,
$$\|\left( { - \Delta } \right)\mathfrak{W}(s)\|_{L^2_x} \le \|\left( { - \Delta } \right){e^{\frac{s}{2}\Delta }}\mathfrak{W}(\frac{s}{2})\|_{L^2_x} + \int_{\frac{s}{2}}^s (s-\tau)^{-\frac{1}{2}}e^{-\delta (s-\tau)}\|\left( { - \Delta } \right)^{\frac{1}{2}}{G_2}(\tau )\|_{L^2_x}d\tau.
$$
Thus Lemma \ref{aoao1}, Proposition \ref{aaop}, (\ref{boot2}), (\ref{boot5}), Remark \ref{tataru0} and Lemma \ref{gdie1} give (\ref{time7}) for $s\in[0,1]$. For $s\in[0,1]$, (\ref{time8}) now arises from (\ref{time5})-(\ref{time7}).\\
{\bf{Step Two.}} We prove (\ref{time5})-(\ref{time7}) for $s\ge1$.
This can be easily obtained by the same arguments as above with the help of $s^{-L}$ decay in the long time case.\\
{\bf{Step Three.}} We prove the large time behavior. The Duhamel principle we use is also
$${\left( { - \Delta } \right)^k}\mathfrak{W}(s) = {\left( { - \Delta } \right)^k}{e^{\frac{s}{2}\Delta }}\mathfrak{W}(\frac{s}{2}) + \int_{\frac{s}{2}}^s {{{\left( { - \Delta } \right)}^k}{e^{(s - \tau )\Delta }}{G_2}(\tau )d\tau }.
$$
Let $s\ge1$, applying smoothing effect we obtain
$${s^L}{\left\| {\mathfrak{W}(s)} \right\|_{L_t^1L_x^2}} \le {s^L}{e^{ - s/8}}{\left\| {\mathfrak{W}(\frac{s}{2})} \right\|_{L_t^1L_x^2}} + {s^L}\int_{\frac{s}{2}}^s {{e^{ - (s - \tau )/8}}{{\left\| {{G_2}(\tau )} \right\|}_{L_t^1L_x^2}}d\tau }.
$$
Then by Hausdorff-Young and (\ref{time5})-(\ref{time7}), for $s\ge1$
\begin{align}\label{1wuhusd}
\|\mathfrak{W}\|_{L^1_tL^2_x}\le \varepsilon^2_1s^{-L}.
\end{align}
Similarly, we have for $s\in[1,\infty)$
\begin{align}\label{wuhusd}
\|\nabla \mathfrak{W}\|_{L^1_tL^2_x}+\|\Delta \mathfrak{W}\|_{L^1_tL^2_x}\le \varepsilon^2_1s^{-L}.
\end{align}
Thus the longtime part of (\ref{time8}) now results from (\ref{1wuhusd}), (\ref{wuhusd}).
\end{proof}

\begin{lemma}
Suppose that (\ref{boot2}) and (\ref{boot5}) hold, then for $0<\gamma\ll1$
\begin{align}
\left\| s^{-\frac{1}{2}}\mathfrak{W}(s)\right\|_{L_s^\infty L_t^2L_x^{3+\gamma}}+\left\| \omega(s)\mathfrak{W}(s)\right\|_{L_s^\infty L_t^2L_x^{3+\gamma}}& \le {\varepsilon _1}\label{time90}\\
{\left\|{\omega (s){\partial _t}{\phi _t}(s)} \right\|_{L_s^\infty L_t^2L_x^{3+\gamma}}} &\le \varepsilon_1.\label{time3}\\
{\left\|{{\partial _t}{A_t}(s)} \right\|_{L_s^\infty L_t^2L_x^{3+\gamma}}} &\le \varepsilon_1.\label{time37}
\end{align}
\end{lemma}
\begin{proof}
(\ref{time3}) is a direct corollary of (\ref{time90}). In fact, the definition of the wave map tension field gives
\begin{align*}
D_t\phi_t=\phi_s+\mathfrak{W}(s).
\end{align*}
Hence $\partial_t\phi_t$ is bounded by $|\phi_s|+|A_t\phi_t|+|\mathfrak{W}|$, then (\ref{time3}) follows by (\ref{boot5}), (\ref{time90}), (\ref{time2}) and (\ref{aaaw811}).
(\ref{time90}) follows by the same arguments as (\ref{time8}). The only difference is to use
\begin{align*}
{\left\| {{h^{ii}}\left( {{\partial _t}\widetilde{u} \wedge {\partial _i}\widetilde{u}} \right){\nabla _i}{\partial _t}\widetilde{u}} \right\|_{L_t^2L_x^{3+\gamma}}} \le {\left\| {\nabla {\partial _t}\widetilde{u}} \right\|_{L_t^2L_x^{6+2\gamma}}}{\left\| {{\partial _t}\widetilde{u}} \right\|_{L_t^\infty L_x^{12+4\gamma}}}{\left\| {d\widetilde{u}} \right\|_{L_t^\infty L_x^{12+4\gamma}}},
\end{align*}
where the term ${\left\| {{\partial _t}\widetilde{u}} \right\|_{L_t^\infty L_x^{12+4\gamma}}}{\left\| {d\widetilde{u}} \right\|_{L_t^\infty L_x^{12+4\gamma}}}$ is bounded by Soboelv embedding and Proposition \ref{sl}. It remains to prove (\ref{time37}).
By the definition of $D_t$ and $A_t$, we have
\begin{align*}
 \left| {{\partial _t}{A_t}(s)} \right| &\le \int_s^\infty  {\left| {{\partial _t}{\phi _t}} \right|\left| {{\phi _s}} \right|} d\kappa + \int_s^\infty  {\left| {{\partial _t}{\phi _s}} \right|\left| {{\phi _t}} \right|d\kappa}  \\
 &\le \int_s^\infty  {\left| {{D_t}{\phi _t}} \right|\left| {{\phi _s}} \right|} d\kappa + \int_s^\infty  {\left| {{A_t}} \right|\left| {{\phi _s}} \right|d\kappa}  + \int_s^\infty  {\left| {{\partial _t}{\phi _s}} \right|\left| {{\phi _t}} \right|d\kappa},
\end{align*}
By $\mathfrak{W}=D_t\phi_t-\phi_s$ and H\"older,
\begin{align}
&{\left\| {\int_s^\infty  {\left| {{D_t}{\phi _t}} \right|\left| {{\phi _s}} \right|} d\kappa} \right\|_{L_t^2L_x^{3+\gamma}}}\nonumber\\
&\le \int_s^\infty  {{{\left\| w \right\|}_{L_t^2L_x^{3+\gamma}}}{{\left\| {{\partial _s}\widetilde{u}} \right\|}_{L_t^\infty L_x^\infty }}} d\kappa + \int_s^\infty  {{{\left\| {{\phi _s}} \right\|}_{L_t^\infty L_x^{6+2\gamma}}}} {\left\| {{\phi _s}} \right\|_{L_t^2L_x^{6+2\gamma}}}d\kappa.\label{huqiancvb}
\end{align}
Since ${\| {{\phi _s}}\|_{L_x^{6+2\gamma}}} \le {\| {{{\left| D \right|}^{\frac{1}{2}}}{\phi _s}} \|_{L_x^p}}$ for $p\in(4,6)$, then
(\ref{huqiancvb}) is acceptable by Proposition \ref{sl} and Proposition \ref{bootstrap}.
Again by H\"older and Sobolev embedding, for $\frac{1}{m}+\frac{1}{4}=\frac{1}{3+\gamma}$
\begin{align*}
{\left\| {\int_s^\infty  {\left| {{\partial _t}{\phi _s}} \right|\left| {{\phi _t}} \right|d\kappa} } \right\|_{L_t^2L_x^{3+\gamma}}} &\le \int_s^\infty  {{{\left\| {{\partial _s}{\phi _t}} \right\|}_{L_t^2L_x^4}}} {\left\| {{\phi _t}} \right\|_{L_t^\infty L_x^m}}d\kappa \\
&\le \int_s^\infty  {{{\left\| {{\partial _s}{\phi _t}} \right\|}_{L_t^2L_x^4}}} {\left\| {\nabla {\partial _t}\widetilde{u}} \right\|_{L_x^2}}d\kappa.
\end{align*}
Since $|\partial_s\phi_t|\le|\partial_t\phi_s|+|A_t\phi_s|$, this is also acceptable by Proposition \ref{sl}, Proposition \ref{bootstrap}, (\ref{aaop11}), and (\ref{aaaw811}). Thus (\ref{time37}) follows.
\end{proof}

\begin{proposition}\label{cuihua}
Suppose that (\ref{boot2}) and (\ref{boot5}) hold.
Then we have for $p\in(2,6)$
\begin{align}
&{\left\| {\omega (s){{\left| D \right|}^{ - \frac{1}{2}}}{\partial _t}{\phi _s}} \right\|_{L_s^\infty L_t^2L_x^p([0,T] \times {\Bbb H^2})}} + {\left\| {\omega (s){{\left| D \right|}^{\frac{1}{2}}}{\phi _s}} \right\|_{L^{\infty}_sL_t^2L_x^p([0,T] \times {\Bbb H^2})}}\nonumber\\
&+ {\left\| {\omega (s){\partial _t}{\phi _s}} \right\|_{L_s^\infty L_t^\infty L_x^2([0,T] \times {\Bbb H^2})}}
+ {\left\| {\omega (s)\nabla {\phi _s}} \right\|_{L_s^\infty L_t^\infty L_x^2([0,T] \times {H^2})}} \le \varepsilon _1^2. \label{huoji}
\end{align}
\end{proposition}
\noindent{\bf{Proof}}
By Lemma \ref{hushuo} and Proposition \ref{tianxia}, we obtain for any $p\in(2,6)$
\begin{align}
&{\omega(s)}{\left\| {{\partial _t}{\phi _s}} \right\|_{L_t^\infty L_x^2}} + {\omega(s)}{\left\| {\nabla {\phi _s}} \right\|_{L_t^\infty L_x^2}} + {\omega(s)}{\left\| {{{\left| \nabla  \right|}^{\frac{1}{2}}}{\phi _s}} \right\|_{L_t^2L_x^p}}\nonumber\\
&+ {\omega(s)}{\left\| {{{\left| D  \right|}^{ - \frac{1}{2}}}{\partial _t}{\phi _s}} \right\|_{L_t^2L_x^p}} +{\omega(s)}{\left\| {{\rho ^\sigma }\nabla {\phi _s}} \right\|_{L_t^2L_x^2}} \nonumber\\
&\lesssim {\omega(s)}{\left\| {{\partial _t}{\phi _s}(0,s,x)} \right\|_{L_x^2}} + {\omega(s)}{\left\| {\nabla {\phi _s}(0,s,x)} \right\|_{L_x^2}} + {\omega(s)}{\left\| G_4 \right\|_{L_t^1L_x^2}}.\label{chunjie1}
\end{align}
where $G_4$ denotes the inhomogeneous term.
First, the $\phi_s(0,s,x)$ term is acceptable by Proposition \ref{sl}, $\mu_2\ll \varepsilon_1$ and
$$
|\nabla_{t,x}\phi_s(0,s,x)|\le |\nabla_{t,x}\partial_sU|+\sqrt{h^{\gamma\gamma}}|A_{\gamma}||\partial_sU|,
$$
where $U(s,x)$ is the heat flow initiated from $u_0$.
Second, the three terms involved with $A_t$ are bounded by
\begin{align*}
{\omega(s)}{\left\| {{A_t}{\partial _t}{\phi _s}} \right\|_{L_t^1L_x^2}} &\le {\left\| {{A_t}} \right\|_{L_t^1L_x^\infty }}{\omega(s)}{\left\| {{\partial _t}{\phi _s}} \right\|_{L_t^\infty L_x^2}} \\
{\omega(s)}{\left\| {{A_t}{A_t}{\phi _s}} \right\|_{L_t^1L_x^2}} &\le {\left\| {{A_t}} \right\|_{L_t^1L_x^\infty }}{\left\| {{A_t}} \right\|_{L_t^\infty L_x^\infty }}{\omega(s)}{\left\| {{\phi _s}} \right\|_{L_t^\infty L_x^2}} \\
{\omega(s)}{\left\| {{\partial _t}{A_t}{\phi _s}} \right\|_{L_t^1L_x^2}} &\le {\left\| {{\partial _t}{A_t}} \right\|_{L_t^2L_x^{3+\gamma}}}{\omega(s)}{\left\| {{\phi _s}} \right\|_{L_t^2L_x^k}},
\end{align*}
where $\frac{1}{k}+\frac{1}{3+\gamma}=\frac{1}{2},$ and $k\in(2,6)$.
They are admissible by (\ref{aaaw811}), (\ref{aaop11}) and (\ref{time37}).
The $\partial_t\widetilde{u}$ term is bounded by
\begin{align*}
{\omega(s)}{\left\| {\mathbf{R}({\partial _t}\widetilde{u},{\partial _s}\widetilde{u})({\partial _t}\widetilde{u})} \right\|_{L_t^1L_x^2}} \le {\left\| {{\partial _t}\widetilde{u}} \right\|_{L_t^2L_x^{6+2\gamma}}}{\left\| {{\partial _t}\widetilde{u}} \right\|_{L_t^\infty L_x^{6+2\gamma}}}{\omega(s)}{\left\| {{\phi _s}} \right\|_{L_t^2L_x^k}},
\end{align*}
where $\frac{1}{k}+\frac{1}{3+\gamma}=\frac{1}{2},$ and $k\in(2,6)$.
The $\partial_s\mathfrak{W}$ term is bounded by (\ref{time8}).
The $A^{con}_i$ terms should be dealt with separately. We present the estimates for these terms as a lemma.

\begin{lemma}[Continuation of Proof of Proposition \ref{cuihua}]
Under the assumption of Proposition \ref{cuihua}, we have
\begin{align}
{\omega(s)}{\left\| {{h^{ii}}A_i^{con}{\partial _i}{\phi _s}} \right\|_{L_t^1L_x^2}} &\le {\varepsilon _1}{\omega(s)}{\left\| {{\rho ^\sigma }\nabla {\phi _s}} \right\|_{L_t^2L_x^2}} + \varepsilon _1^2\label{cuihua1}\\
{\omega(s)}{\left\| {{h^{ii}}A_i^{con}A_i^\infty {\phi _s}} \right\|_{L_t^1L_x^2}}&\le \varepsilon _1^2\label{cuihua2}\\
{\omega(s)}{\left\| {{h^{ii}}A_i^{con}A_i^{con} {\phi _s}} \right\|_{L_t^1L_x^2}}&\le \varepsilon _1^2\label{cuihua3}\\
{\omega(s)}{\left\| {{h^{ii}}\partial_iA_i^{con}{\phi _s}} \right\|_{L_t^1L_x^2}}&\le \varepsilon _1^2\label{cuihua4}\\
{\omega(s)}{\left\| {{h^{ii}}\Gamma _{ii}^kA_k^{con}} \phi_s\right\|_{L_t^1L_x^2}}&\le \varepsilon _1^2\label{cuihua5}.
\end{align}
\end{lemma}
\begin{proof}
Expanding $\phi_i$ as $\phi^{\infty}_i+\int^{\infty}_s\partial_s \phi_id\kappa$ yields
\begin{align}
 A_i^{con} = \int_s^\infty  {{\phi _i} \wedge {\phi _s}} d\kappa = \int_s^\infty  {\left( {\int_{\kappa}^\infty  {{\partial _s}} {\phi _i}(\tau )d\tau  + \phi _i^\infty } \right) \wedge {\phi _s}} (\kappa)d\kappa.
\end{align}
Hence we get
\begin{align*}
&{\omega(s)}{\left\| {{h^{ii}}A_i^{con}{\partial _i}{\phi _s}} \right\|_{L_t^1L_x^2}} \\
&\le {\omega(s)}{\left\| {{h^{ii}}(\int_s^\infty \phi _i^\infty\wedge  {{\phi _s}(\kappa)d\kappa} }) {\partial _i}{\phi _s}\right\|_{L_t^1L_x^2}}\\
&+ {\omega(s)}{\left\| {{h^{ii}} \left( {\int_s^\infty  {{\phi _s}(\kappa) \wedge \left( {\int_{\kappa}^\infty  {{\partial _s}} {\phi _i}(\tau )d\tau } \right)d\kappa} } \right)} {\partial _i}{\phi _s}\right\|_{L_t^1L_x^2}} \\
&\buildrel \Delta \over = B_1 + B_2
\end{align*}
The $B_1$ term is bounded by
\begin{align}
 B_1 &\lesssim {\omega(s)}{\left\| {{\rho ^\sigma }\nabla {\phi _s}} \right\|_{L_t^2L_x^2}}{\left\| {\int_s^\infty  {{\rho ^{ - \sigma }}\phi _i^\infty\sqrt{h^{ii}} {\phi _s}(\kappa)d\kappa} } \right\|_{L_t^2L_x^{\infty}}}\nonumber \\
 &\le {\omega(s)}{\left\| {{\rho ^\sigma }\nabla {\phi _s}} \right\|_{L_t^2L_x^2}}{\left\| {{\rho ^{ - \sigma }}\phi _i^\infty }\sqrt{h^{ii}} \right\|_{L_x^{\infty}}}\int_s^\infty  {{{\left\| {{\phi _s}(\kappa)} \right\|}_{L_t^2L_x^{\infty}}}d\kappa} \nonumber \\
 &\lesssim {\omega(s)}{\left\| {{\rho ^\sigma }\nabla {\phi _s}} \right\|_{L_t^2L_x^2}}{\left\| {{\rho ^{ - \sigma }}\sqrt{h^{ii}}\phi _i^\infty } \right\|_{L_x^{\infty}}}{\left\| {a(s){{\left\| {\nabla{\phi _s}(s)} \right\|}_{L_t^2L_x^{4}}}} \right\|_{L_s^\infty }},\label{guihua}
\end{align}
where we have used the Sobolev embedding in the last step. Hence Proposition \ref{bootstrap} gives an acceptable bound,
\begin{align*}
{B_1} \le C\mu_1\varepsilon_1{\omega(s)}{\left\| {{\rho ^\sigma }\nabla {\phi _s}} \right\|_{L_t^2L_x^2}}.
\end{align*}
The $B_2$ term is bounded by
\begin{align*}
 {B_2} \lesssim {\omega(s)}{\left\| {\nabla {\phi _s}} \right\|_{L_t^\infty L_x^2}}\int_s^\infty  {{{\left\| {{\phi _s}(\kappa)} \right\|}_{L_t^2L_x^\infty }}\left( {\int_{\kappa}^\infty  {{{\left\| {\nabla {\phi _s}(\tau )} \right\|}_{L_t^2L_x^\infty }}} d\tau } \right)d\kappa}.
\end{align*}
Meanwhile, Sobolev embedding and Proposition \ref{bootstrap} give when $\tau  \in (0,1)$
\begin{align*}
 {\left\| {\nabla {\phi _s}(\tau )} \right\|_{L_t^2L_x^\infty }} &\le {\left( {{\tau ^{\frac{3}{4}}}{{\left\| {\nabla {\phi _s}(\tau )} \right\|}_{L_t^2L_x^5}}} \right)^{3/5}}{\left( {{\tau ^{5/4}}{{\left\| {{\nabla ^2}{\phi _s}(\tau )} \right\|}_{L_t^2L_x^5}}} \right)^{2/5}}{\tau ^{ - \frac{1}{2} -9/20}} \\
 &\le {\varepsilon _1}{\tau ^{ -19/20}},
 \end{align*}
 and when $\tau  \in [1,\infty)$
 \begin{align*}
 {\left\| {\nabla {\phi _s}(\tau )} \right\|_{L_t^2L_x^\infty }} &\le {\left( {{\tau ^L}{{\left\| {\nabla {\phi _s}(\tau )} \right\|}_{L_t^2L_x^5}}} \right)^{3/5}}{\left( {{\tau ^L}{{\left\| {{\nabla ^2}{\phi _s}(\tau )} \right\|}_{L_t^2L_x^5}}} \right)^{2/5}}{\tau ^{ - L}} \le {\varepsilon _1}{\tau ^{ - L}}.
\end{align*}
Similarly we deduce by Sobolev embedding $\|f\|_{L^{\infty}}\le \||D|^{\frac{1}{2}}f\|_{L^5}$ that
\begin{align*}
{\left\| {{\phi _s}(\tau )} \right\|_{L_t^2L_x^\infty }} \le {\varepsilon _1}{\tau ^{ - \frac{1}{2}}},\mbox{  }{\rm{when}}\mbox{  }\tau  \in (0,1);\mbox{  }{\left\| {{\phi _s}(\tau )} \right\|_{L_t^2L_x^\infty }} \le {\varepsilon _1}{\tau ^{ - L}},\mbox{  }{\rm{when}}\mbox{  }\tau  \in [1,\infty ).
\end{align*}
Therefore we conclude
\begin{align}
B_2 \le \varepsilon _1^2{\omega(s)}{\left\| {\nabla {\phi _s}} \right\|_{L_t^\infty L_x^2}}\label{guihua1}.
\end{align}
Proposition \ref{bootstrap} together with (\ref{guihua}), (\ref{guihua1}) yields (\ref{cuihua1}). Next we prove (\ref{cuihua2}).
H\"older yields
\begin{align*}
{\omega(s)}{\| {{h^{ii}}A_i^{con}A_i^\infty {\phi _s}}\|_{L_t^1L_x^2}} \le {\| {\sqrt {{h^{ii}}} A_i^{con}} \|_{L_t^2L_x^{\frac{10}{3}}}}{\omega(s)}{\| {{\phi _s}} \|_{L_t^2L_x^5}}.
\end{align*}
Using the expression $A_i^{con} = \int_s^\infty  {{\phi _i} \wedge {\phi _s}} d\kappa$, we obtain
\begin{align}
{\left\| {\sqrt {{h^{ii}}} A_i^{con}} \right\|_{L_t^2L_x^{\frac{10}{3}}}} &\lesssim {\left\| {\int_s^\infty  {\sqrt {{h^{ii}}} {\phi _i} \wedge {\phi _s}} d\kappa} \right\|_{L_t^2L_x^{\frac{10}{3}}}} \le {\| {d\widetilde{u}}\|_{L_t^\infty L_x^{10}}}\int_s^\infty  {{{\left\| {{\phi _s}} \right\|}_{L_t^2L_x^5}}} d\kappa \nonumber\\
&\lesssim {\| {\nabla d\widetilde{u}}\|_{L_t^\infty L_x^2}}{\| {\omega(s){{\| {{\phi _s}(s)} \|}_{L_t^2L_x^5}}}\|_{L_s^\infty }}.\label{guihua7}
\end{align}
Therefore Proposition \ref{bootstrap} gives (\ref{cuihua2}). Third, we verify (\ref{cuihua3}). H\"older yields
\begin{align*}
{\omega(s)}{\| {{h^{ii}}A_i^{con}A_i^{con}{\phi _s}} \|_{L_t^1L_x^2}} \le {\| {\sqrt {{h^{ii}}} A_i^{con}} \|_{L_t^2L_x^{\frac{10}{3}}}}{\| {\sqrt {{h^{ii}}} A_i^{con}} \|_{L_t^\infty L_x^\infty }}{\omega(s)}{\| {{\phi _s}} \|_{L_t^2L_x^5}}.
\end{align*}
The term ${\| {\sqrt {{h^{ii}}} A_i^{con}}\|_{L_t^2L_x^{\frac{10}{3}}}}$ has been estimated in (\ref{guihua7}). The ${\| {\sqrt {{h^{ii}}} A_i^{con}} \|_{L_t^{\infty}L_x^\infty}}$ term is bounded by
\begin{align*}
{\| {\sqrt {{h^{ii}}} A_i^{con}}\|_{L_t^{\infty}L_x^\infty }} \lesssim \left\| {\int_s^\infty  {{{\| {d\widetilde{u}} \|}_{L^\infty_{x} }}{{\| {{\phi _s}} \|}_{L^\infty _{x}}}} d\kappa} \right\|_{L^{\infty}_t}.
\end{align*}
This is acceptable by Proposition \ref{sl} and Lemma \ref{ktao1}. Forth, we prove (\ref{cuihua4}).
H\"older yields
\begin{align*}
{\omega(s)}{\left\| {{h^{ii}}\left( {{\partial _i}A_i^{con}} \right){\phi _s}} \right\|_{L_t^1L_x^2}} \le {\left\| {{h^{ii}}{\partial _i}A_i^{con}} \right\|_{L_t^2L_x^4}}{\omega(s)}{\left\| {{\phi _s}} \right\|_{L_t^2L_x^4}}.
\end{align*}
The $h^{ii}\partial_iA_i$ term is bounded by
\begin{align*}
& {\left\| {{h^{ii}}{\partial _i}A_i^{con}} \right\|_{L_t^2L_x^4}} = {\left\| {\int_s^\infty  {{h^{ii}}{\partial _i}} {\phi _i}{\phi _s}d\kappa + \int_s^\infty  {{h^{ii}}} {\phi _i}{\partial _i}{\phi _s}d\kappa} \right\|_{L_t^2L_x^4}} \\
 &\le \int_s^\infty  {{{\left\| {{h^{ii}}{\partial _i}{\phi _i}} \right\|}_{L_t^\infty L_x^{20}}}} {\left\| {{\phi _s}} \right\|_{L_t^2L_x^5}}d\kappa + {\int_s^\infty  {\left\| {d\widetilde{u}} \right\|} _{L_t^{\infty}L_x^\infty }}{\left\| {\nabla {\phi _s}} \right\|_{L_t^2L_x^4}}d\kappa \\
 &\le \int_s^\infty  {\left( {{{\left\| {\nabla d\widetilde{u}} \right\|}_{L_t^\infty L_x^{20}}} + {{\left\| {{h^{ii}}{A_i}{\phi _i}} \right\|}_{L_t^\infty L_x^{20}}}} \right)} {\left\| {{\phi _s}} \right\|_{L_t^2L_x^5}}d\kappa \\
 &+ {\int_s^\infty  {\left\| {d\widetilde{u}} \right\|} _{L_t^{\infty}L_x^\infty }}{\left\| {\nabla {\phi _s}} \right\|_{L_t^2L_x^4}}d\kappa.
\end{align*}
Thus this is acceptable by Proposition \ref{aaop} and interpolation between the $\|\nabla d\widetilde{u}\|_{L^{\infty}}$ bound and the $\|\nabla d\widetilde{u}\|_{L^{2}}$ bound in Proposition \ref{sl}. Finally we notice that
(\ref{cuihua5}) is a consequence of (\ref{guihua7}) and
$${\omega(s)}{\left\| {{h^{ii}}\Gamma _{ii}^kA_k^{con}} \phi_s\right\|_{L_t^1L_x^2}} \le {\left\| {A_2^{con}} \right\|_{L_t^2L_x^{\frac{10}{3}}}}{\omega(s)}{\left\| {{\phi _s}} \right\|_{L_t^2L_x^5}}.
$$
\end{proof}
$\blacksquare$

Proposition \ref{bootstrap} with Proposition \ref{cuihua} yields
\begin{proposition}\label{xiaozi}
Assume that the solution to (\ref{wmap1}) satisfies (\ref{boot5}) and (\ref{boot2}), then for any  $p\in(2,6)$, $\theta\in[0,2]$
\begin{align*}
{\left\| \omega(s){\nabla\phi_s} \right\|_{L^{\infty}_sL_t^\infty L_x^2}} + {\left\|\omega(s) |D|^{\frac{1}{2}}{\phi_s} \right\|_{L^{\infty}_sL_t^2 L_x^p}} &\le {\varepsilon^2 _1}\\
{\left\| \omega_1(s)|D|{\partial _t}\phi_s \right\|_{L^{\infty}_sL_t^2L_x^p}}
+\left\|\omega_{\theta}(s) (-\Delta)^{\theta}{\phi_s} \right\|_{L^{\infty}_sL_t^2L_x^p}&\le {\varepsilon^2 _1}.
\end{align*}
\end{proposition}

\

\subsection{Close all the bootstrap}
\begin{lemma}
Assume that the solution to (\ref{wmap1}) satisfies (\ref{boot5}) and (\ref{boot2}), then for any $p\in(2,6+2\gamma]$
\begin{align}
{\left\| {(du,\partial_tu)} \right\|_{L_t^\infty L_x^2([0,T] \times {\Bbb H^2})}} + {\left\| {(\nabla{\partial _t}u,\nabla du)} \right\|_{L_t^\infty L_x^2([0,T] \times {\Bbb H^2})}}&\le {\varepsilon^2 _1}\label{boot8}\\
{\left\| {{\partial _t}u} \right\|_{L_t^2L_x^p([0,T] \times {\Bbb H^2})}} &\le {\varepsilon^2_1}.\label{boot9}
\end{align}
\end{lemma}
\begin{proof}
First we prove (\ref{boot9}). By $D_s\phi_t=D_t\phi_s$, $A_s=0$, one has
\begin{align}
{\left\| {{\phi _t}(0,t,x)} \right\|_{L_t^2L_x^p}} \le {\left\| {\int_0^\infty  {\left| {{\partial _s}{\phi _t}} \right|} ds} \right\|_{L_t^2L_x^p}} \le {\left\| {{\partial _t}{\phi _s}} \right\|_{L_s^1L_t^2L_x^p}}+{\left\| {{A _t}{\phi _s}} \right\|_{L_s^1L_t^2L_x^p}}.
\end{align}
Sobolev embedding gives
$${\left\| {{\partial _t}{\phi _s}} \right\|_{L_x^{6 + 2\gamma }}} + {\left\| {{\phi _s}} \right\|_{L_x^{6 + 2\gamma }}} \le {\left\| {{{\left( { - \Delta } \right)}^\vartheta }{\partial _t}{\phi _s}} \right\|_{L_x^{6 - \eta }}} + {\left\| {{{\left( { - \Delta } \right)}^\vartheta }{\phi _s}} \right\|_{L_x^{6 - \eta }}},
$$
where $\frac{\vartheta }{2} = \frac{1}{{6 - \eta }} - \frac{1}{{6 + 2\gamma }}$, $0<\eta\ll1,0<\gamma\ll1$.
Thus (\ref{boot9}) follows by Proposition \ref{xiaozi}  and Proposition \ref{aaop}.
Second, we prove (\ref{boot8}). By Remark \ref{3sect}, ${\phi _i}(0,t,x) = \phi _i^\infty  + \int_0^\infty  {{\partial _s}{\phi _i}d\kappa}.$
Since $|d\widetilde{u}|\le \sqrt{h^{ii}}|\phi_i|$,  $\|\sqrt{h^{ii}}\phi _i^\infty\|_{L^2_x}\le \|dQ\|_{L^2}\le \mu_1$, it suffices to verify for any  $t,x\in[0,T]\times\Bbb H^2$
$$\int_0^\infty  \|\sqrt{h^{ii}}{\partial _s}{\phi _i}\|_{L^2_x}d\kappa\le \varepsilon^2_1.
$$
This is acceptable by Proposition \ref{aaop}, Proposition \ref{xiaozi} and $|\sqrt{h^{ii}}\partial_s\phi_i|\le |\nabla \phi_s|+\sqrt{h^{ii}}|A_i||\phi_s|$.
Recalling (\ref{poijn}) for the equation of $\phi_s$ evolving along the heat flow, we have by integration by parts,
\begin{align*}
 \frac{d}{{ds}}\left\| {\tau (\widetilde{u})} \right\|_{L_x^2}^2 &= \frac{d}{{ds}}\left\langle {{\phi _s},{\phi _s}} \right\rangle  = 2\left\langle {{D_s}{\phi _s},{\phi _s}} \right\rangle  \\
 &= 2{h^{ii}}\left\langle {{D_i}{D_i}{\phi _s} - \Gamma _{ii}^k{D_k}{\phi _s},{\phi _s}} \right\rangle +\left\langle h^{ij}(\phi_s\wedge\phi_i)\phi_j,\phi_s\right\rangle\\
 &=  - 2{h^{ii}}\left\langle {{D_i}{\phi _s},{D_i}{\phi _s}} \right\rangle+\left\langle h^{ij}(\phi_s\wedge\phi_i)\phi_j,\phi_s\right\rangle.
\end{align*}
Hence $\|\partial_s\widetilde{u}\|_{L^2_x}\le e^{-\delta s}$ shows
\begin{align}
 &\left\| {\tau (\widetilde{u}(0,t,x))} \right\|_{L_x^2}^2\lesssim \int_0^\infty  {{h^{ii}}\left\langle {{D_i}{\phi _s},{D_i}{\phi _s}} \right\rangle } ds \nonumber\\
 &\lesssim \int_0^\infty  {\left\langle {\nabla {\phi _s},\nabla {\phi _s}} \right\rangle } ds + \int_0^\infty  {{h^{ii}}\left\langle {{A_i}{\phi _s},{A_i}{\phi _s}} \right\rangle } ds+\int^{\infty}_0|d\widetilde{u}|^2|\phi_s|^2ds.\label{muijbnc}
\end{align}
The nonnegative sectional curvature property of $N=\Bbb H^2$ with integration by parts implies
$$\left\| {\nabla d\widetilde{u}} \right\|_{L_x^2}^2 \lesssim \left\| {\tau (\widetilde{u})} \right\|_{L_x^2}^2 + \left\| {d\widetilde{u}} \right\|_{L_x^2}^2.
$$
Hence (\ref{muijbnc}) gives
\begin{align}
&\left\| {\nabla d\widetilde{u}(0,t,x)} \right\|_{L_x^2}^2 \nonumber\\
&\lesssim \int_0^\infty  {\left\langle {\nabla {\phi _s},\nabla {\phi _s}} \right\rangle } ds
+\int^{\infty}_0|d\widetilde{u}|^2|\phi_s|^2ds+\int_0^\infty  {{h^{ii}}\left\langle {{A_i}{\phi _s},{A_i}{\phi _s}} \right\rangle } ds+\|d\widetilde{u}(0,t,x)\|^2_{L^2_x}.\label{hjbn7v89}
\end{align}
Since the $|d\widetilde{u}|$ term has been estimated, by Proposition \ref{xiaozi} , Proposition \ref{aaop} and (\ref{hjbn7v89}),
$$ {\left\| {\nabla d\widetilde{u}} \right\|_{L_t^\infty L_x^2([0,T] \times {\Bbb H^2})}}\le \varepsilon^2_1.
$$
Finally we prove the desired estimates for $|\nabla\partial_t\widetilde{u}|$. Integration by parts yields,
\begin{align*}
 &\frac{d}{{ds}}\left\| {\nabla {\partial _t}\widetilde{u}} \right\|_{{L^2}}^2 = \frac{d}{{ds}}{h^{ii}}\left\langle {{D_i}{\phi _t},{D_i}{\phi _t}} \right\rangle  = 2{h^{ii}}\left\langle {{D_s}{D_i}{\phi _t},{D_i}{\phi _t}} \right\rangle  \\
 &= 2{h^{ii}}\left\langle {{D_i}{D_t}{\phi _s},{D_i}{\phi _t}} \right\rangle  + 2{h^{ii}}\left\langle {({\phi _s} \wedge {\phi _i}){\phi _t},{D_i}{\phi _t}} \right\rangle  \\
 &=  - 2{h^{ii}}\left\langle {{D_t}{\phi _s},{D_i}{D_i}{\phi _t}} \right\rangle  + 2\left\langle {{D_t}{\phi _s},{D_2}{\phi _t}} \right\rangle  + 2{h^{ii}}\left\langle {({\phi _s} \wedge {\phi _i}){\phi _t},{D_i}{\phi _t}} \right\rangle  \\
 &=  - 2\left\langle {{D_t}{\phi _s},{h^{ii}}{D_i}{D_i}{\phi _t} - {h^{ii}}\Gamma _{ii}^k{D_k}{\phi _t}} \right\rangle  + 2{h^{ii}}\left\langle {({\phi _s} \wedge {\phi _i}){\phi _t},{D_i}{\phi _t}} \right\rangle.
\end{align*}
Recall (\ref{yfcvbn}), the parabolic equation of $\phi_t$ along heat flow, then
$$\frac{d}{{ds}}\left\| {\nabla {\partial _t}\widetilde{u}} \right\|_{{L^2}}^2 =  - 2\left\langle {{D_t}{\phi _s},{D_s}{\phi _t}} \right\rangle  + 2{h^{ii}}\left\langle {({\phi _s} \wedge {\phi _i}){\phi _t},{D_i}{\phi _t}} \right\rangle  + 2{h^{ii}}\left\langle {{D_t}{\phi _s},({\phi _t} \wedge {\phi _i}){\phi _i}} \right\rangle.
$$
Hence we conclude,
\begin{align*}
 &\left\| {\nabla {\partial _t}\widetilde{u}(0,t,x)} \right\|_{{L^2}}^2\\
 &\lesssim \int_0^\infty  {\left\langle {{\partial _t}{\phi _s},{\partial _t}{\phi _s}} \right\rangle d\kappa}  + \int_0^\infty  {\left\langle {{A_t}{\phi _s},{A_t}{\phi _s}} \right\rangle d\kappa}  \\
 &+ \int_0^\infty  {{{\left\| {{\phi _s}} \right\|}_{L_x^2}}{{\left\| {d\widetilde{u}} \right\|}_{L_x^\infty }}{{\left\| {{\partial _t}\widetilde{u}} \right\|}_{L_x^\infty }}{{\left\| {\nabla {\partial _t}\widetilde{u}} \right\|}_{L_x^2}}d\kappa}  + {\int_0^\infty  {\left\| {{\partial _t}\widetilde{u}} \right\|} _{L_x^2}}\left\| {d\widetilde{u}} \right\|_{L_x^\infty }^2{\left\| {{D_t}{\phi _s}} \right\|_{L_x^2}}d\kappa.
\end{align*}
Thus by Proposition \ref{sl}, Proposition \ref{xiaozi} and Proposition \ref{aaop}, we have
$$\left\| {\nabla {\partial _t}\widetilde{u}(0,t,x)} \right\|_{{L^2}}^2 \le \varepsilon _1^4.
$$
Therefore, we have proved all estimates in (\ref{boot8}) and (\ref{boot9}).
\end{proof}

We summarize what we have proved in the following corollary.
\begin{corollary}
Assume $(-T^*,T_*)$ is the lifespan of solution to (\ref{wmap1}). And let $\mu_1,\mu_2$ be sufficiently small, then we have
\begin{align*}
 {\left\| {du} \right\|_{L_t^\infty L_x^2([0,{T_*}] \times {\Bbb H^2})}}& + {\left\| {{\partial _t}u} \right\|_{L_t^\infty L_x^2([ 0,{T_*}] \times {\Bbb H^2})}} + {\left\| {\nabla du} \right\|_{L_t^\infty L_x^2([0,{T_*}] \times {\Bbb H^2})}} \\
 &+ {\left\| {\nabla {\partial _t}u} \right\|_{L_t^\infty L_x^2([0,{T_*}] \times {\Bbb H^2})}} + {\left\| {{\partial _t}u} \right\|_{L_t^2L_x^6([ 0,{T_*}] \times {\Bbb H^2})}} \le \varepsilon _1^2.
\end{align*}
Thus by Proposition \ref{global}, we have $(u,\partial_tu)$ is a global solution to (\ref{wmap1}).
\end{corollary}

\section{Proof of Theorem 1.1}
Finally, we prove Theorem 1.1 based on Proposition \ref{xiaozi}.
\begin{proposition}
Let $u$ be the solution to (\ref{wmap1}) in $\mathcal{X}_{[0,\infty)}$. Then as $t\to \infty$, $u(t,x)$ converges to a harmonic map, namely
$$
\mathop {\lim }\limits_{t \to \infty } \mathop {\lim }\limits_{x\in\Bbb H^2 }{\rm{dist}}_{\Bbb H^2}(u(t,x),Q(x))= 0,
$$
where $Q(x):\Bbb H^2\to \Bbb H^2$ is the unperturbed harmonic map.
\end{proposition}
\begin{proof}
For $u(t,x)$, by Proposition \ref{3.3}, we have the corresponding heat flow converges to some harmonic map uniformly for $x\in\Bbb H^2$. Then by the definition of the distance on complete manifolds, we have
\begin{align}\label{ppo0}
{\rm{dist}_{{\Bbb H^2}}}(u(t,x),Q(x)) \le \int_0^\infty  {{{\left\| {{\partial _s}\widetilde{u}} \right\|}_{L_x^\infty }}ds}.
\end{align}
For any $T>0$, $\mu>0$, since $|\partial_s\widetilde{u}|$ satisfies $(\partial_s-\Delta)|\partial_s\widetilde{u}|\le 0$, one has
\begin{align}
 \int_T^\infty  {{{\left\| {{\partial _s}\widetilde{u}(s,t,x)} \right\|}_{L_x^\infty }}ds}  &\lesssim \int_T^\infty  {{e^{ - \frac{1}{8}s}}{{\left\| {\tau(u(t,x))} \right\|}_{L_x^2}}} ds \lesssim {e^{ - T/8}}{\left\| {\nabla du(t,x)} \right\|_{L_x^2}} \label{mu1}\\
 \int_0^\mu  {{{\left\| {{\partial _s}\widetilde{u}(s,t,x)} \right\|}_{L_x^\infty }}ds}  &\lesssim \int_0^\mu  {{{\left\| {{e^{s{\Delta _{{\Bbb H^2}}}}}\tau(u(t,x))} \right\|}_{L_x^\infty }}} ds \le \int_0^\mu  {{s^{ - \frac{1}{2}}}{{\left\| {\nabla du(t,x)} \right\|}_{L_x^2}}} ds \nonumber\\
 &\lesssim {\mu ^{\frac{1}{2}}}{\left\| {\nabla du(t,x)} \right\|_{L_x^2}} \label{mu2}
 \end{align}
Similarly, we have
\begin{align}
 \int_\mu ^T {{{\left\| {{\partial _s}\widetilde{u}(s,t,x)} \right\|}_{L_x^\infty }}ds}
 &\lesssim \int_\mu ^T {{{\left\| {{e^{(s - \frac{\mu }{2}){\Delta _{{\Bbb H^2}}}}}{\partial _s}\widetilde{u}(\frac{\mu}{2},t,x)} \right\|}_{L_x^\infty }}} ds \nonumber\\
 &\lesssim \int_\mu ^T {{{(s - \frac{\mu}{2})}^{ - \frac{1}{4}}}{{\left\| {{\partial _s}\widetilde{u}(\frac{\mu}{2},t,x)} \right\|}_{L_x^4}}} ds \nonumber\\
 &\lesssim {\mu ^{ - \frac{1}{4}}}\int_\mu ^T {{{\left\| {{\phi _s}(\frac{\mu}{2},t,x)} \right\|}_{L_x^4}}} ds.\label{mu3}
\end{align}
Therefore it suffices to prove for a fixed $\mu>0$
\begin{align}\label{8ding}
\mathop {\lim }\limits_{t \to \infty } {\left\| {{\phi _s}(\mu )} \right\|_{L_x^4}} = 0.
\end{align}
Proposition \ref{xiaozi} implies ${\mu ^{\frac{1}{2}}}{\left\| {{\phi _s}(\mu )} \right\|_{L_t^2L_x^4}}+{\mu ^{\frac{1}{2}}}{\left\| \partial_t{{\phi _s}(\mu )} \right\|_{L_t^2L_x^4}}< \infty $, thus for any $\epsilon>0$ there exists a $T_0$ such that
\begin{align}\label{7vgj}
{\left\| {{\phi _s}(\mu )} \right\|_{L_t^2L_x^4([{T_0},\infty ) \times {\Bbb H^2})}}+ {\left\|\partial_t {{\phi _s}(\mu )} \right\|_{L_t^2L_x^4([{T_0},\infty ) \times {\Bbb H^2})}}< \epsilon.
\end{align}
Particularly, for any interval $[a,a+1]$ of length one with $a\ge T_0$, there exists some $t_{a}\in[a,a+1]$ such that
\begin{align}\label{shipo}
{\left\| {{\phi _s}(\mu ,{t_{a}})} \right\|_{L_x^4}} \le \epsilon/2.
\end{align}
Then by fundamental theorem of calculus for any $t'\in[a,a+1]$
\begin{align}\label{gouq9}
\left| {{{\left\| {{\phi _s}(\mu ,t')} \right\|}_{L_x^4}} - {{\left\| {{\phi _s}(\mu ,{t_a})} \right\|}_{L_x^4}}} \right| \le \int_{{t_a}}^{t'} {\left| {{\partial _t}{{\left\| {{\phi _s}(\mu ,t)} \right\|}_{L_x^4}}} \right|} dt.
\end{align}
Since $|{\partial _t}{\left\| {{\phi _s}(\mu ,t)} \right\|_{L_x^4}}|\le {\left\| {{\partial _t}{\phi _s}(\mu ,t)} \right\|_{L_x^4}}$, by H\"older, (\ref{gouq9}) and (\ref{7vgj}) show
\begin{align*}
\left| {{{\left\| {{\phi _s}(\mu ,t')} \right\|}_{L_x^4}} - {{\left\| {{\phi _s}(\mu ,{t_a})} \right\|}_{L_x^4}}} \right| \le {\left\| {{\partial _t}{\phi _s}(\mu ,t)} \right\|_{L_t^2L_x^4}}{(t' - a)^{\frac{1}{2}}} \le {\left\| {{\partial _t}{\phi _s}(\mu ,t)} \right\|_{L_t^2L_x^4}}.
\end{align*}
Thus we have by (\ref{shipo}) that for any $t\in[a,a+1]$,
\begin{align*}
{\left\| {{\phi _s}(\mu ,t)} \right\|_{L_x^4}}\le \epsilon.
\end{align*}
Since $a$ is arbitrary chosen, we obtain (\ref{8ding}). Therefore, Theorem 1.1 is proved,
\end{proof}

\section{Proof of remaining lemmas and claims}
We first collect some useful inequalities for the harmonic maps.
\begin{lemma}
Suppose that $Q$ is an admissible harmonic map in Theorem 1.1. If $0<\mu_1\ll 1$, then
\begin{align}
\|\nabla dQ\|_{L^2}&\lesssim \mu_1\label{shubai4}\\
\|\nabla^2 dQ\|_{L^2}&\lesssim \mu_1.\label{tianren78}
\end{align}
\end{lemma}
\begin{proof}
By integration by parts and the non-positive sectional curvature of $N=\Bbb H^2$,
\begin{align*}
\|\nabla dQ\|^2_{L^2}&\lesssim \|dQ\|^2_{L^2}+\|\tau(Q)\|^2_{L^2}\\
\|\nabla^2 dQ\|^2_{L^2}&\lesssim \|\nabla\tau(Q)\|^2_{L^2}+\|\nabla dQ\|^3_{L^2}+\|\nabla dQ\|^2_{L^4}\| dQ\|^2_{L^4}+\| dQ\|^6_{L^2}.
\end{align*}
Hence by $\tau(Q)=0$, we have (\ref{shubai4}). And then (\ref{tianren78}) follows from (\ref{as4}), Gagliardo-Nirenberg inequality and Sobolev embedding.
\end{proof}

Now we prove Corollary 2.1.
\begin{lemma}
Fix $R_0>0$, let $0<\mu_1,\mu_2\ll\mu_3\ll1,$ then the initial data $(u_0,u_1)$ in Theorem 1.1 satisfy
\begin{align}
\|du_0\|_{L^2}+\|u_1\|_{L^2}+\|\nabla du_0\|_{L^2}+\|\nabla u_1\|_{L^2}\le \mu_3.\label{ojvbhuy}
\end{align}
\end{lemma}
\begin{proof}
First by (\ref{shubai4}), the harmonic map $Q$ satisfies
\begin{align}\label{kijncvbn}
\|\nabla dQ\|_{L^2}+\|dQ\|_{L^2}\le \mu_1.
\end{align}
By (1.4) and Sobolev embedding,
\begin{align}
\|u^k_0-Q^k\|_{L^{\infty}}\lesssim \|u^k_0-Q^k\|_{H^2}\le \mu_2.
\end{align}
Hence $|u^1_0|+|u^2_0|\lesssim R_0+\mu_2.$ Then choosing $R=CR_0+C\mu_2$ in [Lemma 2.3,\cite{LZ}], we have
\begin{align}\label{zxzxsdu}
\|du_0\|_{L^2}+\|\nabla du_0\|_{L^2}\le Ce^{8(CR_0+C\mu_2)}\big(\|\nabla^2 u^k_0\|_{L^2}+\|\nabla^2 u^k_0\|_{L^2}^2\big).
\end{align}
Again by [Lemma 2.3,\cite{LZ}] and (\ref{kijncvbn}),
\begin{align}\label{zw34vbg}
\|\nabla^2Q^k\|_{L^2}\le Ce^{8(R_0)}\big(\|\nabla dQ\|_{L^2}+\|\nabla dQ\|_{L^2}^2\big)\le Ce^{8(R_0)}\mu_1.
\end{align}
Therefore, (1.4), (\ref{zw34vbg}) and (\ref{zxzxsdu}) give
\begin{align}
\|du_0\|_{L^2}+\|\nabla du_0\|_{L^2}\le Ce^{8(CR_0+C\mu_2)}(\mu_1+\mu_2)
\end{align}
Let $\mu_1$ and $\mu_2$ be sufficiently small depending on $R_0$, we obtain
\begin{align}
\|du_0\|_{L^2}+\|\nabla du_0\|_{L^2}\le \mu_3.
\end{align}
\end{proof}

\begin{lemma}\label{symm}
Let $W$ be the magnetic operator defined in Lemma \ref{hushuo} as
\begin{align}
W\varphi  = - 2 {h^{ii}}A_i^\infty {\partial _i}\varphi  -{h^{ii}}A_i^\infty A_i^\infty \varphi  -{h^{ii}}\left( {\varphi  \wedge \phi _i^\infty } \right)\phi _i^\infty-h^{ii}(\partial_iA^{\infty}_{i}-\Gamma^k_{ii}A^{\infty}_k),
\end{align}
Then $W$ is symmetric with domain $C^{\infty}_c(\Bbb H^2,\Bbb C^2)$. And $-\Delta+W$ is strictly positive if $\mu_1$ is sufficiently small.
\end{lemma}
\begin{proof}
Since we work with complex valued functions here, the wedge operator $\wedge$ should be first extended to the complex number field by taking the inner product in (\ref{nb890km}) to be the complex inner product.
By the explicit formula for $\Gamma^{k}_{ii}$ and $h^{ii}$, one has
\begin{align}\label{kjc6789}
h^{ii}\Gamma^k_{ii}A^{\infty}_k=h^{11}\Gamma^2_{11}A^{\infty}_2=A^{\infty}_2.
\end{align}
It is easy to see by the non-positiveness  and symmetry of the sectional curvature that
$\varphi\longmapsto -{h^{ii}}\left( {\varphi  \wedge \phi _i^\infty } \right)\phi _i^\infty$ is a non-negative and symmetric operator on $L^2(\Bbb H^2,\Bbb C^2)$.
And by the skew-symmetry of $A^{\infty}_i$,
$\varphi\longmapsto -{h^{ii}}\left( {\varphi  \wedge A _i^\infty } \right)A_i^\infty$ is a non-negative and symmetric symmetric operator on $L^2(\Bbb H^2,\Bbb C^2)$.
We claim that
$$\varphi\longmapsto 2 {h^{ii}}A_i^\infty {\partial _i}\varphi +h^{ii}(\partial_iA^{\infty}_{i}-\Gamma^k_{ii}A^{\infty}_k)
$$
is a symmetric operator on $L^2(\Bbb H^2,\Bbb C^2)$ as well.
Indeed, by the skew-symmetry of $A^{\infty}_i$, $\partial_iA^{\infty}_i$, integration by parts and (\ref{kjc6789}),
\begin{align*}
&\left\langle {2{h^{ii}}A_i^\infty {\partial _i}f + {h^{ii}}({\partial _i}A_i^\infty  - \Gamma _{ii}^kA_k^\infty )f,g} \right\rangle  \\
&= \left\langle {2{h^{ii}}A_i^\infty {\partial _i}f + {h^{ii}}{\partial _i}A_i^\infty f - A_2^\infty f,g} \right\rangle  \\
&= \left\langle {{h^{ii}}{\partial _i}A_i^\infty f - A_2^\infty f,g} \right\rangle  - \left\langle {2{h^{ii}}{\partial _i}A_i^\infty f,g} \right\rangle  - \left\langle {2{h^{ii}}A_i^\infty f,{\partial _i}g} \right\rangle + \left\langle {2{h^{22}}A_2^\infty f,g} \right\rangle  \\
&= \left\langle { - {h^{ii}}{\partial _i}A_i^\infty f + A_2^\infty f,g} \right\rangle  - \left\langle {2{h^{ii}}A_i^\infty f,{\partial _i}g} \right\rangle  \\
&= \left\langle {f,{h^{ii}}{\partial _i}A_i^\infty g - A_2^\infty g} \right\rangle  + \left\langle {f,2{h^{ii}}A_i^\infty {\partial _i}g} \right\rangle  \\
&= \left\langle {f,2{h^{ii}}A_i^\infty {\partial _i}g + {h^{ii}}{\partial _i}A_i^\infty g - A_2^\infty g} \right\rangle.
\end{align*}
It remains to prove $-\Delta+W$ is positive. Since we have shown $\varphi\longmapsto -{h^{ii}}\left( {\varphi  \wedge \phi _i^\infty } \right)\phi _i^\infty$ and $\varphi\longmapsto -{h^{ii}}\left( {\varphi  \wedge A _i^\infty } \right)A_i^\infty$ are nonnegative, it suffices to prove for some $\delta>0$
$$\left\langle { - \Delta f + 2{h^{ii}}A_i^\infty {\partial _i}f + {h^{ii}}({\partial _i}A_i^\infty  - \Gamma _{ii}^kA_k^\infty )f,f} \right\rangle  \ge \delta \left\langle {f,f} \right\rangle.
$$
By the skew-symmetry of $A^{\infty}_i$ and $\partial_iA^{\infty}_i$, it reduces to
\begin{align*}
\left\langle { - \Delta f + 2{h^{ii}}A_i^\infty {\partial _i}f,f} \right\rangle  \ge \delta \left\langle {f,f} \right\rangle.
\end{align*}
H\"older, (\ref{{uv111}}) and (\ref{xuejin}) imply for some universal constant $c>0$
\begin{align*}
& \left\langle { - \Delta f + 2{h^{ii}}A_i^\infty {\partial _i}f,f} \right\rangle  \ge \left\| {\nabla f} \right\|_2^2 - 2{\left\| {\sqrt {{h^{ii}}} A_i^\infty } \right\|_\infty }{\left\| {\nabla f} \right\|_2}{\left\| f \right\|_2} \\
&\ge \frac{1}{2}\left\| {\nabla f} \right\|_2^2 + c\left\| f \right\|_2^2 - 2{\left\| {\sqrt {{h^{ii}}} A_i^\infty } \right\|_\infty }{\left\| {\nabla f} \right\|_2}{\left\| f \right\|_2} \\
&\ge \frac{1}{2}\left\| {\nabla f} \right\|_2^2 + c\left\| f \right\|_2^2 - 2{\mu _1}{\left\| {\nabla f} \right\|_2}{\left\| f \right\|_2}.
\end{align*}
Let $\mu_1$ be sufficiently small, then
\begin{align*}
\left\langle { - \Delta f + 2{h^{ii}}A_i^\infty {\partial _i}f,f} \right\rangle\ge \delta \left\langle {f,f} \right\rangle.
\end{align*}
\end{proof}

Recall the equation of the tension field $\phi_s$:
\begin{lemma}
The evolution of differential fields and the heat tension filed along the heat flow are given by the following:
\begin{align}
&\partial_s\phi_s=h^{ii}D_iD_i\phi_s-h^{ii}\Gamma^k_{ii}D_k\phi_s+h^{ii}(\phi_s\wedge\phi_i)\phi_i \label{poijn}\\
&{\partial_s}{\phi _s}-\Delta {\phi _s}= 2{h^{ii}}{A_i}{\partial _i}{\phi _s} + {h^{ii}}\left( {{\partial _i}{A_i}} \right){\phi _s} - {h^{ii}}\Gamma _{ii}^k{A_k}{\phi _s} + {h^{ii}}{A_i}{A_i}{\phi _s}\nonumber \\
&+ {h^{ii}}\left( {{\phi _s} \wedge {\phi _i}} \right){\phi _i}\label{991}\\
&{\partial _s}{\phi _t} - \Delta {\phi _t} = 2h^{ii}{A_i}{\partial _i}{\phi _t} + h^{ii}{A_i}{A_i}{\phi _t} + h^{ii}{\partial _i}{A_i}{\phi _t} - {h^{ii}}\Gamma _{ii}^k{A_k}{\phi _t} \nonumber \\
&+ {h^{ii}}\left( {{\phi _t} \wedge {\phi _i}} \right){\phi _i}.\label{yfcvbn}\\
&{\partial _s}{\partial _t}{\phi _s}= \Delta {\partial _t}{\phi _s} + 2{h^{ii}}\left( {{\partial _t}{A_i}} \right){\partial _i}{\phi _s} + 2{h^{ii}}{A_i}{\partial _i}{\partial _t}{\phi _s} + {h^{ii}}\left( {{\partial _i}{\partial _t}{A_i}} \right){\phi _s}\nonumber\\
&+ {h^{ii}}\left( {{\partial _i}{A_i}} \right){\partial _t}{\phi _s}
- {h^{ii}}\Gamma _{ii}^k\left( {{\partial _t}{A_k}} \right){\phi _s} - {h^{ii}}\Gamma _{ii}^k{A_k}{\partial _t}{\phi _s} + {h^{ii}}\left( {{\partial _t}{A_i}} \right){A_i}{\phi _s}\nonumber\\
&+ {h^{ii}}{A_i}\left( {{\partial _t}{A_i}} \right){\phi _s}
+ {h^{ii}}{A_i}{A_i}{\partial _t}{\phi _s} + {h^{ii}}\left( {{\partial _t}{\phi _s} \wedge {\phi _i}} \right){\phi _i}+ {h^{ii}}\left( {{\phi _s} \wedge {\partial _t}{\phi _i}} \right){\phi _i}\nonumber\\
& + {h^{ii}}\left( {{\phi _s} \wedge {\phi _i}} \right){\partial _t}{\phi _i}.\label{9923}
\end{align}
\end{lemma}
\begin{proof}
Recall that we use the orthogonal coordinates (\ref{vg}) throughout the paper.
Recall the equation of $\phi_s$:
\begin{align}\label{who}
\phi_s=h^{ii}D_i\phi_i-h^{ii}\Gamma^k_{ii}\phi_k.
\end{align}
Applying $D_s$ to (\ref{who}) yields
\begin{align*}
 &{D_s}{\phi _s} = {h^{ii}}{D_s}{D_i}{\phi _i} - {h^{ii}}\Gamma _{ii}^k{D_s}{\phi _k} = {h^{ii}}{D_i}{D_i}{\phi _s} - {h^{ii}}\Gamma _{ii}^k{D_k}{\phi _s} + {h^{ii}}\left( {{\phi _s} \wedge {\phi _i}} \right){\phi _i}\\
 &= \Delta {\phi _s} + 2{h^{ii}}{A_i}{\partial _i}{\phi _s} + {h^{ii}}\left( {{\partial _i}{A_i}} \right){\phi _s} - {h^{ii}}\Gamma _{ii}^k{A_k}{\phi _s} + {h^{ii}}{A_i}{A_i}{\phi _s} + {h^{ii}}\left( {{\phi _s} \wedge {\phi _i}} \right){\phi _i}.
\end{align*}
The tension free identity and commutator identity give
\begin{align*}
{D_s}{\phi _t} &= {D_t}{\phi _s} = {D_t}\left( {{h^{ii}}{D_i}{\phi _j} - {h^{ii}}\Gamma _{ii}^k{\phi _k}} \right) = {h^{ii}}{D_t}{D_i}{\phi _i} - {h^{ii}}\Gamma _{ii}^k{D_t}{\phi _k} \\
&= {h^{ii}}{D_i}{D_i}{\phi _t} - {h^{ii}}\Gamma _{ii}^k{D_t}{\phi _k} + {h^{ii}}\left( {{\partial _t}u \wedge {\partial _i}u} \right){\partial _i}u.
\end{align*}
Therefore the differential filed $\phi_t$ satisfies
\begin{align*}
{\partial _s}{\phi _t} - \Delta {\phi _t} = 2h^{ii}{A_i}{\partial _i}{\phi _t} + h^{ii}{A_i}{A_i}{\phi _t} + h^{ii}{\partial _i}{A_i}{\phi _t} - {h^{ii}}\Gamma _{ii}^k{A_k}{\phi _t} + {h^{ii}}\left( {{\phi _t} \wedge {\phi _i}} \right){\phi _i}.
\end{align*}
Applying $\partial_t$ to (\ref{991}) gives (\ref{9923}).
\end{proof}

\section{Acknowledgments}
We owe our thanks to the anonymous referee for helpful comments which greatly improved this paper.
We thank Prof. Daniel Tataru for helpful comments on our work and especially the necessity of adding Remark 1.1.

The first version of this paper is Chapter 3 of Li's thesis. We divide it into two parts for publication.
Li owes gratitude to Prof. Youde Wang, Hao Yin, Cong Song for guidance on geometric PDEs and geometric analysis.

\end{document}